\documentclass[12pt,reqno,oneside]{amsart}

\usepackage[totalwidth=16.5cm,totalheight=22cm]{geometry}
\usepackage{amsmath}

\usepackage{amsthm}
\usepackage{graphicx}
\usepackage{lscape}
\usepackage[all]{xy}
\usepackage{tikz-cd}
\usepackage{adjustbox}
\usepackage{multicol}
 \usepackage{amssymb}
\usepackage{colortbl}
\usepackage{color}
\usepackage{epsfig}
\usepackage[normalem]{ulem}
\usepackage{mathrsfs}
\usepackage{wrapfig,lipsum}
\usepackage{subcaption}
\usepackage{amssymb}
\usepackage{latexsym}
\usepackage{enumitem}
\usepackage{mathtools}

\makeatletter
\newtheorem*{rep@theorem}{\rep@title}
\newcommand{\newreptheorem}[2]{%
\newenvironment{rep#1}[1]{%
 \def\rep@title{#2 \ref{##1}}%
 \begin{rep@theorem}}%
 {\end{rep@theorem}}}
\makeatother

\makeatletter
\newtheorem*{rep@cor}{\rep@title}
\newcommand{\newrepcor}[2]{%
\newenvironment{rep#1}[1]{%
 \def\rep@title{#2 \ref{##1}}%
 \begin{rep@cor}}%
 {\end{rep@cor}}}
\makeatother

\makeatletter
\newtheorem*{rep@prop}{\rep@title}
\newcommand{\newrepprop}[2]{%
\newenvironment{rep#1}[1]{%
 \def\rep@title{#2 \ref{##1}}%
 \begin{rep@prop}}%
 {\end{rep@prop}}}
\makeatother

\newtheorem{cor}{Corollary}[section]
\newtheorem{corx}{Corollary}

\newtheorem{theorem}[cor]{Theorem}
\newtheorem{thmx}[corx]{Theorem}

\newrepcor{cor}{Corollary}
\newreptheorem{theorem}{Theorem}

\newtheorem*{theorem*}{Theorem}

\newtheorem{fact}{Fact}[]

\newtheorem{lem}[cor]{Lemma}
\newtheorem{prop}[cor]{Proposition}

\theoremstyle{definition}
\newtheorem{definition}[cor]{Definition}

\theoremstyle{remark}
\newtheorem{oss}[cor]{Remark}
\newtheorem{example}[cor]{Example}

\newcommand{\numberset}{\mathbb}
\newcommand{\N}{\numberset{N}}
\newcommand{\Z}{\numberset{Z}}
\newcommand{\R}{\numberset{R}}
\newcommand{\Q}{\numberset{Q}}

\newcommand{\Oo}{\mathcal{O}}

\newcommand{\Bb}{\mathcal{B}}

\newcommand{\Kal}{\ast} %Kaleidoscope
\newcommand{\Iso}{\mathrm{Isom}}

\newlist{steps}{enumerate}{1}
\setlist[steps, 1]{itemsep=8pt,leftmargin=0cm,itemindent=.5cm,labelwidth=\itemindent,labelsep=0cm,align=left,label = \textbf{\emph{Step \arabic*}:\,}}

\begin{document}

\title{The multiple fibration problem for Seifert 3-orbifolds}
%\title{The classification of Seifert spherical 3-orbifolds up to diffeomorphism}

\author[O. Malech]{Oliviero Malech}

\address{O. Malech: Scuola Internazionale Superiori di Studi Avanzati (SISSA), Via Bonomea 265, 34136, Trieste, Italy.} 
\email{omalech@sissa.it}

\author[M. Mecchia]{Mattia Mecchia}
\address{M. Mecchia: Dipartimento Di Matematica e Geoscienze, Universit\`{a} degli Studi di Trieste, Via Valerio 12/1, 34127, Trieste, Italy.} \email{mmecchia@units.it}
\author[A. Seppi]{Andrea Seppi}
\address{A. Seppi: Univ. Grenoble Alpes, CNRS, IF, 38000 Grenoble, France.} \email{andrea.seppi@univ-grenoble-alpes.fr}

%\date{\today}

\begin{abstract}

We conclude the multiple fibration problem for closed orientable Seifert three-orbifolds, namely the determination of all the inequivalent fibrations that such an orbifold may admit. We treat here geometric orbifolds with geometries $\mathbb R^3$ and $\mathbb S^2\times\mathbb R$ and bad orbifolds (hence non-geometric), since the only other geometry for which the multiple fibration phenomenon occurs, namely $\mathbb S^3$, has been treated before by the second and third author. For the geometry $\mathbb R^3$ we recover, by direct and geometric arguments, the computer-assisted results obtained by Conway, Delgado-Friedrichs, Huson and Thurston.
\end{abstract}

\maketitle
\vspace{-1.2cm}
\footnotesize
\tableofcontents
  \normalsize

\section{Introduction} %-----------------------INTRO---------------------------

Smooth orbifolds are topological spaces that are locally homeomorphic to quotients of $\R^n$ by the action of a finite group $G$, and they are therefore a natural generalization of smooth manifolds, allowing the presence of \emph{singular points} corresponding to the fixed points of the action of $G$. Smooth orbifolds can be naturally constructed as quotients of a smooth manifold by a properly discontinuous action. An orbifold is called \emph{good} if it can be obtained as such a global quotient, and \emph{bad} otherwise.

In dimension three, a smooth orbifold is called \emph{geometric} if it is locally modelled on one of the eight Thurston's geometries $\mathbb H^3$, $\R^3$, $\mathbb S^3$, $\mathbb H^2\times\R$, $\mathbb S^2\times \R$, $Nil$, $Sol$ and $\widetilde{SL_2}$. 
Closed geometric orbifolds are good orbifolds. They played a fundamental role, among many things, in the Orbifold Geometrization Theorem proved in \cite{orbifoldtheorem} (see also \cite{libroCHK} and \cite{boileau-maillot-porti}). 
 
\subsection{The multiple fibration problem}

The topology of closed geometric orbifolds, whose geometry is one among $\R^3$, $\mathbb S^3$, $\mathbb H^2\times\R$, $\mathbb S^2\times \R$, $Nil$ or $\widetilde{SL_2}$, is studied very effectively via the notion of \emph{Seifert fibration} for orbifold, a generalization of the classical definition of Seifert fibration for manifolds, where the fibers  are allowed to be either circles or intervals. Fibers homeomorphic to circles are  entirely contained either in the singular locus of the orbifold or in its complement, the regular locus. Those homeomorphic to intervals, instead, are in the regular locus except for the endpoints, which always lie in the singular locus.

\begin{figure}[b]
\centering
\includegraphics[width=.6\textwidth]{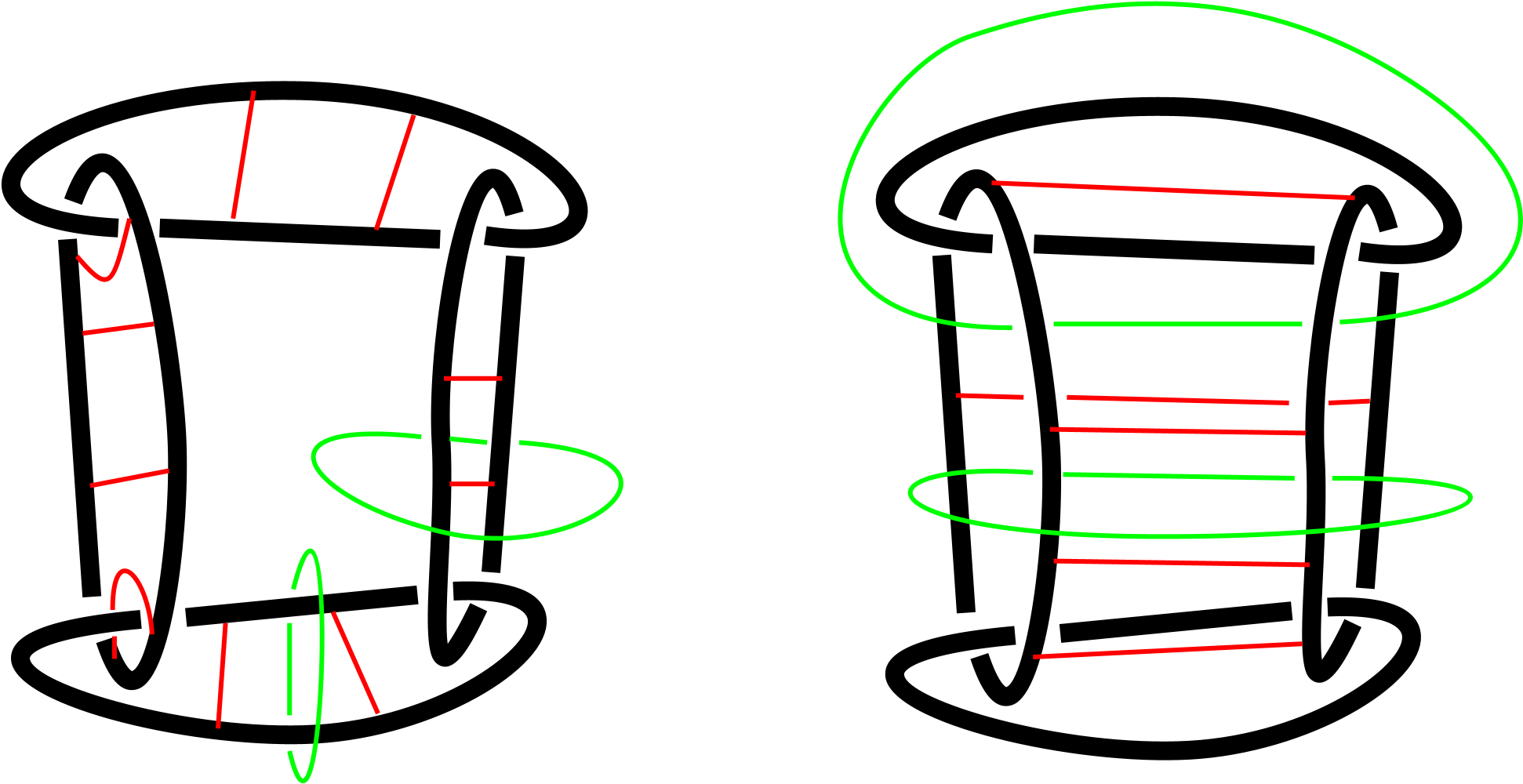} 
\caption{\small A closed flat Seifert orbifold (the underlying space is $S^3$) with two inequivalent fibrations. Green fibres are circles, red fibres are intervals, and the singular locus is in black. The one on the left has base orbifold a disc with four corner points on the boundary, and the local invariant of each corner point is $1/2$. The one on the right has base orbifold a disc with two cone points in the interior, both with local invariant $0/2$. The boundary invariant and Euler number vanish for both fibrations. These fibrations appear in the fourth line of Theorem \ref{FlatMultipleFibrationOrientable}.}\label{fig:example intro}
\end{figure}

Seifert fibered orbifolds are uniquely determined, up to orientation-preserving and fibration-preserving diffeomorphisms, by a collection of \emph{invariants}, which consists of a \emph{base} (a 2-dimensional orbifold $\Bb$, whose underlying topological space is a 2-manifold with boundary) and several rational numbers: the \emph{local invariants} associated to every cone and corner point of $\Bb$, the \emph{boundary invariants} associated to every boundary component of the underlying topological space of $\Bb$, and a global \emph{Euler number}. 

From this perspective, closed geometric orbifolds can be classified via the collection of invariants of their Seifert fibration, with a single major difficulty: the same smooth orbifold may admit several Seifert fibrations which are not equivalent, as in the following definition.

\begin{definition}
Two Seifert fibrations of a 3-orbifold $\Oo$ are \emph{equivalent} if there exists an orientation-preserving diffeomorphism of $\Oo$ mapping one to the other.
\end{definition} 
Equivalently, they have the same collection of invariants. Examples of inequivalent fibrations of the same orbifold are illustrated in Figures \ref{fig:example intro}, \ref{fig:example intro2} and \ref{fig:example intro3}.

%We want to study the Seifert fibered orientable 3-orbifolds that are flat, that is geometric with geometry $\R^3$, or geometric with geometry $S^2\times\R$.
%In particular we want to describe this class of orbifolds up to diffeomorphism from their Seifert data and to do so we have to study and identify the cases with multiple (and inequivalent) Seifert fibrations.

 In this paper we provide a conclusion of what we call the \emph{multiple fibration problem}, namely, determining which closed orientable geometric orbifolds admit several inequivalent Seifert fibrations, and what are those fibrations. Here we will only 
analyse orbifolds with geometry $\R^3$ and $\mathbb S^2\times\R$, and bad orbifolds. Let us explain the reason. A reduction of the problem is provided by the following result proved in  \cite{boileau-maillot-porti}.
\begin{fact}[{\cite[Theorem 2.15]{boileau-maillot-porti}}]\label{fact1}
	Let $\Oo$ be a  compact orientable Seifert fibered good 3-orbifold (possibly with boundary) with infinite fundamental group. If $\Oo$ is not covered by $S^2\times \R$, $T^3$ or $T^2\times I$ then the Seifert fibration on $\Oo$ is unique up to isotopy. %T^2xI caso con bordo?
\end{fact}

A closed orientable good 3-orbifold has finite fundamental group if and only if it is geometric with geometry $\mathbb S^3$. Moreover, by Bieberbach Theorem, closed flat 3-orbifolds are covered by $T^3$. As a consequence, for closed orbifolds, Fact \ref{fact1} above implies the following statement:

\begin{fact}\label{fact2}
	If  a closed orientable Seifert fibered 3-orbifold $\Oo$ admits several inequivalent Seifert fibrations, then it is either geometric with geometry $\mathbb S^3$, $\R^3$ or $\mathbb S^2\times\R$, or bad.
\end{fact}

 For spherical orbifolds the multiple fibration problem has been  solved by the second and third author in \cite{MecchiaSeppi2}, partially relying on previous results achieved in \cite{MecchiaSeppi1,MecchiSeppi3}. This is what motivates the analysis, which is carried out in this work, of the remaining geometries $\R^3$ and $\mathbb S^2\times\R$ (and of the case of bad orbifolds).

%Therefore with this work we want to try to complete the description through the Seifert invariants of compact, orientable, Seifert 3-orbifolds.

\subsection{Geometry $\R^3$}

Closed orbifolds with geometry $\R^3$ (which we will call \emph{flat} in the following) are obtained as the quotients $\R^3/\Gamma$, where $\Gamma$ is a \emph{space group}, that is, a crystallographic group of dimension three. The study of the Seifert fibrations of closed flat three-orbifolds has been tackled in \cite{3Conway} (including the non-orientable case, which is not treated here). 

In particular, for closed flat orbifolds Conway, Delgado-Friedrichs, Huson and Thurston solved the multiple fibration problem, which is called \emph{alias problem} in their work, since a compact notation (a ``name'') is used to denote Seifert fibrations, and a given orbifold may have several ``names''. 
However, in \cite{3Conway} a computer-assisted method is used to solve the problem. That method is based on a consequence of Bieberbach Theorem, namely the fact that two closed flat orbifolds are diffeomorphic if and only if their fundamental groups are isomorphic; hence an algorithm can be used to determine whether two orbifolds, expressed in terms of their Seifert fibrations, have isomorphic fundamental groups. 

The first main theorem that we prove is the following, which recovers the results of \cite{3Conway} by a direct proof, based on geometric and topological arguments. (See Section \ref{sec:intro ideas} below for some ideas of the proof.)

\begin{thmx}\label{FlatMultipleFibrationOrientable}
A closed orientable flat Seifert 3-orbifold has a unique Seifert fibration up to equivalence, with the exceptions contained in the following table:
\begin{table}[h!]
	\begin{center}\begin{tabular}{lll}
			$(S^2(2,2,2,2);0/2,0/2,0/2,0/2;0)$    & $(S^1\times I;;;;0;0;0)$&  \\ \hline
			$(S^2(2,2,2,2);0/2,0/2,1/2,1/2;0)$  & $(S^1\times I;;;;0;1;1)$& $(Mb;;;0;0)$          \\ \hline 
			$(S^2(2,2,2,2);1/2,1/2,1/2,1/2;0)$& $(Kb;;0)$&  \\ \hline
			$(D^2(2,2;);0/2,0/2;;0;0)$ &$(D^2(;2,2,2,2);;1/2,1/2,1/2,1/2;0;0)$  & \\ \hline 
			$(D^2(2,2;);0/2,1/2;;0;1)$ &$(D^2(2;2,2);1/2;1/2,1/2;0;0)$  & \\ \hline 
			$(D^2(2,2;);1/2,1/2;;0;0))$ &$(\R P^2(2,2);0/2,0/2;0)$  & \\ \hline 
			$(D^2(2;2,2);0/2;0/2,0/2;0;0)$&$(D^2(;2,2,2,2);;0/2,0/2,1/2,1/2;0;1)$  & \\ \hline 
		\end{tabular}
	\end{center}
	%\caption{\label{TableMultiFlat} Here $Mb$ is the M\"obius band and $Kb$ is the Klein bottle.}
\end{table}

Two Seifert fibered orbifolds in the table are orientation-preserving diffeomorphic if and only if they appear in the same line. In particular, seven flat Seifert 3-orbifolds admit several inequivalent fibrations;  six of those have exactly two inequivalent fibrations and one has three. 
\end{thmx}

In the table, $Mb$ is the M\"obius band and $Kb$ is the Klein bottle. The multiple fibrations appearing in the first, second and fourth line are pictured in Figures \ref{fig:example intro2}, \ref{fig:example intro3} and \ref{fig:example intro} respectively.

\begin{figure}[htb]
\centering
\includegraphics[width=.56\textwidth]{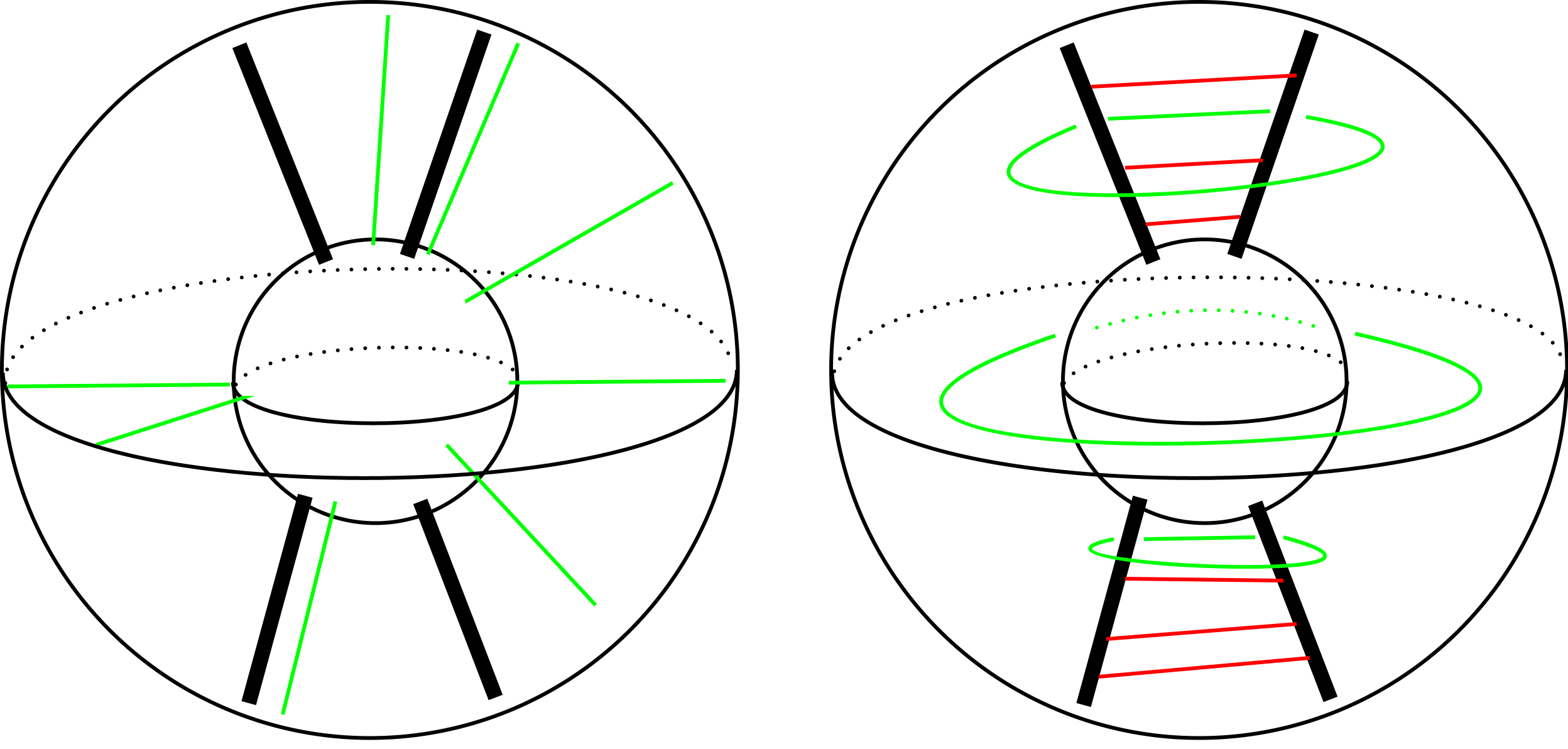} 
\caption{\small The two fibrations of the same orbifold that appear in the first line of the table of Theorem \ref{FlatMultipleFibrationOrientable}. As in Figure \ref{fig:example intro}, green fibres are circles, red fibres are intervals, and the singular locus is in black. The underlying topological space is $S^2\times S^1$, seen by glueing the inner and outer sphere via a homothety. }\label{fig:example intro2}
\end{figure}

\begin{figure}[htb]
\centering
\includegraphics[width=.86\textwidth]{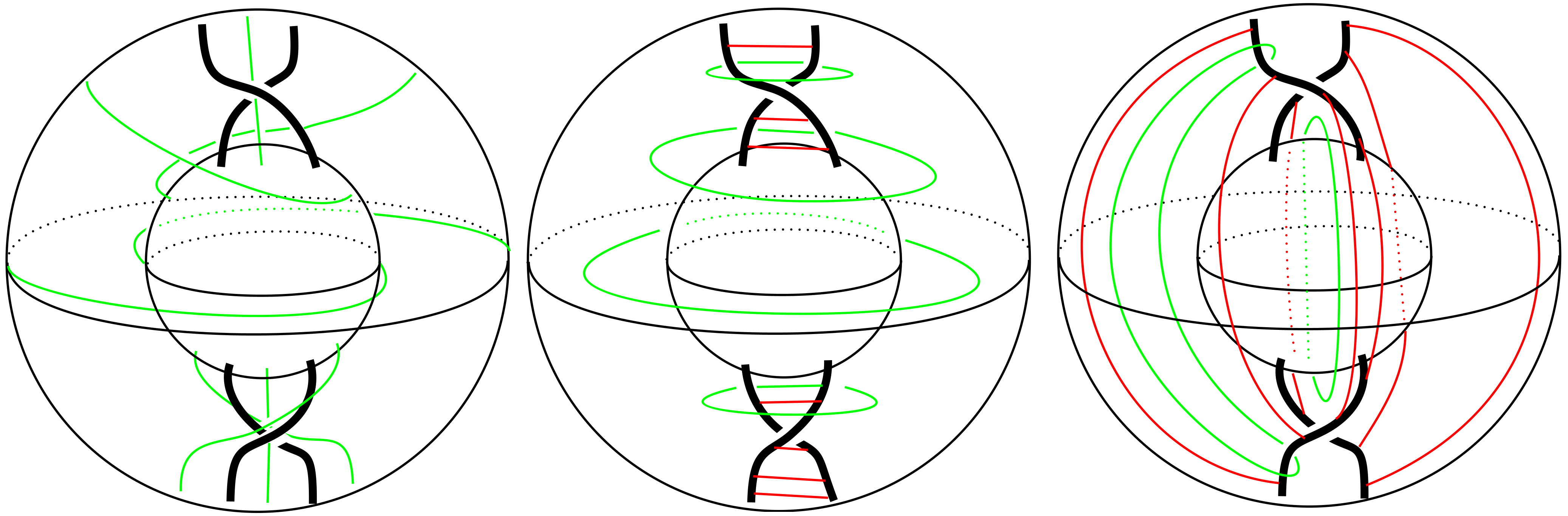} 
\caption{\small The three fibrations of the same orbifold that appear in the second line of the table of Theorem \ref{FlatMultipleFibrationOrientable}. The underlying topological space is again $S^2\times S^1$. }\label{fig:example intro3}
\end{figure}

We have decided to state all the theorems in this introduction by using the classical notation for (orbifold) Seifert fibrations. To explain this notation, let us make an example.

\begin{example}
In the fourth line of the table of Theorem \ref{FlatMultipleFibrationOrientable} (pictured also in Figure \ref{fig:example intro}), $(D^2(2,2;);0/2,0/2;;0;0)$ is the fibration with base orbifold a disc with two cone points in the interior, both with local invariant $0/2$, the boundary invariant and the Euler number both vanish; whereas  $(D^2(;2,2,2,2);;1/2,1/2,1/2,1/2;0;0)$ is the fibration with
base orbifold a disc with four corner points on the boundary, the local invariants of all corner points are $1/2$, and again the boundary invariant and the Euler number vanish. These two fibrations of the same orbifold are shown in Figure \ref{fig:example intro}. 
\end{example}

However, throughout the paper, we will use Conway's notation for Seifert fibrations, which is a little less intuitive, but of very effective use. We explain this notation in Section \ref{sec:conway fibration}, and all the results in the paper will then be expressed in that notation. 

\subsection{Geometry $\mathbb S^2\times\R$}\label{sec:intro SxR}

Let us now move on to the geometry $\mathbb S^2\times\R$. The following theorem provides the classification of multiple fibrations of closed orientable orbifolds with geometry $\mathbb S^2\times\R$:

\begin{thmx}\label{Classification S^2xR}
A closed orientable  Seifert 3-orbifold with geometry $\mathbb S^2\times \R$ has a unique Seifert fibration up to equivalence, with the exceptions contained in the following table:
	\begin{center}
		\begin{tabular}{ll|l}
			&                       & for                                   \\ \hline
			$(S^2(d,d);0/d;0/d;0)$       & $(S^2(n,n);m/n;(n-m)/n;0)$       & $n\geq1$ and $1\leq m\leq n-1$ \\ & & where $d=\mathrm{gcd}(n,m)$ \\ \hline
			$(D^2(;d,d);0/d;0/d;;0;0)$ & $(D^2(;n,n);m/n;(n-m)/n;;1;0)$ & $n\geq1$ and $1\leq m\leq n-1$ \\ & & where $d=\mathrm{gcd}(n,m)$
		\end{tabular}
		\\
	\end{center}
Two Seifert fibered orbifolds in the table are orientation-preserving diffeomorphic if and only if they appear  in the same line, with $d=\mathrm{gcd}(n,m)$.
%{\color{red}{nella colonna di destra, gli invarianti non dovrebbero essere $m/n$ e $n-m/n$?}}
%For any fixed $d\geq 1$, the  fibered orbifolds on the same line of the table are orientation preserving diffeomorphic.
\end{thmx}

Let us explain how to visualize the two inequivalent fibrations of the first line in the following example.

\begin{example}\label{example2intro}
 The orbifold $\Oo$ admitting multiple fibrations in the first line of the table in Theorem \ref{Classification S^2xR}  has underlying space $S^2\times S^1$. Seeing $S^2\times S^1$ as the union of two solid tori $T_1$ and $T_2$ glued by a diffeomorphism of their boundaries sending a meridian of $T_1$ to a meridian of $T_2$, the singular locus consists of the two cores of $T_1$ and $T_2$, both with singularity index $d$. The orbifold $\Oo$ can be obtained as the quotient of $\mathbb S^2\times\R$ by the group $\Gamma$ of isometries generated by a pure translation in the $\R$ direction (with translation length, say, equal to one) and a pure rotation in $\mathbb S^2$ of order $d$. The fibration in vertical lines of $\mathbb S^2\times\R$ then induces the first fibration $(S^2(d,d);0/d;0/d;0)$  of the quotient. To obtain other fibrations, we can consider the group $\Gamma'$ generated by  $\Gamma$ and by an isometry that acts simultaneously on $S^2$ as a rotation of order $n$, and on $\R$ as a translation of length $1/n$. The quotient $(\mathbb S^2\times\R)/\Gamma'$ has again the same diffeomorphism type, and the fibration in vertical lines of $\mathbb S^2\times\R$ now induces the other fibration $(S^2(n,n);m/n;(n-m)/n;0)$.
 
 The orbifolds in the second line have underlying topological space $S^3$, and are obtained as a double quotient of the ones described above.
 \end{example}

\subsection{Bad orbifolds}\label{sec:intro bad}

Finally, the following theorem explains the situation for bad orbifolds.

\begin{thmx}\label{Classification bad}
A closed orientable  Seifert bad 3-orbifold admits infinitely many non-equivalent Seifert fibrations. More precisely, two bad Seifert fibered orbifolds are orientation-preserving diffeomorphic if and only if they appear in the same line of the following table (for $c\neq d$):
	\begin{center}
		\begin{tabular}{ll|l}
			&                       & for                                   \\ \hline
			$(S^2(c,d);0/c;0/d;0)$       & $(S^2(c\nu,d\nu);c\mu/c\nu;d(\nu-\mu)/d\nu;0)$       & $\nu\geq1$ and $1\leq\mu\leq \nu-1$ \\ & & where $\mathrm{gcd}(\mu,\nu)=1$, $c\neq d$ \\ \hline
			$(D^2(;c,d);0/c;0/d;0;0)$       & $(D^2(c\nu,d\nu);c\mu/c\nu;d(\nu-\mu)/d\nu;1;0)$       & $\nu\geq1$ and $1\leq\mu\leq \nu-1$ \\ & & where $\mathrm{gcd}(\mu,\nu)=1$, $c\neq d$ \\ 
		\end{tabular}
		\\
	\end{center}
%{\color{red}{nella colonna di destra, gli invarianti non dovrebbero essere $m/n$ e $n-m/n$?}}
%For any fixed $d\geq 1$, the  fibered orbifolds on the same line of the table are orientation preserving diffeomorphic.
\end{thmx}

We remark that the behavior of the orbifolds that appear in Theorems \ref{Classification S^2xR} and \ref{Classification bad} is very similar. They are treated in a unified way in Lemma \ref{LemmaEquivalence}. We can repeat the considerations of Example \ref{example2intro}; in particular the orbifolds in the first line have underlying space $S^2\times S^1$, and the two core fibers have singularity indices $c$ and $d$. 

\subsection{Outline of the tools}\label{sec:intro ideas}
Let us now pause to briefly discuss some of the ideas in the proofs of Theorems \ref{FlatMultipleFibrationOrientable} and \ref{Classification S^2xR}. 

First, we develop a common theory that we apply for both geometries $\R^3$ and $\mathbb S^2\times\R$. We prove (Proposition \ref{PropLinePresGroupWithVerical}) that a general method to construct Seifert fibration on an orbifold with geometry $\R^3$ and $\mathbb S^2\times\R$ is the following. Consider a discrete subgroup of $\Iso(M)\times\Iso(\R)$, where $M$ is either $\R^2$ or $\mathbb S^2$, with compact quotient $\Oo:=(M\times\R)/\Gamma$. Exactly like in Example \ref{example2intro}, the fibration of $M\times\R$ given by the parallel lines $\{pt\}\times\R$ then induces a Seifert fibration of $\Oo$. By construction, the base 2-orbifold of this fibration of $\Oo$ is a quotient of $M$ (hence it is flat if $M=\R^2$, spherical if $M=\mathbb S^2$), and the Euler number vanishes. These are actually known to be necessary condition: every Seifert fibration of a closed orientable orbifold with geometry $\R^3$ (resp. $\mathbb S^2\times\R$) has flat (resp. spherical) base and vanishing Euler number.

More importantly, we prove that the converse holds true (Corollary \ref{CorLines}): every Seifert fibration of a closed orientable orbifold with vanishing Euler number and flat or spherical base orbifold is equivalent to one obtained by the above construction. 

Nonetheless, the situation for the geometries $\R^3$ and $\mathbb S^2\times\R$ is quite different, and the proofs Theorems \ref{FlatMultipleFibrationOrientable} and \ref{Classification S^2xR}, although partly relying on the general results (Proposition \ref{PropLinePresGroupWithVerical} and Corollary \ref{CorLines}) described above, follow totally independent arguments. On the one hand, by Bieberbach Theorem, two flat orbifolds $\R^3/\Gamma_1$ and $\R^3/\Gamma_2$ are diffeomorphic if and only if the space groups $\Gamma_1$ and $\Gamma_2$ are conjugate by an affine transformation, which in particular sends families of parallel lines to families of parallel lines. Since by Corollary \ref{CorLines} every Seifert fibration of a closed flat orientable orbifold $\Oo$ is induced by a family of parallel lines of $\R^3$, the proof of Theorem \ref{FlatMultipleFibrationOrientable} essentially consists in a careful analysis of the different families of parallel lines of $\R^3$ that a space group $\Gamma<\Iso(\R^3)$ may preserve. 

On the other hand, unlike in $\R^3$, in the geometry $\mathbb S^2\times\R$ there is no notion of ``affine transformation'', and there is a ``privileged'' direction, namely the vertical direction, which is preserved by the isometry group. One might be tempted to conjecture, in analogy with Bieberbach Theorem, that $(\mathbb S^2\times\R)/\Gamma_1$ and $(\mathbb S^2\times\R)/\Gamma_2$ are diffeomorphic (for $\Gamma_1<\Iso(\mathbb S^2)\times\Iso(\R)$) if and only if $\Gamma_1$ and $\Gamma_2$ are conjugate by a transformation acting by isometries on $\mathbb S^2$ and by affine transformations on $\R$. This statement is false: indeed, it would imply the uniqueness of the Seifert fibration for geometry $\mathbb S^2\times\R$. In a certain sense, Theorem \ref{Classification S^2xR} describes the failure of an analogue of Bieberbach Theorem for $\mathbb S^2\times\R$. Its proof shows that the orbifold $\Oo$ in Example \ref{example2intro}, together with a 2-to-1 quotient of $\Oo$ itself, is the only situation where two discrete groups of isometries induce the same diffeomorphism type in the quotient, but the vertical fibration of $\mathbb S^2\times\R$ gives rise to inequivalent fibrations in the quotient.

\subsection{Some consequences}
Finally, let us discuss some consequences of our results.

A particular consequence of our Theorem \ref{FlatMultipleFibrationOrientable}  is that closed orbifolds with geometry $\R^3$ admit at most three inequivalent fibrations. From Theorem \ref{Classification S^2xR} (and Theorem \ref{Classification bad}), the same statement does not hold for 
 geometry $\mathbb S^2\times\R$ (nor for bad orbifolds), since the closed orbifolds for which the fibration is not unique (namely, those in the table of Theorem \ref{Classification S^2xR}), admit infinitely many fibrations. For spherical geometry,  in \cite{MecchiaSeppi2} the second and third author proved that a closed spherical orbifold may admit either infinitely many fibrations, or up to three fibrations. By combining these results, an immediate corollary is the following.

\begin{corx}\label{cor:conigli}
 If a closed Seifert 3-orbifold  does not admit infinitely many inequivalent Seifert fibrations, then it admits at most three inequivalent fibrations.

\end{corx}

Second, we provide a characterization of those closed 3-orbifolds admitting infinitely many inequivalent Seifert fibrations. Before that, we need to introduce some definitions.

A lens space is a 3-manifold obtained by gluing two solid tori along their boundaries by an orientation-reversing diffeomorphism. If we allow  cores of tori to be singular curves,  the gluing gives an orbifold whose underlying topological space 
is a lens space and the singular set is a clopen subset (possibly empty) of the union of the two cores. We call these orbifolds \textit{lens space orbifolds}.  Moreover, we call a \emph{Montesinos graph} a trivalent graph in $S^3$, which consists of a Montesinos link labelled 2, plus possibly one ``strut'' for every rational tangle, namely an interval (with any possible label) whose endpoints lie on the two connected components of the rational tangle. See \cite[Section 4]{Dunbar} for a detailed description.

\begin{corx}\label{cor:char infinitely many}
Let $\Oo$ be a closed Seifert  fibered 3-orbifold. Then $\Oo$ admits  infinitely many inequivalent  fibrations if and only if either it is a lens space orbifold or it has underlying topological $S^3$ and   singular set  a Montesinos graph with at most two rational tangles.
%the orbifold fundamental group of $\Oo$ is either abelian of rank at most two (the group might be finite or not) or a generalized dihedral group such that the normal subgroup of index  2 has at most rank 2.
\end{corx}

\subsection{Organization of the paper}
In Section \ref{sec:geo orb}, we introduce smooth and geometric orbifolds, study their relations with crystallographic groups, and provide a detailed description of orbifolds of dimension two. In Section \ref{sec:3d orb} we study three-dimensional orbifolds, in particular using Seifert fibrations, and we discuss the standard and the Conway notation for Seifert fibrations. In Section \ref{sec:euler zero} we give some general results (in particular, Proposition \ref{PropLinePresGroupWithVerical} and Corollary \ref{CorLines}) on the Seifert fibrations of closed orbifolds with geometries $\R^3$ or $\mathbb S^2\times\R$; equivalently, on the Seifert fibrations with flat or spherical base orbifold and with vanishing Euler number. In Section \ref{sec:flat} we prove Theorem \ref{FlatMultipleFibrationOrientable}. In Section \ref{SxR} we prove Theorems \ref{Classification S^2xR} and \ref{Classification bad} and Corollary \ref{cor:char infinitely many}.

\subsection{Acknowledgments}
The third author has been partially supported by the LabEx PERSYVAL-Lab (ANR-11-LABX-0025-01) funded by the French program Investissements d’avenir. The second and third author are members of the national research group GNSAGA.

\section{Geometric orbifolds}\label{sec:geo orb}

Let us start by introducing some basic notions on smooth and geometric orbifolds in any dimension. Additional details can be found in \cite{boileau-maillot-porti}, \cite{ratcliffe} or \cite{Choi}. 

\subsection{First definitions}

\begin{definition}\label{Defi iniziale}
A \emph{smooth orbifold} $\mathcal{O}$ {\emph{(without boundary)}} of dimension $n$ is a paracomapact Hausdorff topological space $X$ endowed with an atlas $\varphi_i:U_i\to \widetilde U_i/\Gamma_i$, where:
\begin{itemize}
\item The $U_i$'s constitute an open covering of $X$.
\item The $\widetilde U_i$'s are open subsets of $\R^n$, the $\Gamma_i$'s are finite groups, and each $\Gamma_i$ acts smoothly and effectively on  $\widetilde U_i$.
\item The $\varphi_i$'s are homeomorphisms and, for each $i$ and $j$, the composition $\varphi_j\circ\varphi_i^{-1}$ lifts to a diffeomorphism $\widetilde \varphi_{ij}:\widetilde U_i\to\widetilde U_j$. 
\end{itemize}
\end{definition}
The topological space $X$ is called the \emph{underlying topological space} of $\Oo$, and will be denoted by $|\Oo|$. Throughout the paper, all our orbifolds will be connected. We call $\Oo$ a \emph{closed} orbifold if $|\Oo|$ is compact. 

Observe that in general the underlying topological space $|\Oo|$ of a closed orbifold might not be a manifold. In the situations treated in this work,  $|\Oo|$ will always be a manifold, but  often (in particular in dimension two, see Section \ref{subsec:2dorbifolds}) with boundary.

To every point $x\in \Oo$ one can associate a \emph{local group}, which is the smallest group $\Gamma_x$ that gives a local chart $\varphi:U\to\widetilde U/\Gamma_x$ around $x$. 

\begin{oss}\label{oss local group isometric}
It turns out that any smooth action of a finite group on an open subset of $\R^n$ fixing a point $x$ is conjugate via a diffeomorphism, on a small neighbourhood of $x$,  to an action that fixes $0\in\R^n$ and preserves the standard euclidean distance on an open subset of $\R^n$ containing the origin. Hence all local groups $\Gamma_x$ are isomorphic to finite subgroups of $O(n)$.
\end{oss}

If the local group $\Gamma_x$ is the trivial group, then $x$ is called a  \emph{regular point}. Otherwise $x$ is called a \emph{singular point}. The subset of $\Oo$ consisting of regular points  is, by definition, a smooth manifold. In particular, every manifold is an orbifold all of whose points are regular.

\subsubsection*{Oriented orbifolds} Of course, one can put additional structures on orbifolds. For example, the orbifold $\Oo$ is \emph{oriented} if, in Definition \ref{Defi iniziale}, each $\widetilde U_i$ is endowed with an orientation which is preserved by the action of $\Gamma_i$, and the lifts $\widetilde \varphi_{ij}:\widetilde U_i\to\widetilde U_j$ of the compositions $\varphi_j\circ\varphi_i^{-1}$ preserve such orientation. The orbifold $\Oo$ is called \emph{orientable} if such a consistent choice of orientations for the $\widetilde U_i$ exists.

\subsubsection*{Geometric orbifolds} In a similar spirit, let us introduce geometric orbifolds. Let $(M,g)$ be a Riemannian manifold, which we will assume to be simply connected without loss of generality, and let $\Iso(M,g)$ be its group of isometries. An orbifold $\Oo$ is called \emph{geometric} with \emph{geometry} $(M,g)$ if, in Definition \ref{Defi iniziale}, the open subsets $\widetilde U_i\subset\R^n$ are replaced by open subsets of $M$, the elements of the groups $\Gamma_i$ act on $\widetilde U_i$ as the restrictions of elements in $\Iso(M,g)$, and similarly the lifts $\widetilde \varphi_{ij}:\widetilde U_i\to\widetilde U_j$ are the restrictions of isometries in $\Iso(M,g)$.

The fundamental examples of $(M,g)$ that the reader is advised to keep in mind for the present paper are:
\begin{itemize}
\item the Euclidean space $\R^n$ (the orbifolds with geometry $\R^n$ are called \emph{flat})
\item the sphere $\mathbb S^n$ (the orbifolds with geometry $\mathbb S^n$ are called \emph{spherical})
\item products of the above two items, in particular the three-manifold $\mathbb S^2\times\R$. 
\end{itemize}

In dimension three, there are particular cases of the celebrated \emph{eight Thurston's geometries} used in the Geometrization Program, namely: $\mathbb H^3$, $\R^3$, $\mathbb S^3$, $\mathbb H^2\times\R$, $\mathbb S^2\times \R$, $Nil$, $Sol$, $\widetilde{SL_2}$.

\subsubsection*{Diffeomorphisms and isometries} A \emph{diffeomorphism}  between orbifolds $\mathcal O$ and $\mathcal O'$ is a homeomorphism $f:|\Oo|\to |\Oo'|$  such that each composition $\varphi'_{j'}\circ f|_{U_i}\circ \varphi_i^{-1}$, when it is defined,  lifts to a diffeomorphism of $\widetilde U_i$ onto its image in $\widetilde U'_{j'}$. 

When $\mathcal O$ and $\mathcal O'$ are oriented, then $f$ is called \emph{orientation-preserving} if the lifts as above preserve the orientations chosen on $\widetilde U_i$ and $\widetilde U'_{j'}$. 
When $\mathcal O$ and $\mathcal O'$ are geometric (with the same geometry $(M,g)$), $f$ is called \emph{isometry} if the lifts are %isometries with respect to the Riemannian metrics $g_i$ on $\widetilde U_i$ and $g_j'$ on $\widetilde U'_{j'}$.
restrictions of elements in $\Iso(M,g)$.

\subsubsection*{Good orbifolds} Orbifolds naturally arise as the quotients $\mathcal O=M/\Gamma$, for $M$ a manifold and $\Gamma$ a group acting smoothly and properly discontinuously on $M$. In this situation, the local group of a point $[x]\in M/\Gamma$ is precisely the stabiliser of $x$ (which is finite since the action is properly discontinuous). If $M$ is endowed with a Riemannian metric $g$ and $\Gamma$ acts moreover by isometries on $(M,g)$, then the quotient $M/\Gamma$ is geometric with geometry $(M,g)$. Similarly, if $\Gamma$ preserves an orientation on $M$, then $M/\Gamma$ is oriented. 

An orbifold is \emph{good} if it is diffeomorphic to a quotient $M/G$ as above; otherwise it is called \emph{bad}. By a standard argument, one can show that closed good geometric orbifolds with geometry $(M,g)$ are \emph{isometric} to the quotient $M/\Gamma$, where $\Gamma$ is a subgroup of $\Iso(M,g)$ acting properly discontinuously on $M$.

We mention here that there is a notion of \emph{orbifold fundamental group}. We will not introduce the formal definition; for our purpose, it will be sufficient to observe that for a good orbifold $\Oo\cong M/\Gamma$ where $M$ is a simply connected manifold, the orbifold fundamental group $\pi_1(\Oo)$ is isomorphic to $\Gamma$.

\subsection{Flat orbifolds and crystallographic groups}\label{subsec:cry}

As a special case of the previous paragraphs, closed good flat orbifolds are isometric to $\R^n/\Gamma$, where $\Gamma$ is a discrete group acting isometrically on $\R^n$. Hence the following definition will be of fundamental importance:

\begin{definition} A group $\Gamma$ is a \emph{crystallographic group} of dimension $n\in\N$ if it is a discrete subgroup of $\Iso(\R^n)$ and $\R^n/\Gamma$ is compact.
	A \emph{wallpaper group} is a crystallographic group of dimension 2 and a \emph{space group} is a crystallographic group of dimension 3.
\end{definition}

Let us provide some more detailed properties of crystallographic groups. First, it is well-known  that every $g\in \Iso(\R^n)$ is of the form $g(x)=Ax+t$ for all for some $A\in O(n)$ and $t\in\R^n$.
To simplify the notation, we will denote $g$ by the pair $(A,t)$. 
In particular $\Iso(\R^n)\cong O(n)\ltimes T(n)$ where $T(n)\cong\R^n$ is the translation subgroup, and
 the following is a short exact sequence:
	\begin{center}	\begin{tikzcd}
		0\ar[r]&T(n)\ar[r]&\Iso(\R^n)\ar[r,"\rho"]&O(n)\ar[r]&0
	\end{tikzcd}\end{center}
where $\rho(g)=A$. 

\begin{oss}
Given a crystallographic group $\Gamma$, the quotient $\R^n/\Gamma$ is orientable if and only if $\Gamma$ is orientation-preserving, i.e. if and only if  $\rho (\Gamma)\subseteq SO(n)$.
\end{oss}

\begin{definition}\label{defi point translation group}
Let $\Gamma$ be a subgroup of $\Iso(\R^n)$. Then:
\begin{itemize}
\item the \emph{point group} of $\Gamma$ is the image $\rho(\Gamma)\subset O(n)$;
\item the \emph{translation subgroup} of $\Gamma$ is
$T(\Gamma)=\mathrm{Ker}(\rho|_{\Gamma})=\{g\in\Gamma:\rho(g)=id\}$.
\end{itemize}
\end{definition}

From the definitions, we thus have the following short exact sequence:
	\begin{center}	\begin{tikzcd}
		0\ar[r]&T(\Gamma)\ar[r]&\Gamma\ar[r,"\rho|_\Gamma"]&\rho(\Gamma)\ar[r]&0.
\end{tikzcd}\end{center}

The following is a very classical result on crystallographic groups. See for instance \cite{zbMATH06083893}.

\begin{theorem}[Bieberbach]\label{Bieberbach}
	Let $\Gamma$ be a crystallographic group of dimension $n$. Then:
	\begin{itemize}\item the translation subgroup $T(\Gamma)$ is isomorphic to $\Z^n$;
		\item the point group  $\rho(\Gamma)$ is a finite group.
	\end{itemize}
Two crystallographic groups are abstractly isomorphic if and only if they are conjugate by an affine transformation of $\R^n$.
\end{theorem}

\begin{example}\label{ex:dim one}
There are only two crystallographic groups of dimension $1$, up to isomorphism (hence up to affine conjugation, by Theorem \ref{Bieberbach}):
\begin{itemize}
	\item the infinite cyclic group generated by the translation of $1$;
	\item the infinite dihedral group generated by the reflections fixing the points $0$ and $1/2$.
\end{itemize} 
Clearly the former is an index two subgroup of the latter. 
\end{example}

We take advantage of this one-dimensional example to fix some useful notation. Any $f\in\Iso(\R)$ is described by a pair $(A,t)$ where $A\in\{\pm1\}$; we will then use the notation $f=t\pm$.
In this notation the two crystallographic groups of Example \ref{ex:dim one} are $<1+>$ and $<0-,1->$.

Moreover, any isometry of $\R$ induces an isometry on $\R/\Z=\mathbb S^1$, and all the isometries of $\mathbb S^1$ are obtained in this way. Therefore for each element $g$ of $O(2)\cong\Iso(\mathbb S^1)$, we will write $g=t\pm$ by a little abuse of notation, to mean that $g$ is induced by an isometry $f=t\pm$ of $\R$. Hence when writing $g\in O(2)$ as $g=t\pm$, $t$ is only defined modulo $1$, that is in $\R/\Z$. For example, the rotation of angle $2\pi\alpha$ is denoted as $\alpha+$, and the reflection in the line with slope $\pi\beta$ is denoted as $\beta-$. 

\begin{oss}\label{rmk:sum}
The computational advantage of such notation is that the composition of two elements is very easy to express. Indeed we have
$$(\alpha+)\cdot (\beta \pm)=(\alpha+\beta)\pm \qquad (\alpha-)\cdot (\beta \pm)=(\alpha-\beta)\mp~.$$
\end{oss}

%{\color{red}{non sarebbe meglio porre $\mathbb S^1=\R/2\pi\Z$?}}

%\[Iso(\R) =\{a\pm: a\in\R\} \qquad Iso(S^1) =\{a\pm: a\in \R/\Z\}\]

\subsection{Two-dimensional orbifolds}\label{subsec:2dorbifolds}
Let us now consider orbifolds of dimension 2. In this section, we will denote all 2-orbifolds by the symbol $\Bb$ since 2-orbifolds will represent, as explained in Section \ref{subsec:seifert fibr}, the base orbifold of Seifert fibrations. In the following, the symbol $\Oo$ will be thus used mostly for 3-orbifolds. 

\subsubsection*{The local models}
In a 2-orbifold $\Bb$, every point has a neighborhood that is modelled on $D^2/\Gamma_0$ where, using Remark \ref{oss local group isometric}, $\Gamma_0$ can be assumed to be a finite subgroup of $O(2)$. There are then four possibilities (see Figure~\ref{lm2o}):
\begin{itemize}
\item If $\Gamma_0$ is the trivial group, then $x$ is a regular point;
\item If $\Gamma_0$ is a cyclic group of rotations, then $x$ is called  \emph{cone point} and is labelled with the order of $\Gamma_0$;
\item If $\Gamma_0$ is a group of order 2 generated by a reflection, then $x$ is called  \emph{mirror point};
\item If $\Gamma_0$ is a dihedral group, then $x$ is called \emph{corner point} and is labelled with the order of the index two rotation subgroup of $\Gamma_0$.
\end{itemize}
The labels are also called \emph{singularity indices}. Observe that, if $x$ is a mirror point, then, in the singular locus of $\Bb$, $x$ has a neighbourhood that consists entirely of mirror points. We will call a connected subset of the singular locus consisting of mirror points a \emph{mirror reflector}. 

\begin{figure}[htbp]
\centering
\includegraphics[width=.7\textwidth]{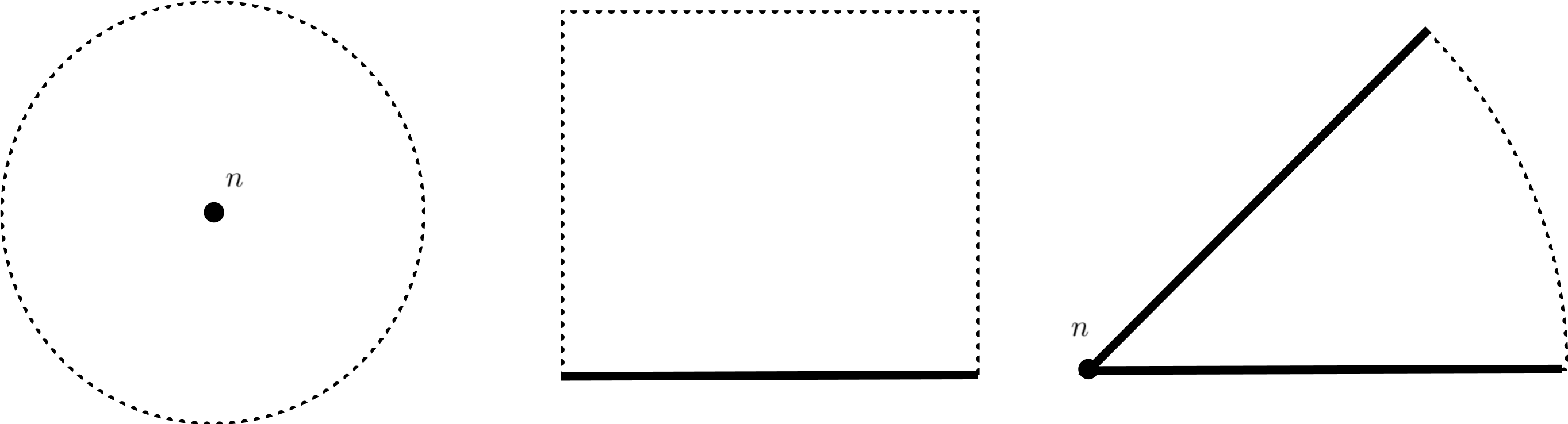} 
\caption{\small Local models of 2-orbifolds. From left to right, a cone point, a mirror reflector and a corner point.}\label{lm2o}
\end{figure}

In particular, the underlying topological space $|\Bb|$ is a manifold with boundary. The points on the boundary of $|\Bb|$ are precisely mirror points and corner points.

\subsubsection*{Standard and Conway notation}

The diffeomorphism type of a closed (connected) 2-orbifold $\Bb$ is uniquely determined by the underlying (compact) manifold with boundary $|\Bb|$, by the number of cone points together with their labels, and by the corner points, together with their labels and their order up to a cyclic permutation, on each boundary component of $|\Bb|$.

In the standard notation, $\Bb$ is therefore denoted by the symbol 
\begin{equation}\label{orb:standard}
\Bb=S(n_1,\dots,n_h;n_{11},\dots,n_{1h_1};\dots;n_{b1},\dots,n_{bh_b})~,
\end{equation} 
where $S$ is the underlying manifold (surface) with boundary, $n_1,\ldots,n_h$ are the labels  of  cone points and $n_{i1},\ldots,n_{ih_i}$ are the labels of the  corner points lying on the $i$-th boundary component of $S$.  %Since in our cases $h\leq 3$, then the order of the list of corner points is irrelevant. The 

It will be extremely useful to adopt the Conway's notation, which is a little less intuitive, but very practical to work with (see \cite{2Conway}). First, let us  introduce the Conway's notation for surfaces (i.e. 2-manifolds). 
The surface 
\begin{equation}\label{sur:standard}
S=S^2\#(\#_tT^2)\#(\#_p\R P^2)-\bigsqcup_b B^2
\end{equation} 
 obtained as the connected sum of $t$ tori and $p$ projective planes with $b$ boundary components will be denoted by:

\begin{equation}\label{sur:conway}
S=\underbrace{\circ\dots\circ}_t\underbrace{\ast\dots\ast}_b \underbrace{\times\dots \times}_p
\end{equation} 

In other words, ``ring" simbols represent connected sums of $T^2$, ``cross'' symbols represent connected sums of $\R P^2$, and ``kaleidoscope" symbols represent boundary components.

%\begin{center}\begin{tabular}{l|lll}
	%	Symbol & $\circ$ & $\ast$        & $\times$   \\ \hline
	%	Name   & Ring    & Kaleidoscope & Cross
	%\end{tabular}
%\end{center}
%Here some examples:
%\begin{center}\begin{tabular}{l|lllllll}
	%	Surface         & $S^2$   & $T^2$    & $D^2$ & $S^1\times I$& $\R P^2$ & Klein bottle & Moebius band \\ \hline
	%	Conway notation & $\cdot$ & $\circ$ & $\textasteriskcentered$    & $\textasteriskcentered \textasteriskcentered$    & $\times$     & $\times\times$  & $\textasteriskcentered \times$  
%	\end{tabular}
%\end{center}

%In the standard notation a 2-orbifold is represented by \[\Oo=S(p_1,\dots,p_n;b_1;\dots; b_m)\] where $S$ is the underlying surface $p_i$ are the singularity indices of the cone points and $b_i$ are  vector of numbers that represent the singularity indices of the corner points in the boundary components of $S$.

When $\Bb$ is a 2-orbifold, we can insert  the singularity indices in Conway's notation \eqref{sur:conway} for the underlying manifold, as follows. The orbifold \eqref{orb:standard}, where $S$ is as in \eqref{sur:standard}, is represented by: 

\begin{equation}\label{orb:conway}
\Bb=\underbrace{\circ\dots\circ}_t n_1,\dots,n_h\ast n_{11},\dots,n_{1h_1}\ast\dots\ast n_{b1},\dots,n_{bh_b}\underbrace{\times\dots \times}_p
\end{equation} 
The reader can compare the two notations through many examples by looking at Table \ref{table:2orbifolds}. 

%We report below some examples:
%\begin{center}\begin{tabular}{l|llll}
%		Classical notation      & $S^2(2,2,2,2)$ & $D^2(;)$ & $D^2(2;2,3)$ & $S^1\times I(;2,3,4;)$ \\ \hline
%		Conway notation & $2222$         & $\ast$    & $2\ast23$       & $\ast234\ast$                    
%	\end{tabular}
%\end{center}

\subsubsection*{Fundamental group}

We summarize here the standard presentation for the fundamental group of closed 2-orbifolds. For the details of the proof, which essentially relies on Van Kampen theorem for orbifolds (\cite[Corollary 2.3]{boileau-maillot-porti}), see \cite[Appendix III]{3Conway} or \cite{2Conway}.

\begin{prop}\label{prop fund gp}
Let $\Bb$ be a closed 2-orbifold as in \eqref{orb:standard} and  \eqref{orb:conway}. Then $\pi_1\Bb$ has a presentation given by the following generators and relations:
\begin{itemize}
\item For each $\circ$ symbol, there are two generators $x_s,y_s$ (for $s=1,\dots,t$);
\item For each $\times$ symbol, there is a generator $z_r$ (for $r=1,\dots,p$);
\item For each cone point of order $n_k$ (for $k=1,\dots,h$), there is a generator $\gamma_k$, satisfying the relation $\gamma_k^{n_k}=1$;
\item For each boundary component of $|\Bb|$, corresponding to a string $\ast n_{i1},\dots,n_{ih_i}$ (for $i=1,\dots,b$), there are $h_i+2$ generators $\delta_i,\rho_{i0},\rho_{i1},\dots,\rho_{ih_i}$, satisfying the relations 
\begin{align*}
&\delta_i^2=\rho_{i0}^2=\dots=\rho_{ih_i}^2=1 \\
&(\rho_{i0}\cdot\rho_{i1})^{n_{i1}}=(\rho_{i1}\cdot\rho_{i2})^{n_{i2}}=\dots=(\rho_{i,h_i-1}\cdot\rho_{ih_i})^{n_{ih_i}}=\delta_i\cdot\rho_{ih_i}\cdot\delta_i^{-1}\cdot\rho_{i0}=1~.
\end{align*}
\end{itemize}
Moreover, there is a global relation:
\begin{equation}\label{eq:global}
\prod_{s=1}^t[x_s,y_s]\cdot\prod_{r=1}^p z_r^2\cdot \prod_{k=1}^h \gamma_k \cdot \prod_{i=1}^b \delta_i=1~.
\end{equation}
\end{prop}

The geometric interpretation of this presentation is the following, as illustrated in Figure \ref{fig:fundgroup}.

\begin{figure}[htbp]
\centering
\includegraphics[width=.75\textwidth]{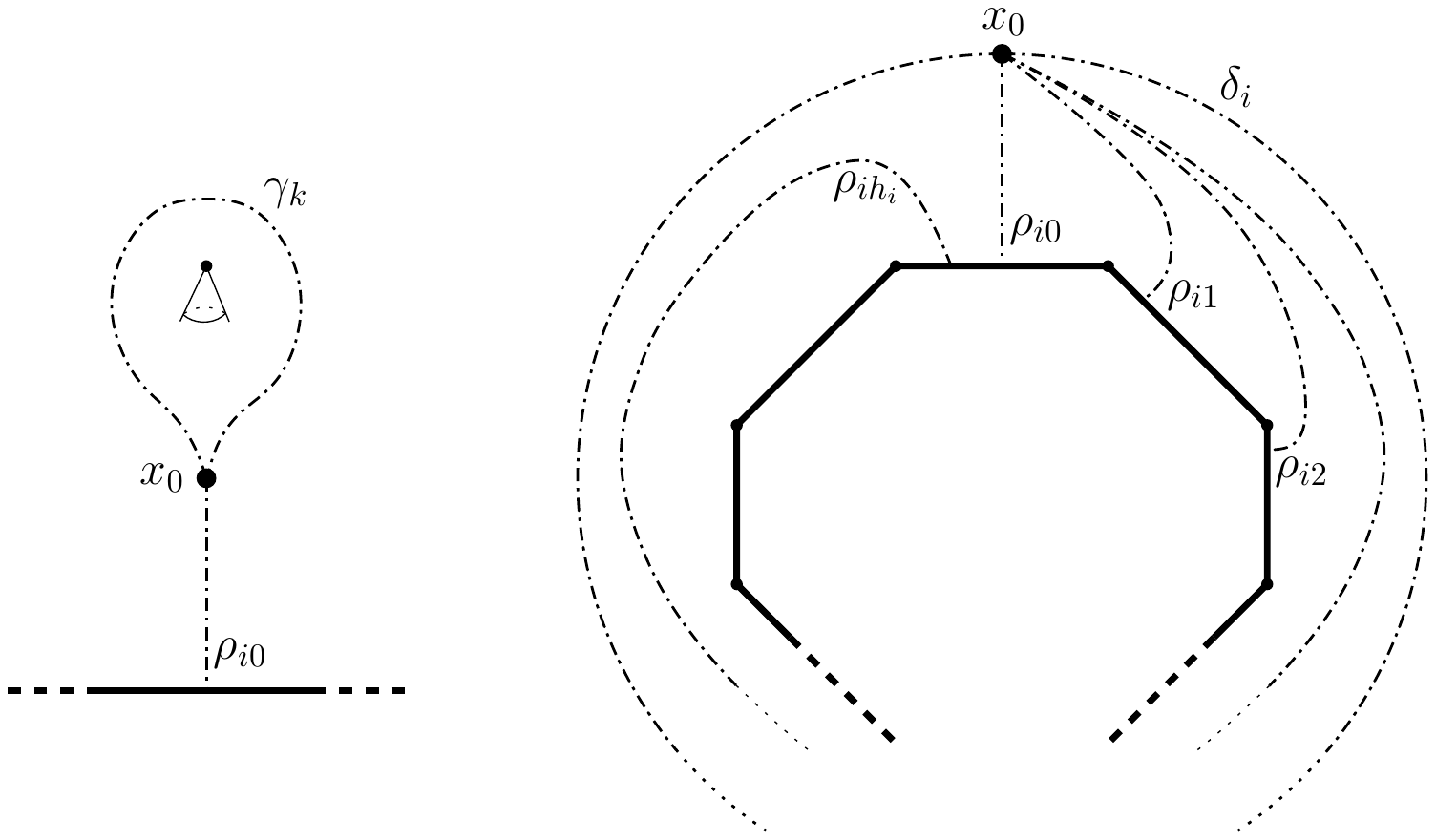} 
\caption{\small On the left,  $\gamma_k$ represents a loop winding around the $k$-th cone point, and, in the case of a mirror reflector with no corner points, $\rho_{i0}$ represents a path reaching the mirror reflector and going back to the basepoint $x_0$. On the right, for the $i$-th boundary component of $|\Bb|$ containing $h_i$ corner points, $\delta_i$ represents a loop around the boundary component, while the $\rho_{ij}$ represent paths reaching the mirror reflectors enclosed between consecutive corner points.}\label{fig:fundgroup}
\end{figure}

For each cone point with label $n_k$, $\gamma_k$ represents a loop that winds around the cone point; the relation $\gamma_k^{n_k}$ corresponds to the fact that the corresponding local group acts as a rotation of order $n_k$. 
For boundary components of the underlying manifold, let us first consider the case where there are no corner points. In this case the above presentation gives two generators (say $\delta$ and $\rho$) with the relation $\rho^2=1$ and $[\delta,\rho]=\delta\cdot\rho\cdot\delta^{-1}\cdot\rho=1$. The subgroup generated by $\delta$ and $\rho$ is the fundamental group of a neighbourhood of the boundary component, where $\delta$ represents a loop around the boundary component, and $\rho$ represents a path that reaches a mirror point and goes back. It has order two since the corresponding action is a reflection. When the boundary component has corner points, the $\rho_{ij}$ represent the paths that reach mirror reflectors enclosed between two corner points; the relations naturally come from the fact that the action inducing a corner point is dihedral.

\subsubsection*{Euler characteristic}

Let us recall the definition of Euler characteristic for orbifolds.

\begin{definition}
Let $\Bb$ be a closed 2-orbifold. The \emph{Euler characteristic} of  $\Bb$ is the rational number:
 $$\chi(\Bb):=\chi(|\Bb|)-\sum_{k=1}^h\left(1-\frac{1}{n_k}\right)-\frac{1}{2}\sum_{i=1}^b\sum_{j=1}^{h_i}\left(1-\frac{1}{n_{ij}}\right)~$$
where $\chi(|\Bb|)$ is the Euler characteristic  of the underlying manifold with boundary; as in \eqref{orb:standard} and \eqref{orb:conway},  the first sum is over all cone points and $n_i$ is the label of the $k$-th cone point; the second sum contains all corner points  and $n_{ij}$ is the label of the $j$-th corner point of the $i$-th boundary component. 
\end{definition}

\subsubsection*{Flat and spherical 2-orbifolds}

Let us briefly provide the classification of flat and spherical closed 2-orbifolds. Closed flat  2-orbifolds are good and are precisely the quotients of $\R^2$ by a wallpaper group, and similarly closed spherical 2-orbifolds are the quotients of $\mathbb S^2$ by a finite subgroup of $\Iso(\mathbb S^2)$.

\begin{oss}
As a consequence of the Gauss-Bonnet formula for orbifolds, the Euler characteristic of a closed flat 2-orbifold vanish. The Euler characteristic of a closed spherical 2-orbifold is positive and is indeed equal to $2\pi$ times its area. One also has the formula $\chi(\mathbb S^2/\Gamma)=\chi(\mathbb S^2)/|\Gamma|=2/|\Gamma|$. 
\end{oss}

The following table contains the list of the closed  flat and spherical 2-orbifolds, both in standard and in Conway notation. In the flat case there are 17 orbifolds in the list; the spherical orbifolds are infinitely many but are collected in families possibly depending on an integer parameter.

\begin{table}[!htb]
    
    \begin{minipage}{.5\linewidth}

      \centering
        \begin{tabular}{|c|c|}
            \multicolumn{2}{c}{Flat} \\ \hline
            Standard       & Conway  \\ \hline
            $S^2(6,3,2)$ & $632$ \\
            $S^2(4,4,2)$ & 442 \\
            $S^2(3,3,3)$ & 333 \\
            $S^2(2,2,2,2)$ & 2222 \\
            $D^2(;6,3,2)$ & $\ast 632$ \\
            $D^2(;4,4,2)$ & $\ast 442$ \\
            $D^2(;3,3,3)$ & $\ast 333$ \\
            $D^2(;2,2,2,2)$ & $\ast 2222$ \\
                        $D^2(4;2)$ & $4\ast 2$ \\
            $D^2(3;3)$ & $3\ast 3$ \\
            $D^2(2,2;2)$ & $22\ast$ \\
            $D^2(2;2,2)$ & $2\ast 22$ \\
            $\R P^2(2,2)$ & $22\times$ \\
            $T^2$ & $\circ$ \\
            $Kb^2$ & $\times\times$ \\
             $S^1\times I$ & $\ast\ast$ \\
            $Mb^2$ & $\ast\times$\\ \hline

        \end{tabular}
    \end{minipage}%
    \begin{minipage}{.5\linewidth}
      \centering

        \begin{tabular}{|c|c|c|}
        \multicolumn{3}{c}{Spherical} \\ \hline
            Standard       & Conway  &   \\ \hline
            $S^2$ & & \\
             $S^2(n,n)$ & $nn$ & $n\geq 2$ \\
              $S^2(2,2,n)$ & $22n$ & $n\geq 2$ \\
              $S^2(3,3,2)$ & $332$ & \\
              $S^2(4,3,2)$ & $432$ & \\
              $S^2(5,3,2)$ & $532$ & \\
              $D^2$ & $\ast$ & \\
              $D^2(n)$ & $\ast n$ & $n\geq 2$ \\
              $D^2(;n,n)$ & $\ast nn$ & $n\geq 2$ \\ 
              $D^2(2;n)$ & $2\ast n$ & $n\geq 2$ \\
              $D^2(;2,2,n)$ & $\ast 22n$ & $n\geq 2$ \\
              $D^2(3;2)$ & $3\ast 2$ & \\ 
              $D^2(;3,3,2)$ & $\ast 332$ & \\
               $D^2(;4,3,2)$ & $\ast 432$ & \\ \
               $D^2(;5,3,2)$ & $\ast 532$ & \\
              $ \R P^2$ & $\times$ & \\
               $\R P^2(n)$ & $n\times$ & $n\geq 2$ \\ \hline
        \end{tabular}
    \end{minipage} 
    	\caption{\small List of all closed orbifolds with geometry $\R^2$ or $\mathbb S^2$. Recall that $Mb^2$ is the M\"obius band and $Kb^2$ is the Klein bottle} \label{table:2orbifolds}
\end{table}

 \begin{oss}\label{rmk reconstruct 1}
It is relatively easy to reconstruct a wallpaper group (up to conjugation by an affine transformation of $\R^2$) from the datum of the quotient flat orbifold. Recall that any isometry of 
$\R^2$ is a rotation, a translation, a reflection or a glide reflection. 
Translation and glide reflection do not have fixed points, hence they do not create singular points in the quotient orbifolds. If a wallpaper group $\Gamma$ instead contains a reflection, then $\R^2/\Gamma$ has a mirror reflector. Finally, if $\Gamma$ contains a rotation, then $\R^2/\Gamma$ has a cone or corner point (that is, the image in the quotient of the fixed point $p$). The point will be a corner point if and only if there is a reflection in $\Gamma$ fixing $p$.

\end{oss}

 \begin{oss}\label{rmk reconstruct 2}
 Similarly, isometries of $\mathbb S^2$ (considered here as the unit sphere in $\R^3$) are rotations, reflections, and compositions of a reflection and a rotation in a line orthogonal to the fixed plane of the reflection. Reflections give rise to mirror reflectors in the quotient, while rotation give rise to cone points or corner points. Again, one can easily reconstruct a finite subgroup of $O(3)$ (up to conjugation in $O(3)$) from the quotient spherical orbifold. 
 \end{oss}

\subsubsection*{Some examples}

Let us briefly see some examples of flat and spherical orbifolds, putting in practice Remarks \ref{rmk reconstruct 1} and \ref{rmk reconstruct 2}, and the computation of fundamental groups.

\begin{example}
The wallpaper group associated to the orbifold $632$ is generated by the rotations around the vertices of a Euclidean triangle with angles $\pi/2,\pi/3,\pi/6$, with rotation angles $\pi,2\pi/3,\pi/3$. Recall that this group is isomorphic to the orbifold  fundamental group $\pi_1(632)$, whose presentation from Proposition \ref{prop fund gp} is indeed:
\begin{align*}
\pi_1(632)=\langle \gamma_1,\gamma_2,\gamma_3\,|\,\gamma_1^2=\gamma_2^3=\gamma_3^6=\gamma_1\gamma_2\gamma_3=1\rangle~.
\end{align*}
 
 Similarly, the orbifold $22n$ is the quotient of $\mathbb S^2$ by a group generated by three rotations having fixed points in the vertices of a spherical triangle of angles $\pi/2,\pi/2$ and $\pi/n$, with rotation angles $\pi,\pi$ and $2\pi/n$, and has a completely analogous presentation.
%\item the wallpaper group associated to the orbifold $2222$ is generated by the rotations of angle $\pi$ around the vertices of a rectangle;
%\item the wallpaper group associated to the orbifold $\ast2222$ is generated by the reflections in the sides of a rectangle.
%	\item the wallpaper group associated to the orbifold $\circ$ is a group $\Z^2$ generated by two linearly independent translations.

	%\item the wallpaper group associated to the orbifold $\ast236$, $\ast 442$ or $\ast333$ is generated by the reflections defined by the sides of a triangle with the correct angles: $\pi/2,\pi/3,\pi/6$ in the first case, $\pi/4,\pi/4,\pi/2$ in the second and $\pi/3,\pi/3,\pi/3$ in the last.
	
	%\item the wallpaper group associated to the orbifold $\ast\ast$ is generated by two parallel lines reflections and a translation in the direction of the lines.
	%\item the wallpaper group associated to the orbifold $\ast\times$ is generated by a reflection and a glide reflection in the same direction of the line fixed by the reflection.
	%\item the wallpaper group associated to the orbifold $\times\times$ is generated by  two parallel glide reflection that have translation parts of the same length.
	%\item the wallpaper group associated to the orbifold $22\ast$ is generated by a reflection and two rotation of $\pi$ with centres on a line parallel to the one fixed by the reflection.
	%\item the wallpaper group associated to the orbifold $2\ast22$ is generated by two parallel reflection and an orthogonal one and a rotation of $\pi$ with centre
\end{example}
 
 \begin{example}
The wallpaper group associated to the orbifold $\ast 632$ is generated by the reflections in the sides of a Euclidean triangle with angles $\pi/2,\pi/3,\pi/6$. From Proposition \ref{prop fund gp}, $\pi_1(\ast 632)$ is generated by $\delta,\rho_0,\rho_1,\rho_2,\rho_3$ with the relations 
$$\delta^2=\rho_0^2=\rho_1^2=\rho_2^2=\rho_3^2=(\rho_0\rho_1)^2=(\rho_1\rho_2)^3=(\rho_2\rho_3)^6=\delta\rho_0\delta^{-1}\rho_3=\delta=1~.$$ After a little simplification we obtain:
\begin{align*}
\pi_1(\ast 632)=\langle \rho_1,\rho_2,\rho_3\,|\,\rho_1^2=\rho_2^2=\rho_3^2=(\rho_3\rho_1)^2=(\rho_1\rho_2)^3=(\rho_2\rho_3)^6=1\rangle~.
\end{align*}
\end{example}

\section{Three-dimensional Seifert orbifolds}\label{sec:3d orb}
Let us now move on to dimension three. In this paper we only consider orientable 3-orbifolds. 

\subsection{Local models}
By Remark \ref{oss local group isometric}, any point admits a local chart of the form $D^3/\Gamma_0$ for $\Gamma_0$ a finite subgroup of $SO(3)$, and the local model is thus the cone over the spherical orientable 2-orbifold $\mathbb S^2/\Gamma_0$. In the right column of Table \ref{table:2orbifolds}, only the first 6 lines include  orientable spherical orbifolds. Hence the local models (apart from the trivial model $D^3$ with no singularity) are as in Figure \ref{lm3o}:

%\begin{figure}[htbp]
%\centering
%\begin{minipage}[c]{.2\textwidth}
%\centering
%\includegraphics[height=3cm]{sphere1-eps-converted-to.pdf} 

%\end{minipage}%
%\begin{minipage}[c]{.2\textwidth}
%\centering
%\includegraphics[height=3cm]{sphere2-eps-converted-to.pdf} 

%\end{minipage}%
%\begin{minipage}[c]{.2\textwidth}
%\centering
%\includegraphics[height=3cm]{sphere3-eps-converted-to.pdf} 

%\end{minipage}%
%\begin{minipage}[c]{.2\textwidth}
%\centering
%\includegraphics[height=3cm]{sphere4-eps-converted-to.pdf} 

%\end{minipage}%
%\begin{minipage}[c]{.2\textwidth}
%\centering
%\includegraphics[height=3cm]{sphere5-eps-converted-to.pdf} 

%\end{minipage}%
%\caption{The local models of orientable 3-orbifolds.}\label{lm3o}
%\end{figure}

\begin{figure}[htbp]
\centering
\includegraphics[width=0.95\textwidth]{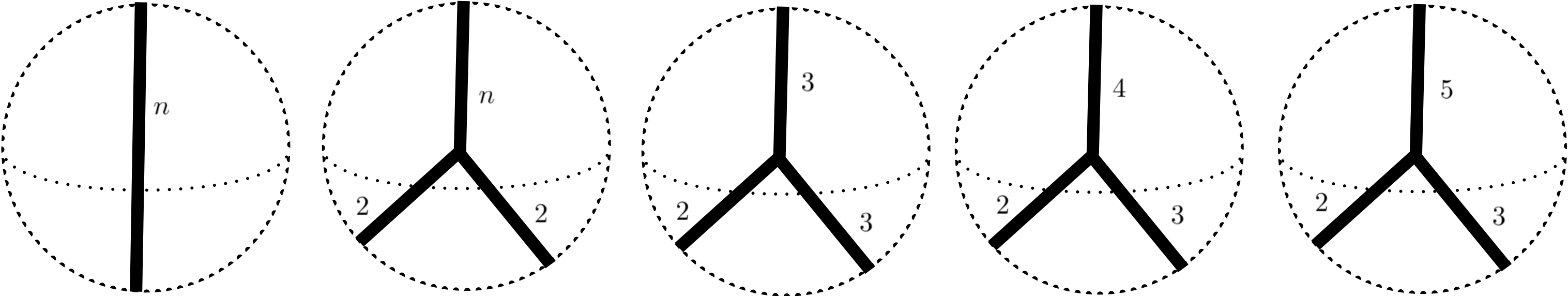} 

\caption{\small The local models of orientable 3-orbifolds.}\label{lm3o}
\end{figure}

As a consequence, the underlying topological space of an orientable 3-orbifold is a manifold and the singular set is a trivalent graph; the local group is cyclic except at the vertices of the graph. The edges of the graph are usually labelled with the order of the local (cyclic) group.

%\begin{figure}[htb]
%\begin{center}
%\includegraphics[height=3.0cm]{cyclic-3orbifold.eps}\hspace{0.3cm}\includegraphics[height=3.0cm]{dihedral-3orbifold.eps}\hspace{0.3cm}\includegraphics[height=3.0cm]{tetrahedral-3orbifold.eps}\hspace{0.3cm}\includegraphics[height=3.0cm]{octahedral-3orbifold.eps}\hspace{0.3cm}\includegraphics[height=3.0cm]{icosahedral-3orbifold.eps}
%\caption{Local models of 3-orbifolds}\label{lm3o}
%\end{center}
%\end{figure}

%------------------------------------SECTION FIBRATIONS-----------------------------

\subsection{Seifert fibrations}\label{subsec:seifert fibr}

We now focus on the topological description of orientable 3-orbifolds via Seifert fibrations. 

%\subsection{Definition of Seifert fibrations for orbifolds}
\begin{definition} \label{defi seifert fibration}
Let $\mathcal O$ be an orientable 3-orbifold. A Seifert fibration for $\mathcal O$ is a surjective map $f:\mathcal O \rightarrow \mathcal B$ with image a 2-orbifold  $\mathcal B$, such that for every point $x\in\mathcal  B$ there exist:
\begin{itemize}
\item an orbifold chart $\varphi:U\cong \widetilde U/\Gamma$ for $\mathcal B$ around $x$;
\item an action of $\Gamma$ on $S^1$;
\item an orbifold diffeomorphism $\phi:f^{-1}(U)\to (\widetilde U\times S^1)/\Gamma$, where the action of $\Gamma$ on $\widetilde U\times S^1$ is diagonal and preserves the orientation;
\end{itemize}
such that the following diagram commutes:
\[
\xymatrix{
f^{-1}(U) \ar[d]_-{f} \ar[r]^-{\phi} & (\widetilde{U}\times S^1)/ \Gamma  \ar[d]^p % & \widetilde{U}\times S^1 \ar[l] \ar[d] 
\\
  U \ar[r]^-{\varphi} & \widetilde{U}/ \Gamma  %&  \widetilde{U}  \ar[l]
  }
~,\]
where the map $p:(\widetilde{U}\times S^1)/ \Gamma\to \widetilde{U}/ \Gamma$ is induced by the projection $pr_1:\widetilde{U}\times S^1\to \widetilde{U}$ onto the first factor. 
\end{definition}

%If we restrict our attention to orientable 3-orbifolds $\mathcal O$, then the action of $\Gamma$ on $\tilde U\times S^1$ needs to be orientation-preserving. 
\begin{oss}
In Definition \ref{defi seifert fibration}, the action of $\Gamma$  on $\widetilde U\times S^1$ is assumed to preserve the orientation.  Hence each element of $\Gamma$ acts either by preserving both the orientation of $\widetilde U$ and that of $S^1$, or by reversing both orientations.
\end{oss}
%Observe that each fiber $\pi^{-1}(x)$ is topologically either a simple closed curve or an interval. 

Fibres that project to a regular point of $\mathcal B$ are called \emph{generic}; the others are called \emph{exceptional}. A \emph{fibration-preserving diffeomorphism} is an orbifold diffeomorphism $\Oo\to\Oo'$, that induces an orbifold diffeomorphisms $\Bb\to\Bb'$ of the base 2-orbifolds via the fibrations $f:\Oo\to\Bb$ and $f':\Oo'\to\Bb'$.

\subsubsection*{Fundamental group}

We collect here a useful description of the orbifold fundamental group of a Seifert 3-orbifold.
\begin{prop} \label{fundamental}
 Let $f:\Oo\to \Bb$ be a Seifert fibration for an orientable 3-orbifold. Then $f$ induces an exact sequence
$$1 \rightarrow C \rightarrow \pi_1(\Oo) \rightarrow \pi_1(\Bb) \rightarrow 1$$
where $C$ is cyclic (either finite or infinite). Moreover, $C$ is finite if and only if $\pi_1(\Oo)$ is finite.
\end{prop}
For the proof see \cite[Proposition 2.12]{boileau-maillot-porti}.

\subsubsection*{Local models and invariants}
Let us discuss the local models for orientable Seifert 3-orbifolds. Relevant references are  \cite{bonahon-siebenmann} or \cite{DunbarTesi}. Clearly if $x\in\Bb$ is a regular point (that is, the fiber $\pi^{-1}(x)$ is generic), then there exists a small neighbourhood $U$ of $x$ such that $\pi^{-1}(U)$ is a trivial circle bundle. Let us therefore analyse the situation for singular points $x\in\Bb$.

%Locally the fibration is given by the curves  induced on the quotient $(\tilde U \times S^1)/\Gamma$ by   the standard fibration of $\tilde U\times S^1$  given by the curves $\{y\}\times S^1$.

If $x$ is a cone point labelled by $n$, the local group $\Gamma$ is a cyclic group of order $n$ acting by rotations both on $\widetilde U$ and on $S^1$. Hence $\pi^{-1}(x)$ has a neighborhood which is a solid torus, fibered in the usual sense of Seifert fibrations for manifolds, except that the central fiber might be contained in the singular locus. If a generator of $\Gamma$ acts on $\widetilde U$ by rotation of angle ${2\pi}/{n}$ and on $S^1$ by rotation of $-{2\pi m}/{n}$, then  the \emph{local invariant} of $\pi^{-1}(x)$ is defined as the ratio $m/n\in\mathbb{Q}/\mathbb{Z}$. One can in fact choose $m$ so that $m/n\in[0,1)$.
%Then we define the local invariant associated to $x$ to be the ratio $p/q\in\mathbb{Q}/\mathbb{Z}$.  We remark that in the literature different sign conventions are used, we use the same as in \cite{BS} while in 
%\cite{Dun1} the invariant is defined to be $-p/q.$ In the orbifold context $p$ and $q$ are not necessarily coprime. In this section  the local invariants $p/q$ have to be considered normalized so that $0\leq p<q$. In the formulae we compute in Sections~\ref{abelian} and \ref{remaining} we give the local invariants in a non-normalized form.   
 In the orbifold context $m$ and $n$ are not necessarily coprime and the fiber $\pi^{-1}(x)$ has singularity  index equal to $\gcd(m,n)$ if $\gcd(m,n)>1$;  otherwise, the fiber is regular. %If $\gcd(p,q)=1$ the fiber is not singular.
{See Figure \ref{fig:seifert case 1}.}

\begin{figure}[htbp]
\centering
\begin{minipage}[c]{.5\textwidth}
\centering
\includegraphics[width=.8\textwidth]{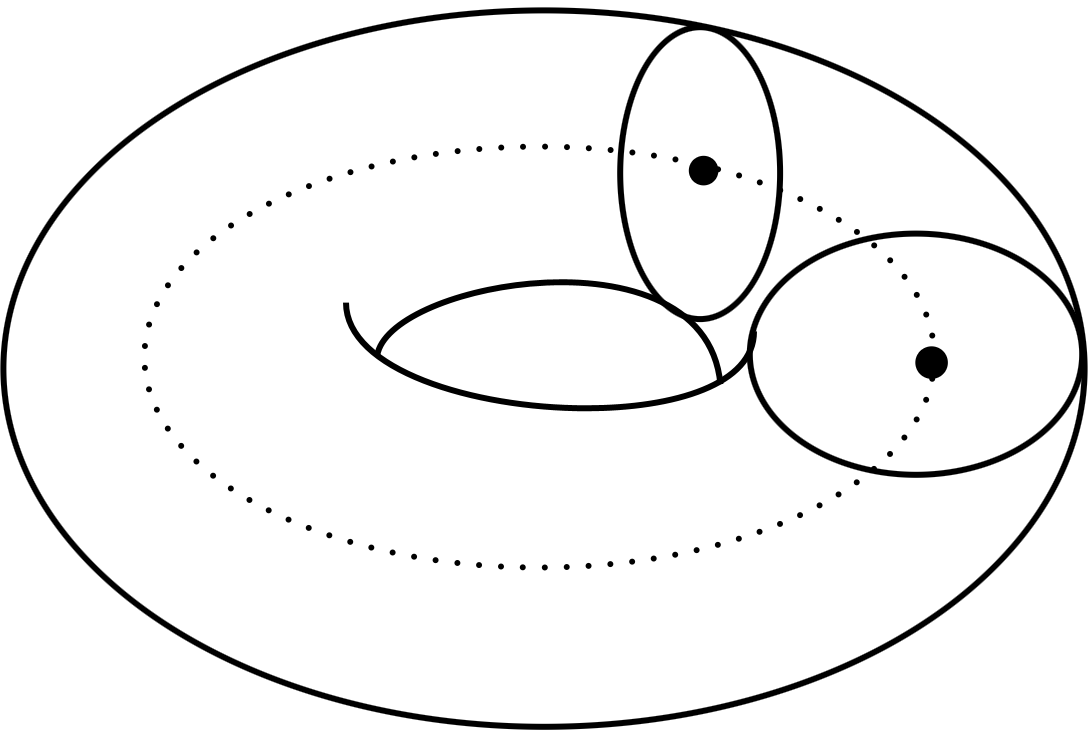} 
%\captionsetup{labelformat=empty}
\end{minipage}%
%\hspace{5mm}
\begin{minipage}[c]{.5\textwidth}
\centering
\includegraphics[width=.8\textwidth]{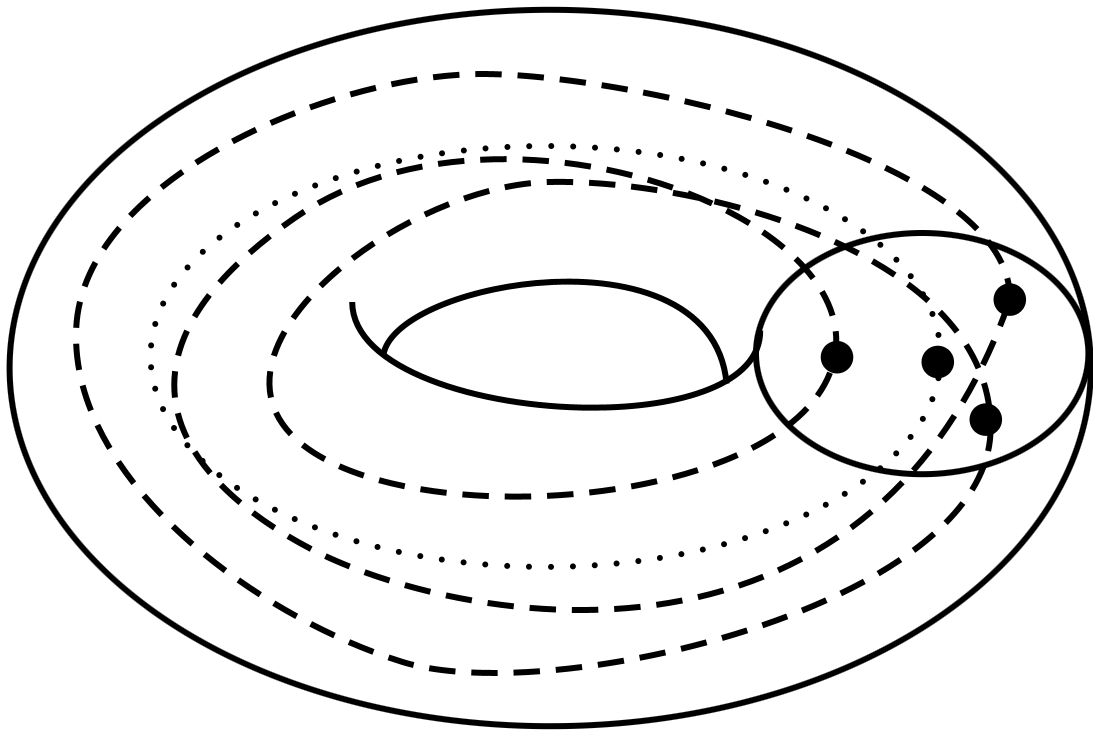} 
%\captionsetup{labelformat=empty}
\end{minipage}%
\caption{\small On the left, a fundamental domain, bounded by two meridional discs, for the  action of $\Gamma$ on the solid torus $\widetilde U\times S^1$ generated by a simultaneous rotations  on $\widetilde U$ and on $S^1$. On the right, the quotient by this action. The central fibre is dotted and has singularity index $p=\mathrm{gcd}(m,n)$ if this number is $>1$. The generic fibre is dashed, and the local invariant is $p/3p$.}\label{fig:seifert case 1}
\end{figure}

If $x$ is a mirror point, the generator of the local group $\mathbb{Z}_2$ acts by reflection both on $\widetilde U$ and on $S^1$, hence by an hyperelliptic involution on the solid torus $\widetilde U\times S^1$. The quotient is topologically a 3-ball, that contains two singular arcs of index 2. The fiber $\pi^{-1}(x)$ is an interval with each endpoint contained in a singular arc. {See Figure \ref{fig:seifert case 2}.}

\begin{figure}[htbp]
\centering
\includegraphics[width=0.25\textwidth]{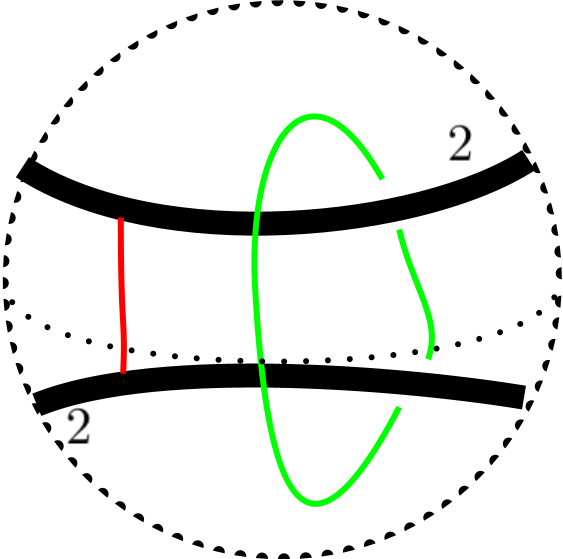} 
\caption{\small The local model over a mirror reflector. The singular locus consists of two parallel arcs with index 2. The generic fibres (in green) wind around the singular locus; the exceptional fibres (in red) are interval with endpoints in the singular locus.}\label{fig:seifert case 2}
\end{figure}

If $x$ is a corner point,  $\Gamma$ is a dihedral group. The index 2 cyclic subgroup acts exactly as for the case of cone point above;   the non-central involutions in $\Gamma$ act by reflection both on $\widetilde U$ and on $S^1$, hence by an hyperelliptic involution as above. The local model is again a topological 3-ball (called solid pillow) with a singular set composed of two arcs, and possibly   a singular interval connecting them. Indeed, the fiber  $\pi^{-1}(x)$ is  an interval as before, and is singular if and only if $\gcd(m,n)>1$, with singularity index $\gcd(m,n)$. The local invariant associated to $\pi^{-1}(x)$ is defined as the local invariant of the cyclic index two subgroup. Observe that the fiber $\pi^{-1}(y)$ for $y\neq x$ is an interval with endpoints in the singular locus if $y$ is a mirror point, or a circle if $y$ is a regular point.  However, these circles are twisted around the singular locus, according to the local invariant over $x$. {See Figure \ref{fig:CornerNonStandard}.}

\begin{figure}[htbp]
\centering

\begin{minipage}[c]{.5\textwidth}
\centering
\includegraphics[width=.5\textwidth]{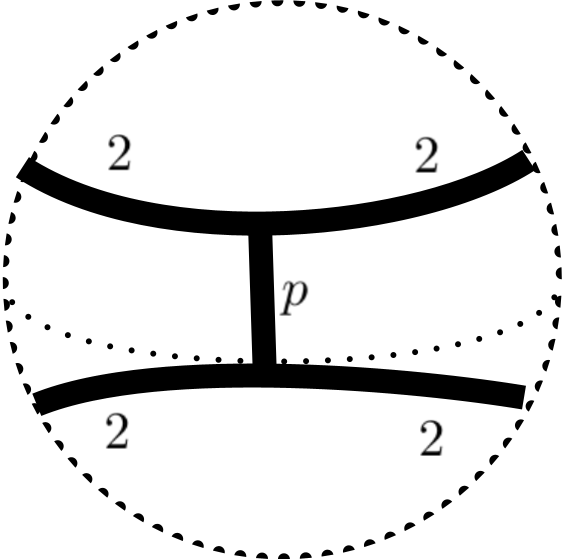} 
%\captionsetup{labelformat=empty}
\end{minipage}%
%\hspace{5mm}
\begin{minipage}[c]{.5\textwidth}
\centering
\includegraphics[width=.5\textwidth]{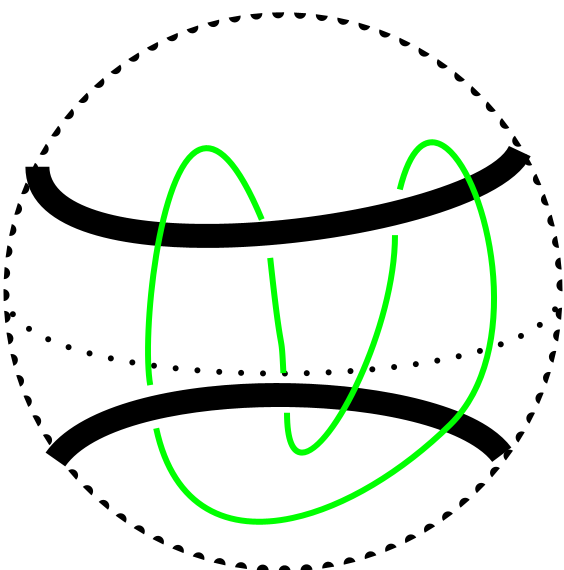} 
%\captionsetup{labelformat=empty}
\end{minipage}%
\caption{\small On the left, $\pi^{-1}(U)$ for $U$ a neighbourhood of a corner point, for $p=\gcd(m,n)>1$. On the right (in the case $p=1$), the generic fibre is shown in green, for  local invariant equal to $1/2$.}\label{fig:CornerNonStandard}
\end{figure}

Nevertheless, there is a canonical diffeomorphism that ``untwists'' the regular fibres, so as to make them wind simply around the singular locus. The drawback of this transformation is that the singular locus become knotted. See Figure \ref{fig:FigureStandardLocal}. We call \emph{standard local model} this version of the preimage $\pi^{-1}(U)$ for $U$ a neighbourhood of $x$.

\begin{figure}[h!]
	\centering
	\begin{subfigure}[b]{0.20\textwidth}
		\centering
		\includegraphics[width=\textwidth]{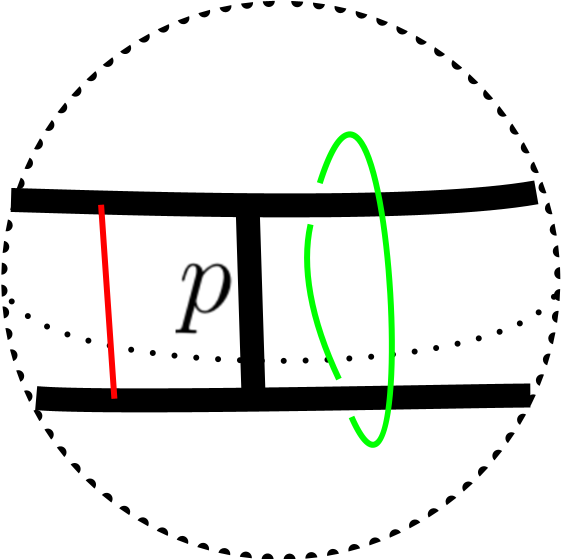}
		\caption{$0/p$}
	\end{subfigure}
	\hfill
	\begin{subfigure}[b]{0.20\textwidth}
		\centering
		\includegraphics[width=\textwidth]{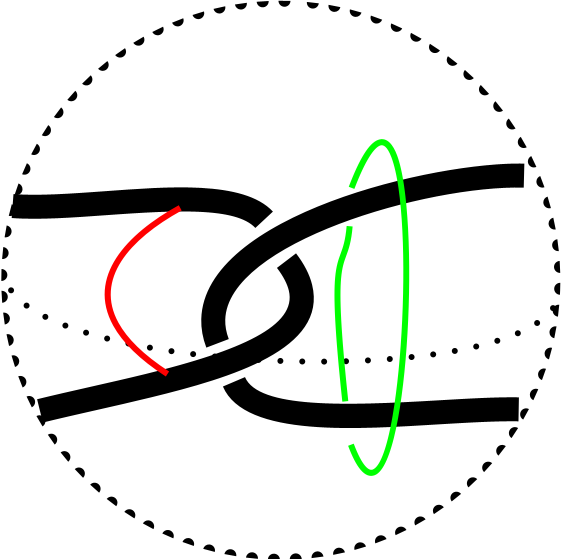}
		\caption{$1/2$}
	\end{subfigure}
	\hfill
	\begin{subfigure}[b]{0.20\textwidth}
		\centering
		\includegraphics[width=\textwidth]{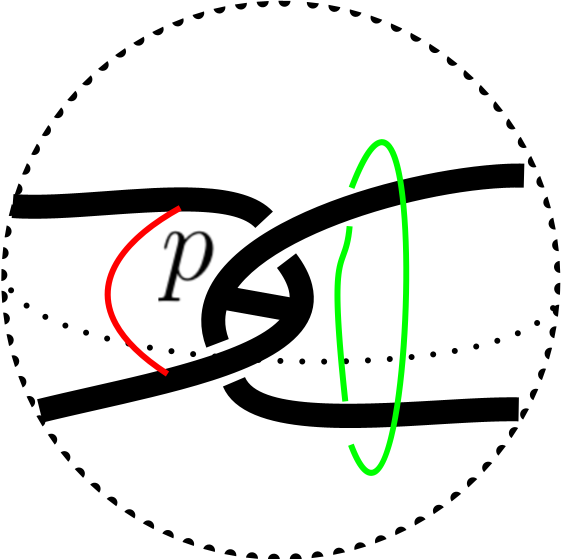}
		\caption{$p/2p$}
	\end{subfigure}
	\hfill
	\begin{subfigure}[b]{0.20\textwidth}
		\centering
		\includegraphics[width=\textwidth]{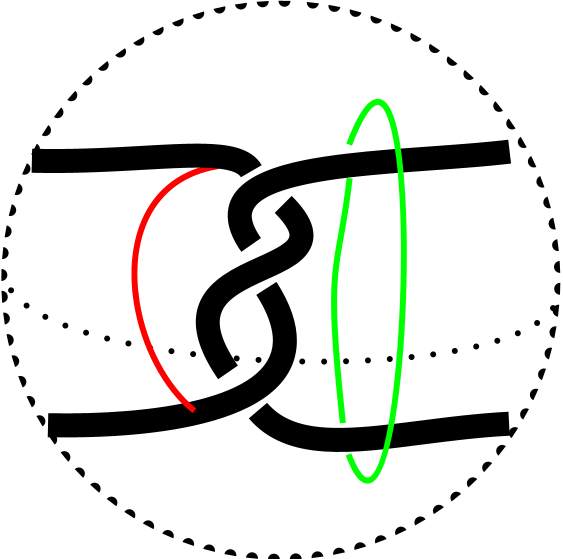}
		\caption{$1/3$}
	\end{subfigure}
	\caption{\small Some examples of standard local models over corner points. The red fibres (intervals) are exceptional; the green fibres (circles) are regular. The fraction under each figure is the local invariant.\label{fig:FigureStandardLocal}}
\end{figure}

%The fibers of $U\times S^1$ intersecting the axes of reflections of $\Gamma$ in $\tilde U$ project to segments that are exceptional fibers of the 3-orbifold; the other fibers of $\tilde U\times S^1$  project to simple closed curves.  

%In Figure~\ref{corner-reflector} the  horizontal segments are not fibers but consist of the endpoints of the fibers that are segments; they are singular (in the sense of orbifold singularities) of index two.

%\begin{figure}[htb]
%\begin{center}
%\includegraphics[height=5cm]{solid-pillows-3.eps}
%\caption{Two copies of a fibered neighborhood of an exceptional fiber of invariant 1/2 corresponding to a corner point}\label{corner-reflector}
%\end{center}
%\end{figure}

\subsubsection*{Boundary invariants and Euler number}

Besides the local invariants described above, there are two additional invariants that must be considered in order to classify Seifert fibered orbifolds. 

An invariant $\xi\in\{0,1\}$ is associated to each boundary component of the underlying manifold $|\Oo|$. It is defined by glueing together the  standard local models over the $h$ corner points of the boundary component. To close up the solid torus, one can perform a number of twists. Up to composing with Dehn twists, the diffeomorphism type then only depends  on the oddity of the number of twists. Hence we define $\xi=0$ if the solid torus is obtained by applying no twist, and $\xi=1$ if one twist is applied. See Figure \ref{fig:BuondaryInvariantTris}.

\begin{figure}[htbp]
\centering
\includegraphics[width=0.9\textwidth]{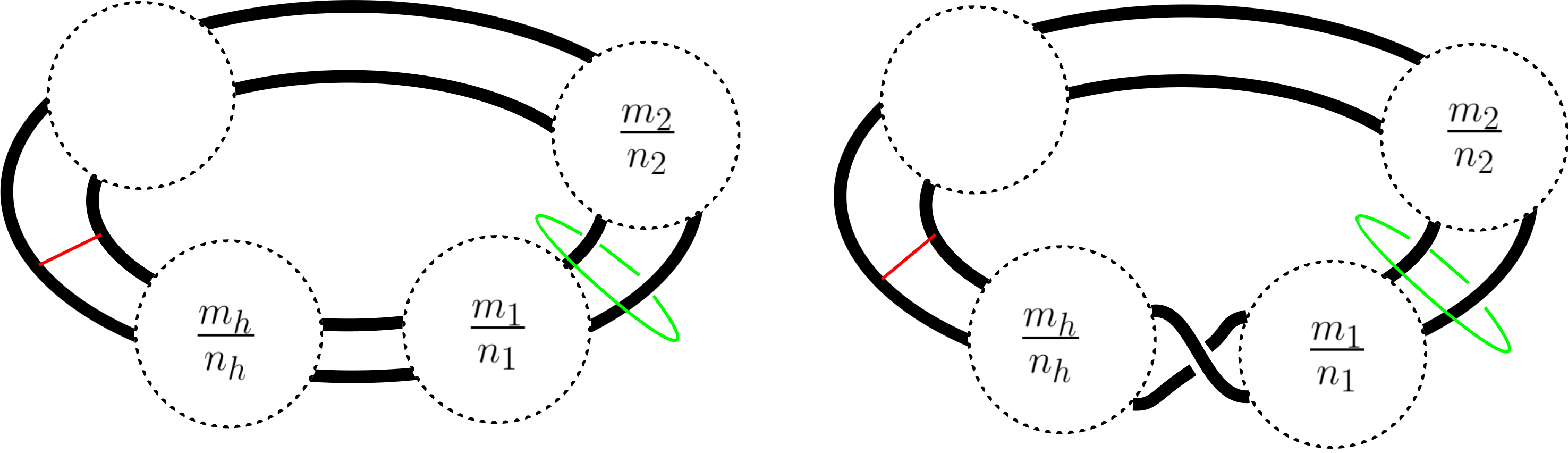} 
\caption{\small The local invariant associated to a boundary component of $|\Bb|$ is $\xi=0$ in the left case and $\xi=1$ in the right case.}\label{fig:BuondaryInvariantTris}
\end{figure}

\begin{oss}\label{rmk:xi one bdy}
In the case of a boundary component only composed of mirror points (no corner points), the preimage of a neighbourhood of the boundary component is a solid torus with singular set a 1-dimensional manifold. If the singular set has two connected components then $\xi=0$; otherwise, if it is connected, then $\xi=1$. See Figure \ref{fig:Boundary_NoCorner}.
\end{oss}

\begin{figure}[htbp]
\centering
\includegraphics[width=0.9\textwidth]{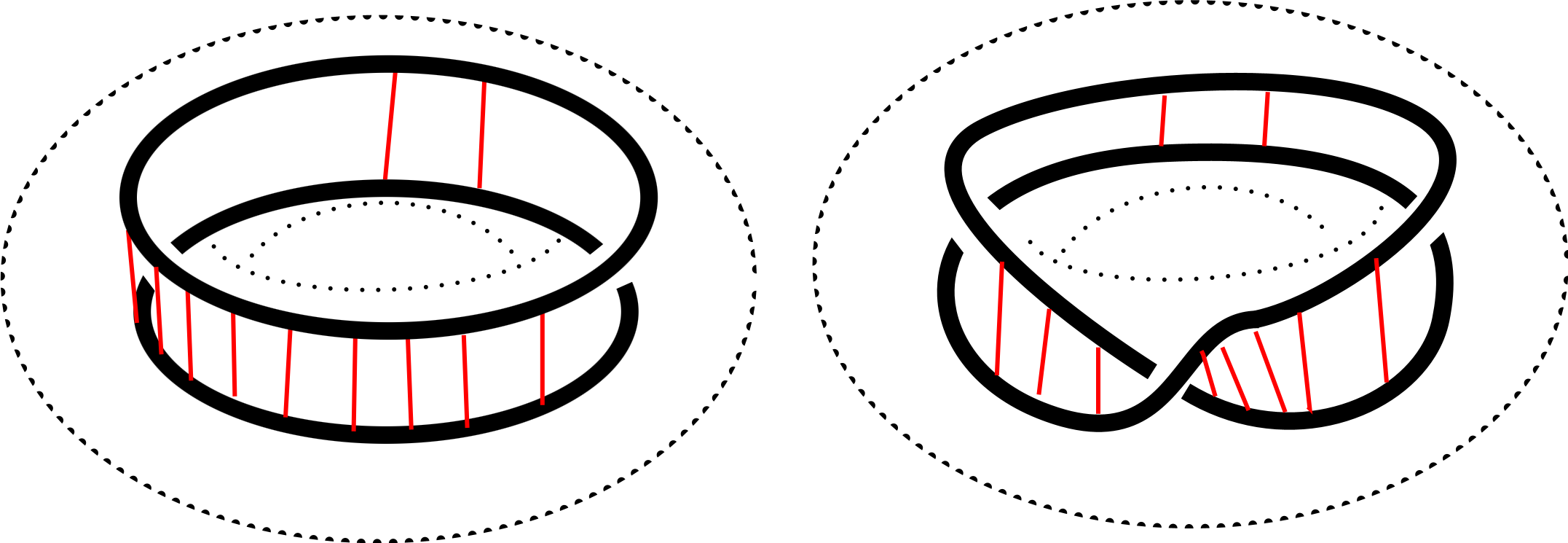} 
\caption{\small The local invariant in the case of a boundary component with no corner points. It turns out that $\xi=0$ if the singular locus has two boundary components, and  $\xi=1$ if it is instead connected.}\label{fig:Boundary_NoCorner}
\end{figure}

Finally, the Euler number $e(f:\Oo\to\Bb)$ (which we will often simply denote $e(f)$) is a rational number associated to the fibration. 

The base orbifold together with the local  invariants, the boundary components invariants and the Euler number  completely classify Seifert fibered  orbifolds in the following sense. See \cite{bonahon-siebenmann} for more details.

\begin{theorem}\label{thm:classification fibered orbifolds}
Let $f:\Oo\to\Bb$ and $f':\Oo'\to\Bb'$ be Seifert fibered 3-orbifolds. There is an orientation-preserving, fibration-preserving orbifold diffeomorphism $\Oo\to\Oo'$ if and only if $e(f)=e(f')$ and there is an orbifold diffeomorphism $\Bb\to\Bb'$ which preserves the invariants associated to every cone point, corner point and boundary component of $\Bb$ and $\Bb'$. 
\end{theorem}

It is also important to remark that the local invariants, boundary components and Euler number are not unconstrained, but must satisfy an identity:

\begin{theorem}
	\label{InvariantRelation}
	Let $f:\Oo\to\Bb$ be a Seifert fibration of a closed orientable 3-orbifold, and let $e(f)$ be its Euler number. If the normalized local invariants of $f$ are ${m_k}/{n_k}$ over cone points, $m_{ij}/n_{ij}$ over corner points on the $i$-th boundary components, and the boundary invariant of the $i$-th boundary component is $\xi_i$, then \[e(f)+\sum_{k=1}^h \frac{m_k}{n_k}+\sum_{i=1}^b\sum_{j=1}^{h_i}\frac{m_{ij}}{2n_{ij}}+\sum_{i=1}^b\frac{\xi_i}{2}=0 \quad in \quad \Q/\Z\]
\end{theorem}
Notice that for this formula to hold, the local invariants $m_{ij}/n_{ij}$ over corner point must be \emph{normalized}, that is,  $0\leq m_{ij}<n_{ij}$.

Finally, the topology of the base orbifold $\Bb$ and the 

\begin{theorem}\label{ThmSeiferTableGeometries} Let $f:\Oo\to B$ be a Seifert fibration of a closed orientable 3-orbifold. Then $\Oo$ is geometric, with geometry one of the eight Thurston's geometries, if and only if $\Bb$ is a good orbifold or the Euler number $e(f)$ does not vanish. 

In this case, the Thurston's geometry of $\Oo$ is unique and is determined by the Euler number and the Euler characteristic of the base space as in the following table.
\begin{center}
	\begin{tabular}{l|ccc}
		& $\chi(\Bb)<0$    & $\chi(\Bb)=0$ & $\chi(\Bb)>0$    \\ \hline
		$e(f)=0$     & $\mathbb H^2\times\R$  & $\R^3$      & $\mathbb S^2\times \R$ \\
		$e(f)\not=0$ & $\widetilde{SL_2}$ & $Nil$       & $\mathbb S^3$         
	\end{tabular}
\end{center}
\end{theorem}

The proof of Theorem \ref{ThmSeiferTableGeometries} is presented in \cite{Dunbar}; another important reference is \cite{Scott}.
	In particular, for the geometries $\R^3$ and $\mathbb S^2\times \R$ of interest in this work, we get the following characterisation:
	
\begin{cor}	\label{CorSeiferTableGeometries} 
Let $f:\Oo\to B$ be a Seifert fibration of a closed orientable 3-orbifold.
		\begin{itemize}
		\item $\Oo$ is flat if and only if $\Bb$ is a good flat orbifold and the Euler number $e(f)$ vanish.
		\item $\Oo$ is geometric with geometry $\mathbb S^2\times \R$ if and only if $\Bb$ is a good spherical orbifold and the Euler number $e(f)$ vanish.
	\end{itemize}
\end{cor}

\subsection{Standard and Conway notation}\label{sec:conway fibration}

We conclude these preliminaries by introducing compact notations to denote 
 Seifert fibered 3-orbifold. In light of Corollary \ref{CorSeiferTableGeometries}, we will only consider Seifert fibered 3-orbifold with vanishing Euler number. Hence we will always omit the Euler number is the notation introduced below; we will implicitly intend that $e(f)=0$. 

The standard notation is the following. Suppose the base orbifold $\Bb$ is as in Equation \eqref{orb:standard}. Then for a Seifert fibration $f:\Oo\to\Bb$ we write:

\begin{equation}\label{fib:standard}
\Oo=\left(\Bb;\frac{m_1}{n_1},\dots,\frac{m_h}{n_h};\frac{m_{11}}{n_{11}},\dots,\frac{m_{1h_1}}{n_{1h_1}};\dots;\frac{m_{b1}}{n_{b1}},\dots,\frac{m_{bh_b}}{n_{bh_b}};\xi_1,\dots,\xi_b\right)
\end{equation}
where all the local invariants are in the same notation as above, see for instance the statement of Theorem \ref{InvariantRelation}. 

In Conway's notation, we simply use the expression \eqref{orb:conway} for the base orbifold, to which we add in subscript the $m_k$ (resp.  $m_{ij}$)
 corresponding to the labels $n_k$ (resp. $n_{ij}$) of cone points (resp. corner points), and we add the $\xi_i$ again in subscript to the symbol $\ast$ representing the boundary component. The fibration \eqref{fib:standard} is thus represented by:

\begin{equation}\label{fib:conway}
\Oo=\circ\dots\circ {n_1}_{m_1},\dots,{n_h}_{m_h}\ast_{\xi_1} {n_{11}}_{m_{11}},\dots,{n_{1h_1}}_{m_{1h_1}}\ast_{\xi_2}\dots\ast_{\xi_b} {n_{b1}}_{m_{b1}},\dots,{n_{bh_b}}_{m_{bh_b}}\bar\times\dots \bar\times
\end{equation}

%\[\Oo=(\circ\dots\circ {p_1}_{p_1-q_1},\dots,{p_n}_{p_n-q_n}\ast_{\xi_1} \bar{b}_1\ast\dots\ast_{\xi_m} \bar{b}_m \bar\times\dots \bar\times)\]
%where for all $j$ we have $\bar {b}_j=({p_{1j}}_{p_{1j}-q_{1j}},\dots,{p_{kj}}_{p_{kj}-q_{kj}})$.  

\begin{oss} 
In \cite{3Conway}, a richer notation is use, that also includes non-orientable orbifolds. We used the symbol $\bar\times$ instead of $\times$ in \eqref{fib:conway} for consistency with \cite{3Conway}, although in the orientable case there is no ambiguity. Also, in \cite{3Conway} Seifert orbifolds whose generic fiber is $S^1$ are enclosed in round brackets, to distinguish them from other orbifolds where the generic fiber is an interval. Since in the orientable case the generic fiber is always $S^1$, we will sometimes omit the brackets.  Finally, we adopt here an opposite sign convention with respect to \cite{3Conway} for  local invariants over cone points and corner points.
\end{oss}

\begin{example}
The following table contains some examples.
\begin{center}\begin{tabular}{|c|c||c|c|} \hline
		$\Bb$ Standard     & $\Bb$ Conway&$\Oo$ Standard     & $\Oo$ Conway   \\ \hline
		$D^2(;{2},{2},{2},{2})$ &  $\Kal2222$    &$(D^2(;{2},{2},{2},{2});;\frac{0}{2},\frac{1}{2},\frac{0}{2},\frac{1}{2};1)$ &  $(\Kal_12_02_12_02_1)$         \\ 
		$S^2({6},{3},{2})$&  $632$&$(S^2({6},{3},{2});\frac{1}{6},\frac{1}{3},\frac{1}{2})$&  $(6_53_22_1)$   \\ 
		$\R P^2({2},{2})$ &  $22\times$&$(\R P^2({2},{2});\frac{0}{2},\frac{0}{2})$ &  $(2_02_0\bar\times)$         \\  \hline
		
\end{tabular}\end{center}
\end{example}

\begin{oss}\label{rmk:bdy invariant determined}
Observe that if  $b=1$ (i.e. $\Bb$ has only one boundary component), the invariant $\xi_1$ can be omitted, since it is determined by the others invariants (see Theorem \ref{InvariantRelation}).
\end{oss}

%\begin{center}\begin{tabular}{l|l}$\Oo$ Standard     & $\Oo$ Conway   \\ \hline$D^2(;\frac{0}{2},\frac{1}{2},\frac{0}{2},\frac{1}{2};0;1)$ &  $(\Kal_12_02_12_02_1)$         \\ $S^2(\frac{1}{6},\frac{1}{3},\frac{1}{2};0)$&  $(6_13_12_1)$   \\ 		$\R P^2(\frac{0}{2},\frac{0}{2};0)$ &  $(2_02_0\bar\times)$         \\  \end{tabular}\end{center}
%-----------------------------------------------------------------------------------

\section{Vanishing Euler number}\label{sec:euler zero}
%{The geometries $\R^3$ and $S^2\times \R$}

In this section we discuss Seifert fibered orbifolds with vanishing Euler number, with particular attention to the geometries $\R^3$ or $\mathbb S^2\times\R$. More concretely, we will prove Theorem \ref{FromSeifertToIsometry} and Corollary \ref{CorLines}, that give use a direct way to construct, given the Seifert invariants of a geometric orientable orbifold $\Oo$ with geometry $\R^3$ (resp. $\mathbb S^2\times\R$), a discrete group of isometries of $\R^3$ (resp. $\mathbb S^2\times\R$) whose quotient is diffeomorphic to $\Oo$.

\subsection{Diagonal actions}

Let us start by showing that Seifert orbifolds with geometry $\R^3$ or $\mathbb S^2\times\R$ and vanishing Euler number can be obtained as quotients by a diagonal action.

\begin{prop}\label{diagonalActionSeifertFibr}
	Let $\Gamma$ be a discrete group of isometries of $M$, where $M \in\{\mathbb S^2,\R^2\}$, such that $M/\Gamma$ is compact. Suppose moreover that $\Gamma$ acts on $\mathbb S^1$ by isometries and that the diagonal action of $\Gamma$ on $M\times \mathbb S^1$ is orientation-preserving. 
	
	Then $(M\times \mathbb S^1)/\Gamma$ is a closed oriented orbifold and it admits a Seifert fibration with base $ M/\Gamma$, which is induced by the fibration of $M\times \mathbb S^1$ with fibers $\{pt\}\times \mathbb S^1$.
	 Furthermore the Euler number of the fibration vanishes.
\end{prop}
\begin{proof}
		Let $\Oo=( M\times \mathbb S^1)/\Gamma$ and $\Bb= M/\Gamma$. We can define in a unique way $f:\Oo\to \Bb$ such that the following diagram is commutative.

	\begin{center}
		\begin{tikzcd}
			 M\times \mathbb S^1 \arrow[r, "\pi_{\Delta \Gamma}"] \arrow[d,"pr_1"]
			&  \Oo \arrow[d, "f"] \\
			M \arrow[r, "\pi_{\Gamma}"]
			& \Bb
		\end{tikzcd}
	\end{center}
	where $\pi_{\Delta \Gamma}$ and $\pi_\Gamma$ simply denote the quotient maps, and we use the notation $\Delta\Gamma$ to remind that the action of $\Gamma$ on  $M\times \mathbb S^1$ is diagonal.
%Notice that $f$ is well defined, since $\Gamma$ preserves the fibration of $M\times \mathbb S^1$ given by $\{pt\}\times \mathbb S^1$.
Since $\pi_\Gamma,\pi_{\Delta \Gamma}$ are quotient maps, $f$ is a continuous function.
%To finish the proof we need to check:
%\begin{itemize}\item $f$ is a Seifert fibration.
%	\item $\Oo$ is compact and connected.
%	\item $e(f)=0$.
%\end{itemize}

\begin{steps}

\item{ \it $f$ is a Seifert fibration.}\\
If $x\in \Bb$ then by definition  there exists a chart $\varphi:U\to\widetilde U/\Gamma_{\tilde x}$, where $\tilde x\in\widetilde U\subset M$,  and we can assume that $\Gamma_{\tilde x}$ is the stabilizer of $\tilde x$ in $\Gamma$. As already observed, the projection $(pr_1)|_{\widetilde U\times\mathbb S^1}:\widetilde U\times\mathbb S^1\to \widetilde U$ induces a map $p:(\widetilde U\times\mathbb S^1)/\Gamma_{\tilde x}\to \widetilde U/\Gamma_{\tilde x}$, and we have the following commutative diagram:

%around $x$ that can be written as $(U,\tilde{U},h\circ\pi_G,h^{-1} Stab_G(\tilde{x}) h)$ where $\tilde{x}\in M$, $\pi_G(\tilde{x})=x$, $h:\tilde{U}\to V\subseteq M$ is a diffeomorphism where $V$ is the connected component of $\pi_G^{-1}(U)$ that contain $\tilde x$ and we have  that \[(\star)\qquad Stab_G(\tilde{x})=\{g\in G: g(V)\cap V\not=\emptyset\}\]  

\[
\xymatrix{
f^{-1}(U) \ar[d]_-{f} \ar[r]^-{} & (\widetilde{U}\times S^1)/ \Gamma_{\tilde x}  \ar[d]^p % & \widetilde{U}\times S^1 \ar[l] \ar[d] 
\\
  U \ar[r]^-{\varphi} & \widetilde{U}/ \Gamma_{\tilde x}  %&  \widetilde{U}  \ar[l]
  }
~,\]
as in Definition \ref{defi seifert fibration}.
%we can call $\psi=(\pi_{\Delta G})\circ (h, id)$ and we can notice that this map is an orbifold submersion since it is the composition of a submersion and a diffeomorphism. 

%To conclude this part of the proof is sufficient to check that $\pi_{\Delta G}|_{V\times S^1}$ factorize through a diffeomorphism between $(V\times S^1)/Stab(\tilde{x})$ and $f^{-1}(U)$ and this follows from the fact that $\pi_{\Delta G}$ can be identified with $\pi_{\Delta Stab(\tilde{x})}$ on ${V\times S^1}$ by propriety $(\star)$. 

\item{\it $\Oo$ is closed.}\\
Clearly $\Oo$ is connected since it is the quotient of $M\times \mathbb S^1$, which is connected. 
Since $M/\Gamma$ is compact, there exists a compact $K$ in $M$ such that $\pi_\Gamma(K)=M/\Gamma$. Therefore $\Oo$ is the image through $\pi_{\Delta  G}$ of $K\times \mathbb S^1$, which is compact,  and therefore $\Oo$ is compact. 

\item $e(f)=0$.\\
%The key idea of this part of the proof is to reduce to a trivial Seifert fibration that finitely \emph{covers} the fibration $f$ and use 
We claim that there exists a Seifert fibered orbifolds $f':\Oo'\to\Bb'$ of vanishing Euler number that finitely covers the fibration $f$, meaning that there is a degree $d$ quotient map $c:\Oo'\to \Oo$ that induces a map $\Bb'\to\Bb$ between the base orbifolds. In this situation, by
\cite[Proposition 2.1]{DunbarTesi}, $e(f')=d\cdot e(f)$, hence we will obtain $e(f)=0$.

If $M=\mathbb S^2$ then the projection on the first factor $pr_1:\mathbb S^2\times \mathbb S^1\to \mathbb S^2$ is clearly a Seifert fibration with vanishing Euler number that covers the fibration $f$.

If $M=\R^2$, by Theorem \ref{Bieberbach} the translation subgroup $T(\Gamma)$ is a normal subgroup of finite index in $\Gamma$ , therefore $\Oo=(\R^2\times \mathbb S^1)/\Gamma$ is the quotient of $(\R^2\times \mathbb S^1)/T(\Gamma)$ by the action of $\Gamma/T(\Gamma)$. Observe moreover that $\R^2/T(\Gamma)\cong T^2$. 
It remains to show that $(\R^2\times \mathbb S^1)/T(\Gamma)$ is diffeomorphic to $T^2\times S^1$ via a diffeomorphism that maps the fibration of $(\R^2\times \mathbb S^1)/T(\Gamma)$, whose fibers are induced by the vertical fibers $\{pt\}\times \mathbb S^1$, to the standard vertical fibration of $T^2\times S^1$.

%In particular we want to show that the diagonal action is conjugated via a diffeomorphism to the translations group $\Z^2$ acting trivially on $S^1$ and the fibration induced by the partition $\{pt\}\times S^1$ is preserved by this diffeomorphism. 

By Theorem \ref{Bieberbach}, up to conjugating $T(\Gamma)$ by an affine transformation (and leaving the action on $\mathbb S^1$ unchanged) we can assume that $T(\Gamma)$ is the standard lattice $\Z^2<\R^2$, Let $a$ and $b$ be its standard generators, namely $a(x,y)=(x+1,y)$ and $b(x,y)=(x,y+1)$.
%, so if $g:\R^2\to\R^2$ is a diffeomorphism that conjugate $T(G)$ and $\Z^2$, then $(g,id_{S^1}):\R^2\times S^1\to \R^2\times S^1$ is a diffeomorphism that conjugate the diagonal action $\Delta T(G)$ with a diagonal action of $\Z^2$.

%We can suppose that the action of $\Z^2$ is generated by the translations $a,b$ where

Recall that the diagonal action of $\Gamma$ is orientation-preserving, and denote  $\alpha+$ and $\beta+$ the rotations associated to the actions of $a$ and $b$ on $\mathbb S^1$.
It is easy to check that the diffeomorphism $h:\R^2\times \mathbb S^1\to\R^2\times \mathbb S^1$ defined by
\[h(x,y,t)=(x,y,t-\alpha x-\beta y)\]
conjugates the diagonal action of $\Z^2$ to the action on $\R^2\times \mathbb S^1$ which is given by the standard action on $\R^2$ and the identity on $\mathbb S^1$. Indeed, 
\begin{align*}
h\circ (a,\alpha+) \circ h^{-1}(x,y,t)&=h\circ (a,\alpha+)(x,y,t+\alpha x+\beta y) \\
&=h(x+1,y,t+\alpha (x+1)+\beta y) \\
&=(x+1,y,t)=(a(x,y),t)=(a,\mathrm{id}_{S^1})(x,y,t)~,
\end{align*}
and similarly $h\circ (b,\beta+) \circ h^{-1}=(b,\mathrm{id}_{S^1})$. This concludes the proof.
\qedhere
\end{steps}
\end{proof}

Second, we want to show that every Seifert fibration with vanishing Euler number and closed base orbifold of the form $M/\Gamma$ with $M=\R^2$ or $\mathbb S^2$ can be obtained by the construction as in Proposition \ref{diagonalActionSeifertFibr}. Recall that, in this situation, the fundamental group of $\Bb$ is isomorphic to $\Gamma$.

\begin{theorem}\label{FromSeifertToIsometry}
	Let $f:\Oo\to \Bb$ a Seifert fibration of a closed orientable 3-orbifold such that $e(f)=0$ and $\Bb$ is a good orbifold with geometry $M\in\{\R^2,\mathbb S^2\}$. Then there exists a representation $\psi:\pi_1(\Bb)\to \Iso(\mathbb S^1)$ such that $(M\times \mathbb S^1)/\pi_1(\Bb)$ (where $\pi_1(\Bb)$ is acting diagonally), endowed with the Seifert fibration induced by the partition in circles $\{pt\}\times \mathbb S^1$ of $M\times \mathbb S^1$, is equivalent to the fibration $f:\Oo\to \Bb$.
\end{theorem}
\begin{proof}
We will divide again the proof in several steps.
\begin{steps}
\item {\it construction of $\psi$.}\\
	Let us define  $\psi:\pi_1(\Bb)\to \Iso(\mathbb S^1)$. We will use the presentation of the fundamental group described in Proposition \ref{prop fund gp}.
	We will define $\psi$ for each generator, and check that the relations are satisfied. We will use the notation introduced at the end of Section \ref{subsec:cry} for elements of $\Iso(\mathbb S^1)$. 
	
	Let us send the generators $x_s$ and $y_s$ associated to a symbol $\circ$ to $0+$ (i.e. the identity), and each generator $z_r$ associated to a symbol $\times$ to $0-$. Each generator $\gamma_k$ corresponding to a cone point with local invariant $m_k/n_k$ is sent to $(-m_k/n_k)+$. The relations $\psi(\gamma_k)^{n_k}=\mathrm{id}$ are clearly satisfied.
	
	Let us now consider boundary components of the underlying manifold, namely symbols $\ast$. We define, using the same notation for the generators as in Proposition \ref{prop fund gp} and for the local invariants as in Section \ref{sec:conway fibration}, $\psi(\rho_{i0})=0-$ and  $\psi(\rho_{i,j-1}\cdot \rho_{ij})=(-m_{ij}/n_{ij})+$ for $j=1,\dots,h_i$.  This defines $\psi(\rho_{ij})$ uniquely for all $i$. Finally, define $\psi(\delta_i)=(-(\sigma_i+\xi_i)/2)+$, where we set 
	$$\sigma_i:=\sum_{j=1}^{h_i}\frac{m_{ij}}{n_{ij}}~.$$
The relations $\psi(\rho_{ij})^2=1$ are clearly satisfied since $\psi(\rho_{ij})$ is a reflection, and clearly also $\psi(\rho_{i,j-1}\cdot \rho_{ij})^{n_{ij}}=1$ for $j=1,\dots,h_i$. To check the relation $\psi(\delta_i)\cdot\psi(\rho_{ih_i})\cdot\psi(\delta_i^{-1})\cdot\psi(\rho_{i0})=1$, 
observe that $\rho_{ih_i}=\rho_{i0}^2\cdot \rho_{i1}^2\cdot\dots\cdot \rho_{i,h_i-1}^2\cdot \rho_{ih_i}=\rho_{i0}\cdot (\rho_{i0}\cdot \rho_{i1})\dots (\rho_{i,h_i-1}\cdot\rho_{i,h_i})$, hence (using Remark \ref{rmk:sum} to compute the compositions):
\[\psi(\rho_{ih_i})=(0-)\cdot\left(-\frac{m_{i1}}{n_{i1}}+\right)\cdot\dots\cdot \left(-\frac{m_{ih_i}}{n_{ih_i}}+\right)=(0-)\cdot(-\sigma_i +)=\sigma_i -~.\]
Then 
\begin{align*}
\psi(\delta_i)\cdot\psi(\rho_{ih_i})\cdot\psi(\delta_i^{-1})\cdot\psi(\rho_{i0})&=
\left(-\frac{\sigma_i+\xi_i}{2}+\right)\cdot(\sigma_i-)\cdot\left(\frac{\sigma_i+\xi_i}{2}+\right)\cdot(0-) \\
&=\left(-\frac{\sigma_i+\xi_i}{2}+\sigma_i-\frac{\sigma_i+\xi_i}{2}\right)+=-\xi_i+=0+
%&=\left(\frac{\Sigma+b}{2}-\Sigma-)\cdot(-\frac{\Sigma+b}{2}-)=b+=0+
\end{align*}
since $\xi_i$ is an integer number.

Finally, it remains to check that the global relation, induced by \eqref{eq:global}, holds. This follows immediately from the observation that $[\psi(x_s),\psi(y_s))]=\psi(z_r)^2=1$ and from Theorem \ref{InvariantRelation}.

By construction, the diagonal action defined by the representation $\psi$ is orientation-preserving. Define $\Oo'=(M\times \mathbb S^1)/\pi_1(\Bb)$, which is a connected orientable orbifold.

\item {\it $\Oo'$ is a closed Seifert fibered orbifold with $\Bb$ as base space and $e=0$.}\label{step:e=0}\\
Using Proposition \ref{diagonalActionSeifertFibr} with $\Gamma=\pi_1(\Bb)$, we obtain that $\Oo'$ has a Seifert fibration over $M/\Gamma\cong \Bb$ and vanishing Euler number. The fibration is induced by the partition in circles of $M\times \mathbb S^1$. Call $f':\Oo'\to B$ this fibration. We will conclude the proof by showing that $\Oo'$ and $ \Oo$ are orientation-preserving and fibration-preserving diffeomorphic. By Theorem \ref{thm:classification fibered orbifolds}, it will be sufficient to check that all Seifert invariants of $\Oo'$ and $ \Oo$ are the same.

\item {\it the case of only one boundary component.} \\
By construction of $\psi$, the local invariants of $f'$ over cone (or corner) points of $\Bb$ are the same to the ones of the fibration $f$, and by Step 2 the Euler number for the fibration $f'$ is $0$. If $\Bb$ has a unique boundary component, the boundary invariant is determined by Theorem \ref{InvariantRelation}, hence in this case all the invariants match automatically.

\item {\it the case of several boundary components.} \\
From Table \ref{table:2orbifolds}, the only situation with several boundary invariants occurs for $\Bb=\Kal\Kal=S^1\times I$.
Its fundamental group has the following group presentation:
\begin{align*}
\pi_1(\Kal\Kal)=&\langle \rho_1,\delta_1,\rho_1,\delta_2\,|\, \rho_1^2=\rho_2^2=[\rho_1,\delta_1]=[\rho_2,\delta_2]=\delta_1\cdot\delta_2=1\rangle \\
\cong&\langle \rho_1,\delta_1,\rho_2\,|\,\rho_1^2=\rho_2^2=[\rho_1,\delta_1]=[\rho_2,\delta_1]=1\rangle~,
\end{align*}
 which is thus isomorphic to $(D_2\Kal D_2)\times \Z$.
It can be realized (see Remark \ref{rmk reconstruct 1}) as the wallpaper group generated by two reflections $x\mapsto Ax$ and $x\mapsto Ax+(1,0)$ in parallel lines, where $A=\mathrm{diag}(-1,1)$, and a translation $x\mapsto x+(0,1)$.
Recalling that $\psi(\rho_1)=0-$, the fixed point set of the action of $(\rho_1,\psi(\rho_1))$ is $L=\{0\}\times \R\times \{(\pm1,0)\}$; its image in the quotient $\Oo'$ is contained in the singular locus and in the preimage of one boundary component of $|\Bb|$. 

If the boundary invariant $\xi_1$ of $\Oo$ equals $0$, then by construction $\psi(\delta_1)=0+$, hence $(\delta_1,\psi(\delta_1))$ preserves each connected component of $L$. This means that the image of $L$ in $\Oo'$ has two components, and thus the invariant $\xi_1'$ is again equal to $0$ by Remark \ref{rmk:xi one bdy}. Similarly, if $\xi_1=1$, then $\psi(\delta_1)=(1/2)+$, hence the action of $(\delta_1,\psi(\delta_1))$ switches the two connected components of $L$. In this case the image of $L$ in $\Oo'$ is connected and thus $\xi_1'=\xi_1=1$. 

We could repeat the same argument for the second boundary invariant $\xi_2$; however this is not necessary since it is determined by all the others by Theorem \ref{InvariantRelation}. This shows that $f:\Oo\to \Bb$ and $f':\Oo'\to \Bb$ have the same Seifert invariants and therefore concludes the proof.\qedhere
\end{steps}
\end{proof}

\begin{oss}
We remark that the first part of the construction in the proof of Theorem \ref{FromSeifertToIsometry} works for $M=\mathbb H^2$ as well. However, the second part of the proof relied strongly on the fact that flat and spherical closed 2-orbifolds have at most one boundary component with a single exception that we controlled in Step 4. For hyperbolic orbifolds, there might be more boundary components, and the situation is therefore more complicated.
\end{oss}

\subsection{Horizontal part and vertical translations}

In Proposition \ref{diagonalActionSeifertFibr}, we considered quotients of $M\times \mathbb S^1$ via diagonal actions of a discrete group of isometries of $M$ with compact quotient, and in Theorem \ref{FromSeifertToIsometry} we showed that all Seifert orbifolds with flat or spherical base and vanishing Euler number are obtained in this way. However, in this work we aim to study quotients of $M\times\R$. We thus need some additional results that ensure that it is not restrictive to only consider quotients of $M\times \mathbb S^1$ as above.

\begin{definition}\label{DefGammaH}
	Let $M\in\{\R^2,\mathbb S^2\}$ and let $\Gamma$ be a subgroup of $\Iso(M)\times \Iso(\R)<\Iso(M\times \R)$.
	The \emph{horizontal part} of  $\Gamma$  is the group
	\[\Gamma_H=\{g\in \Iso(M):(g,h)\in \Gamma\text{ for some } h\in \Iso(\R) \}~.\]
	We say that $\Gamma$ admits a vertical translation if it contains a non-trivial element of the form $(\mathrm{id}_M,c+)$ for $c+\in\Iso(\R)$.
\end{definition}

The following proposition must be compared with Proposition \ref{diagonalActionSeifertFibr}, now for subgroups of $\Iso(M)\times \Iso(\R)$ instead of $\Iso(M)\times \Iso(\mathbb S^1)$. 

\begin{prop}\label{PropLinePresGroupWithVerical}Let $\Gamma$ be a discrete subgroup of $\Iso(M)\times \Iso(\R)$, where $M\in\{S^2,\R^2\}$, such that $(M\times\R)/\Gamma$ is compact.
	Suppose moreover that $\Gamma$ is orientation-preserving and admits a vertical translation. Then $\Gamma_H$ is a discrete group of isometries of $M$, $M/\Gamma_H$ is compact and  $(M\times\R)/\Gamma$ admits a Seifert fibration with base $M/\Gamma_H$ induced by the fibration of $M\times \R$ with fibers $\{pt\}\times \R$.
	Furthermore the Euler number of the fibration vanishes.	
	\end{prop}
\begin{proof}
Let us split the proof into two parts.
\begin{steps}
\item {\it $\Gamma_H$ is a discrete group.}\\
	Let $\{g_n\}_{n\in\N}$ be a sequence in $\Gamma_H$ such that $g_n\to g \in \Iso(M)$. % we want to prove that $g_n$ is constant for $n$ big enough.
	Since $\Gamma$ admits a vertical translation, it contains an element of the form $(\mathrm{id}_M,c+)$, with $c>0$.  Hence there exists a sequence $h_n=(g_n,v_n+)$ in $\Gamma$ with $v_n\in [0,c]$. Up to extracting a subsequence, we can assume $h_n\to (g,v+)$. Since $\Gamma$ is discrete, $h_n$ is eventually constant, and thus so is $g_n$.

\item {\it reducing to a quotient of $M\times\mathbb S^1$}.\\
Let $\Oo=( M\times \R)/\Gamma$ and $\Bb= M/\Gamma_H$. As in Proposition \ref{diagonalActionSeifertFibr}, we can define in a unique way $f:\Oo\to \Bb$ such that the following diagram is commutative.

	\begin{center}
		\begin{tikzcd}
			 M\times \R \arrow[r, "\pi_{\Gamma}"] \arrow[d,"pr_1"]
			&  \Oo \arrow[d, "f"] \\
			M \arrow[r, "\pi_{\Gamma_H}"]
			& \Bb
		\end{tikzcd}
	\end{center}
	Since $\pi_G,\pi_{\Delta G}$ are quotient maps, $f$ is a continuous function, and $\Bb=f(\Oo)$ is compact.

	Let $G=\{(\mathrm{id}_M,h)\in\Gamma\}$. By discreteness, $G$ is generated by a single element, which we can assume (up to applying an affine transformation in the $\R$ factor) to be $(\mathrm{id}_M,1+)$. Observe moreover that $G$ is normal in $\Iso(M)\times \Iso(\R)$. 
	Given $g\in\Gamma_H$, by definition there exists $h\in \Iso(\R)$ such that $(g,h)\in\Gamma$. Then we can define $\psi(g)$ to be the element in $\Iso(\mathbb S^1)$ induced by $h$. It is easily checked that this does not depend on the choice of $h$, and therefore defines a homomorphism $\psi:\Gamma_H\to \Iso(\mathbb S^1)$.
	Since $G$ is normal in $\Gamma$, we can factorize the diagram as follows
	\begin{center}
		\begin{tikzcd}
			M\times \R
			\arrow[drr, bend left, "pr_1"]
			\arrow[ddr, bend right, "\pi_{\Gamma}"]
			\arrow[dr, "{\pi_G}"] & & \\
			& M\times \mathbb S^1 \arrow[r, "\tilde{pr}_{1}"] \arrow[d, "\pi_{\Delta\Gamma_H}"]
			& M\arrow[d, "\pi_{\Gamma_H}"] \\
			& \Oo \arrow[r, "f"]
			& \Bb
		\end{tikzcd}
	\end{center}
	where  $\Gamma_H$ is acting on $\mathbb S^1$ via the homomorphism $\psi$ and and $\Gamma/G$ has been identified with $\Delta \Gamma_H$.
	Therefore we can use Proposition \ref{diagonalActionSeifertFibr} to complete the proof.\qedhere
	\end{steps}
\end{proof}

\begin{oss}
In Proposition \ref{PropLinePresGroupWithVerical}, the hypothesis that  $\Gamma$ admits a vertical translation is essential. Indeed, we claim that if $\Gamma< \Iso(M)\times \Iso(\R)$ is a discrete orientation-preserving  subgroup and the partition in vertical lines induces a Seifert fibration in the quotient, then $\Gamma$ admits a vertical translation.

To see this, by definition of Seifert fibration, we have that the image in the quotient of each vertical line  $\{pt\}\times\R$ is homeomorphic to $S^1$ or $I$. In particular, the quotient map is never injective when restricted to any vertical line. This means that  for every point $p\in M$ there exists $(g,v)\in\Gamma\setminus\{\mathrm{id}\}$ such that $g(p)=p$. Then 
	\[M=\bigcup_{\substack{(g,v)\in\Gamma\setminus\{\mathrm{id}\}}}Fix(g).\]
	%Notice that any isometry of $M$ with fixed point has finite order, so by Lemma \ref{lemmafixed} the set of fix points is a closed set with empty interior. 
Observe that $\Gamma$ is a discrete subset of the second countable space $\Iso(M)\times \Iso(\R)$, hence it is countable. By Baire Category Theorem there exists $(g,v)\in\Gamma\setminus\{\mathrm{id}\}$ such that $Fix(g)$ has nonempty interior in $M$. 
Since isometries are uniquely determined by their differential at a point, $g=\mathrm{id}_M$. Hence $(g=\mathrm{id}_M,v)$ is a vertical translation.
%Each isometry in $G$ that has a fixed point has finite order since $G$ is discrete, therefore by by Newman's theorem  \cite[Corollary 4.1.10]{Choi} we can conclude that $g=id_M$. Since $(g,v)$ is not the identity and it is orientation preserving by hypothesis it is a vertical translation.
\end{oss}

\begin{cor}\label{CorLines}Let $f:\Oo\to \Bb$ a Seifert fibration of a closed orientable 3-orbifold such that $e(f)=0$ and $\Bb$ is a good orbifold with geometry $M\in\{\R^2,\mathbb S^2\}$. Then there exists a discrete subgroup $\Gamma<\Iso(M)\times \Iso(\R)$ that admits a vertical translation such that $(M\times\R)/\Gamma$, endowed with the Seifert fibration induced by the partition in lines $\{pt\}\times \R$ of $M\times \R$, is equivalent to the fibration $f:\Oo\to \Bb$.
In particular, $\Oo$ is geometric with geometry $M\times \R$ and $\Bb$ is diffeomorphic to $M/\Gamma_H$.
\end{cor}
\begin{proof}
	Consider the obvious homomorphism $q:\Iso(\R)\to \Iso(\mathbb S^1)$, and define
		\[\Gamma=\{(g,h)\in \Iso(M)\times \Iso(\R): g\in \pi_1(\Bb)\text{ and }q(h)=\psi(g)\}\]
		where $\psi$ is defined as in Theorem \ref{FromSeifertToIsometry}.
		
		% Define the homomorphism  	$\varphi: Iso(M)\times Iso(\R)\to Iso(M)\times Iso(S^1)\text{ such that }\varphi(g,h)=(g,q(h)).$ and 

	By construction,  $\Gamma$ is the preimage in $\Iso(M)\times \Iso(\R)$ of  $\{(g,\psi(g))\in \pi_1(\Bb)\times \Iso(\mathbb S^1)\}$, it is orientation-preserving and admits the vertical translation $(\mathrm{id}_M,1+)$. It is discrete because $\pi_1(\Bb)$ is. 
	By Proposition \ref{PropLinePresGroupWithVerical}, $(M\times\R)/\Gamma$ admits a Seifert fibration induced by the partition in lines $\{pt\}\times \R$ of $M\times \R$.
	Finally, observe that the kernel $K$ of the obvious homomorphism $\Iso(M)\times\Iso(\R)\to \Iso(M)\times \Iso(\mathbb S^1)$ is generated by $(\mathrm{id}_M,1+)$, is a normal subgroup of $\Gamma$, and $\Gamma/K$ is isomorphic to $\{(g,\psi(g))\in \pi_1(\Bb)\times \Iso(\mathbb S^1)\}$.  Therefore $(M\times \R)/\Gamma$ is   orientation-preserving and fibration-preserving diffeomorphic to $(M\times \mathbb S^1)/\pi_1(\Bb)$, which is equivalent to the fibration $f:\Oo\to \Bb$.
\end{proof}

%Remark that this is a concrete construction and therefore we have a nice description of the group of isometry for a Seifert fibered orbifold with zero Euler number and good base orbifold with non negative Euler number from its Seifert invariant. In particular this construction give us a metric for the orbifold starting from a  metric of the base space, a choice of a representative of $\pi_1(B)$ as isometry group.
%------------------------------------------------------------------------------------------------------------------------------------------------------------------------------------------------------

%-------------------------------------------------------------------------------------------------------------------------------------------------------------------------------------
\section{Flat Seifert 3-Orbifolds}\label{sec:flat}

The main result of this section is Theorem \ref{FlatMultipleFibrationOrientable} of the introduction, which we rewrite here using Conway notation. 

\begin{reptheorem}{FlatMultipleFibrationOrientable}
A closed orientable flat Seifert 3-orbifold has a unique Seifert fibration up to equivalence, with the exceptions contained in the following table:
		\begin{center}\begin{tabular}{lll}
				$(2_02_02_02_0)$    & $(\Kal_0\Kal_0)$&  \\ \hline
				$(2_02_02_12_1)$  & $(\Kal_1\Kal_1)$& $(\Kal_0\bar\times)$          \\ \hline 
				$(2_12_12_12_1)$& $(\bar\times\bar\times)$&  \\ \hline
				$(2_02_0\Kal_0)$ &$(\Kal_02_12_12_12_1)$  & \\ \hline 
				$(2_02_1\Kal_1)$ &$(2_1\Kal_02_12_1)$  & \\ \hline 
				$(2_12_1\Kal_0)$ &$(2_02_0\bar\times)$  & \\ \hline 
				$(2_0\Kal_02_02_0)$&$(\Kal_12_02_02_12_1)$  & \\ \hline 
			\end{tabular}
		\end{center}
Two Seifert fibered orbifolds in the table are orientation-preserving diffeomorphic if and only if they appear in the same line. In particular, seven flat  3-orbifolds admit several inequivalent fibrations;  six of those have exactly two inequivalent fibrations and one has three.
\end{reptheorem}

\subsection{Point groups}

Let $\Gamma$ be a crystallographic group. Recall that its point group $\rho(\Gamma)$ was defined in Definition \ref{defi point translation group}, and it is discrete by Theorem \ref{Bieberbach}.

\begin{definition}
Given a crystallographic group $\Gamma<\Iso(\R^{n+1})$, its \emph{point orbifold} is the spherical orbifold $\mathbb S^n/\rho(\Gamma)$.
\end{definition} 

By a little abuse of notation, if $\Oo\cong \R^{n+1}/\Gamma$ is a good flat orbifold, then we define the point orbifold of $\Oo$ to be the point orbifold of $\Gamma$, which is well-defined up to diffeomorphism. %Notice that  two diffeomorphic flat orbifolds have diffeomorphic point orbifolds.

\begin{definition}
Given a crystallographic group $\Gamma<\Iso(\R^{n+1})$, we say that $\Gamma$ \emph{preserves} a direction $[v]\in\R P^n$ (where $v\in \R^{n+1}\setminus\{0\}$) if every element of $\Gamma$ maps  a line parallel to $\mathrm{Span}(v)$ to a line parallel to $\mathrm{Span}(v)$.
\end{definition}

This is equivalent to the condition that $\Gamma$ preserves the partition of $\R^{n+1}$ into the lines parallel to $\mathrm{Span}(v)$. 

\begin{oss}A direction $[v]$ is preserved by $\Gamma$ if and only if $v$ is an eigenvector of all elements in $\rho(\Gamma)$.
\end{oss}

%  direction $[v]$ induces a factorization of $\R^3$ into $\R^2\times \R$ via an isometry that sends lines with direction $[v]$ to lines of the form $\{pt\}\times\R$. Therefore if $\Gamma$ preserves a direction $[v]$ then $\Gamma$ is conjugate to a subgroup $\Gamma'$ of $Iso(\R^2)\times Iso(\R)$. Therefore we can extend Definition \ref{DefGammaH} to groups with an invariant direction $[v]$ and we define $\Gamma_{[v]}$ as $\Gamma'_H$.

%Notice that such conjugation is not unique, and a choice of such isometry is equivalent to a choice of a basis of the orthogonal plane to $[v]$, therefore $\Gamma_{[v]}$ is well defined up to conjugation by an isometry of $\R^2$. 

Let us now focus on dimension three. If a direction $[v]\in\R P^2$ is preserved by a space group $\Gamma<\Iso(\R^3)$, then $\Gamma$ is contained in the subgroup $\Iso(v^\perp)\times\Iso(\mathrm{Span}(v))<\Iso(\R^3)$. Hence we will extend Definition \ref{DefGammaH} in this context, by defining 
\[\Gamma_H^{[v]}:=\{g\in \Iso(v^\perp):(g,h)\in \Gamma\text{ for some } h\in \Iso(\mathrm{Span}(v)) \}~.\]
Clearly $\Gamma_H^{[v]}$ is identified with a subgroup of $\Iso(\R^2)$, which is well-defined up to conjugacy. 

In this setting, Proposition \ref{PropLinePresGroupWithVerical} tells us that, if $\Gamma$ admits a  translation parallel to $v$, then $\R^3/\Gamma$ has a Seifert fibration induced by the partition of $\R^3$ in lines parallel to $\mathrm{Span}(v)$, with base $v^\perp/\Gamma_H^{[v]}$. The following result tells us that the converse is true, namely all Seifert fibration of closed flat 3-orbifolds is obtained by this construction. 

\begin{prop}\label{PropFlatFibrationAreInducedbyParallel Lines}
	Let $\Gamma$ be a space group and let $f:\R^3/\Gamma\to \Bb$ be a Seifert fibration. Then $f$ is equivalent to a fibration induced by a fibration of $\R^3$ into parallel lines, where the direction $[v]$ of these lines  is preserved by the group $\Gamma$, and $\Gamma$ admits a non-trivial translation parallel to $v$. Furthermore $\Bb$ is diffeomorphic to $v^\perp/\Gamma_H^{[v]}$.
\end{prop}
\begin{proof}
	By Theorem \ref{ThmSeiferTableGeometries}, $\Bb$ is a flat orbifold and the Euler number of $f$ vanishes. By Corollary \ref{CorLines} there exists a discrete subgroup $\Gamma'<\Iso(\R^2)\times \Iso(\R)<\Iso(\R^3)$ admitting a vertical translation such that $\R^3/\Gamma$ is orientation-preserving and fibration-preserving diffeomorphic to $(\R^2\times\R)/\Gamma'$, endowed with the fibration induced by the partition of $\R^3$ by vertical lines. 
	
By Theorem \ref{Bieberbach},  $\Gamma$ and $\Gamma'$ are conjugate by an affine transformation $\varphi:\R^3\to\R^2\times\R$
%, see \cite[Proposition 3.1]{DunbarTesi}. 
Hence $\varphi^{-1}$ sends the partition in vertical lines of $\R^2\times\R$, which is invariant under $\Gamma'$, to a partition in parallel lines of $\R^3$ invariant under $\Gamma$. This concludes the proof.
  \end{proof}
A very similar result in a more general context and with a different approach is obtained in \cite[Theorem 7]{FlatFibration}.

In \cite{Scott} it is showed that if a space group $\Gamma$ admits an invariant direction $[v]$ then it also admits an invariant direction $[v']$, possibly different, such that $\Gamma$ admits a non-trivial translation parallel to $v'$. Therefore if a closed flat orbifold $\R^3/\Gamma$ is not Seifert fibered, then it does not admit any invariant direction. Furthermore $\Gamma$ has an element of order $3$, see \cite[Lemma 4.2]{Scott}.
 Notice also that if $\Gamma$ has an element of order $3$, then this element permutes the vertices of an equilateral triangle and therefore its barycentre is a fixed point.

 An immediate consequence is the following, since for flat manifolds $\Gamma$ acts freely on $\R^3$:
 \begin{cor}
 All closed flat 3-manifolds are Seifert fibered manifolds.
 \end{cor}
\begin{oss}
The analogue statement for flat orbifolds does not hold. Indeed the orbifolds with point orbifold $432$ and $332$ are exactly the one that are not Seifert fibered. A list and a brief description of all flat orbifold can be found in \cite{DunbarTesi}. In Figure \ref{figure8} we have an example of a closed flat 3-orbifold which is not Seifert fibered.  
\end{oss}

 \begin{figure}[htb]
 \centering\includegraphics[width=0.12\textwidth]{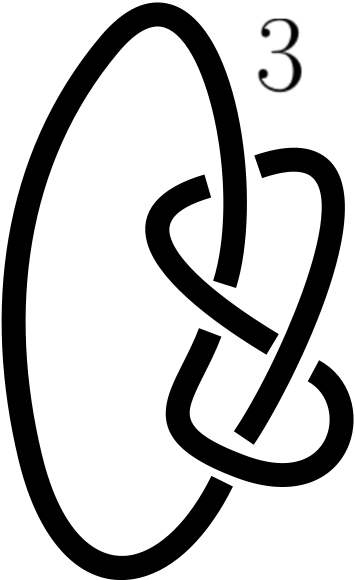}
 \caption{\small An example of closed flat 3-orbifold which does not admit Seifert fibrations. Its underlying topological space is $S^3$, its singular locus is the figure eight knot with singulary index $3$, its point orbifold is $332$ and the associated space group is $P2_13$ (in Hermann–Mauguin notation).}\label{figure8}
 \end{figure}

On the road towards the proof of Theorem \ref{FlatMultipleFibrationOrientable}, we will classify the point groups of space groups that admit several invariant directions. We first give here a restricted list of possible candidates.

\begin{prop}\label{Invariant directions}\label{ThmMultipleFibr}
	Let $\Gamma<\Iso(\R^3)$ be a space group.
	
\begin{itemize}
\item If $\Gamma$ admits several invariant directions in $\R P^2$, then
	$\mathbb S^2/\rho(\Gamma)$ is diffeomorphic to $1,\times , \ast, 22,2\ast,\ast22,222,\ast222,2\times$ or $2\!\Kal\!2$.
	\item
	If moreover $\Gamma$ is orientation-preserving, then
	$\mathbb S^2/\rho(\Gamma)$ is diffeomorphic to $1,22$ or $222$.
	\end{itemize}
\end{prop}
By ``$1$'' here we simply mean $S^2$, that is, $\rho(\Gamma)$ is the trivial group.
\begin{proof}
	First, we claim that all elements of $\rho(\Gamma)$ have order at most 2. By hypothesis there exists two linearly independent eigenvectors $v,w\in \R^3$   of all the elements in $\rho(\Gamma)$.
Fix $g\in\rho(\Gamma)$. Since $g$ is orthogonal, then  there exists $\lambda,\mu\in\{\pm1\}$ such that 
\[g(v)=\lambda v\qquad g(w)=\mu w~.\]
Then $g^2(v)=\lambda g(v)=\lambda^2 v=v$ and also $g^2(w)=w$. Therefore $g^2$ is the identity on a plane. Since it is an orthogonal linear map, and it is orientation-preserving, $g^2=\mathrm{id}$.

Therefore all cone and corner points of  $\mathbb S^2/\rho(\Gamma)$ have singularity index $2$. The conclusion follows by direct inspection on Table \ref{table:2orbifolds}. 
%Let $\Oo=\R^3/\Gamma$ and $f_i:\Oo\to B_i$ for $i=1,2$ two inequivalent Seifert fibrations. By Proposition \ref{PropFlatFibrationAreInducedbyParallel Lines} there exist $\F_i$ for $i=1,2$ partition of $\R^3$ into parallel lines preserved, simultaneously. Since the two fibrations are inequivalent then $\F_1$ and $\F_2$ are distinct and   have lines with different directions; $\Gamma$ preserves at least $2$ different directions.By Proposition \ref{Invariant directions} we have that $S^2/\rho(\Gamma)\in \{1,22,222\}$, since they are the only orientable orbifolds in the list.
\end{proof} 

\begin{oss}
Proposition \ref{ThmMultipleFibr} can be refined by eliminating $2\times$ and $2\!\Kal\!2$ in the list of the first item, but this will not be necessary for the purposes of our work.
\end{oss}

\subsection{Multiple fibrations for flat orientable orbifolds.}

By Proposition \ref{ThmMultipleFibr}, to study orientation-preserving space groups $\Gamma$ whose quotients have several fibrations, it suffices to consider  those having point orbifold in $\{1,22,222\}$.

The following statement will be useful to further reduce the problem. 

\begin{prop}\label{PropThreedirectionTransaltions}Let $\Gamma$ be a space group, $G$ be a wallpaper group and $f:\R^3/\Gamma\to \R^2/G$ be a Seifert fibration. If $G$ has two linearly independent translations and $G$ preserves their directions, then $\Gamma$ has three linearly independent translations and it preserves their three directions.  
\end{prop}
\begin{proof}
	Construct the space group $\Gamma'$ starting from the Seifert fibration $f$ as in Corollary \ref{CorLines}. By Theorem \ref{Bieberbach} there exists an affine map  conjugating $\Gamma$ and $\Gamma'$. Therefore it will be sufficient to prove the statement for $\Gamma'$.
	
	Let $a,b\in G< \Iso(\R^2)$ be two linearly independent  translations. By the construction in Theorem \ref{FromSeifertToIsometry}, there exist elements $(a,\alpha+),(b,\beta+)\in\Gamma'\subseteq \Iso(\R^2)\times \Iso(\R)$ where $\alpha,\beta\in\Q$. 
	Then there exist $n\in\N^+$ and $m\in\Z$ such that 
	\[(a,\alpha+)^n=(a^n,n\alpha+)=(a^n,m+)~.\]
	
	Since by construction $\Gamma'$ admits the vertical translation $(\mathrm{id}_{\R^2},1+)$, $(a^n,0+)$ is an element of $\Gamma'$.
	Similarly $(b^{n'},0+)\in\Gamma'$. Clearly the three translations $(\mathrm{id}_{\R^2},1+),(a^n,0+)$ and $(b^{n'},0+)$ are linearly independent.
	
	 Now we will prove that the corresponding three directions are preserved by $\Gamma'$. We already know that  the vertical direction is preserved by the group $\Gamma'$. So let $[v]$ be the direction defined by one of the other two translations, and $L$ be a line with this direction, which will be of the form $L=L_0\times\{z_0\}$ where $L_0$ is a line in $\R^2$ and $z_0\in\R$. We want to show that for all $(g,h)\in\Gamma'$ the line $(g,h)(L)$ has direction $[v]$.
But $(g,h)(L)=g(L_0)\times\{h(z_0)\}$. By hypothesis the direction of $L_0$ is preserved by $G$, so $g(L_0)$ has the  direction of $L_0$ and therefore $(g,h)(L)$ has direction $[v]$.
\end{proof}

\subsubsection*{Point orbifold is diffeomorphic to $1$.}
First, let us prove the following lemma.
\begin{lem}\label{LemPoint 1}
	A closed orientable flat Seifert  3-orbifold has trivial point group if and only if it is diffeomorphic to $T^3$. Furthermore any Seifert fibration of $T^3$ is equivalent to the projection on the first factor $pr_1:T^3=T^2\times S^1\to T^2$.
\end{lem}
\begin{proof}Let $\Gamma$ be a space group such that $\rho(\Gamma)=\{1\}$. Then $\Gamma=T(\Gamma)$ consists only of translation, and by Theorem \ref{Bieberbach} it is equivalent to $\Z^3$. This shows that $\R^3/\Gamma\cong T^3$.  Conversely, if $\R^3/\Gamma$ is diffeomorphic to $T^3$, by Theorem \ref{Bieberbach} $\Gamma$ is affinely conjugate to the standard subgroup $\Z^3$ of translations hence its point group is trivial.
	
	The second part follows from Proposition \ref{PropFlatFibrationAreInducedbyParallel Lines}. Indeed if $f:T^3\to \Bb$ is a Seifert fibration, then  we can assume that $\Bb=v^\perp/\Gamma_H^{[v]}$ for some $v$.  This implies that  $\Gamma_H^{[v]}$ is a wallpaper group that consists of sole translations, and is thus  equivalent to the translations group $\Z^2$. Hence $\Bb\cong T^2$ and this concludes the proof.
\end{proof}

  \subsubsection*{Point orbifold diffeomorphic to $22$.}
  Next, we prove the following lemma.
  \begin{lem}\label{LemPoint 22}
  A closed orientable flat Seifert  3-orbifold has point orbifold diffeomorphic to $22$ if and only it is in the following table.
  		\begin{center}\begin{tabular}{lll}
  			$(2_02_02_02_0)$    & $(\Kal_0\Kal_0)$&  \\ \hline
  			$(2_02_02_12_1)$  & $(\Kal_1\Kal_1)$& $(\Kal_0\bar\times)$          \\ \hline 
  			$(2_12_12_12_1)$& $(\bar\times\bar\times)$&  \\ \hline
  		\end{tabular}
  	\end{center}
  Two of those Seifert fibered orbifolds  are orientation-preserving diffeomorphic if and only if they are in the same line of the table.
  \end{lem}
\begin{proof}
	We will split the proof into four steps.
	
\begin{steps}
\item {\it The base orbifold must be $2222,\Kal\Kal,\Kal\times$ or $\times\times$.}
	
	If $f:\R^3/\Gamma\to \Bb$ is a Seifert fibration for $\Gamma$ a space group, by Proposition \ref{PropFlatFibrationAreInducedbyParallel Lines} we can assume that $\Bb=\R^2/\Gamma_H^{[v]}$ where $[v]$ is  an invariant direction of $\Gamma$.
Now, if $\mathbb S^2/\rho(\Gamma)=22$ then $\rho(\Gamma)$ can be chosen, up to conjugation by a linear map,  to be the group  \[\rho(\Gamma)=\langle\begin{bmatrix}
	-1 & 0 & 0\\
	0 & -1 & 0\\
	0 & 0 & 1
\end{bmatrix}\rangle~.\] The invariant directions are the equivalence classes in $\R P^2$ of the following set \[(\mathbb S^1\times\{0\})\cup \{(0,0,\pm 1)\}\subseteq \mathbb S^2.\]
Therefore we have two possibilities:

\begin{enumerate}
\item if the invariant direction is $[e_3]$, we get 
\[\rho(\Gamma_H^{[e_3]})=\langle\begin{bmatrix}
		-1 & 0\\
		0 & -1
	\end{bmatrix}\rangle \cong \Z_2.\] \\
In this case $\R^2/\Gamma_H^{[e_3]}$ is orientable, has at least one cone point and each cone point has singularity index equal to $2$. From Table \ref{table:2orbifolds}, $\R^2/\Gamma_H^{[e_3]}$ is diffeomorphic to the orbifold $2222$.
	\item if the invariant direction is  $[v]$, where $v\in \mathbb S^1\times\{0\}$, we get
	\[\rho(\Gamma_H^{[v]})=\langle\begin{bmatrix}
		-1 & 0\\
		0 & 1
	\end{bmatrix}\rangle\cong D_2.\]\\
In this case $\R^2/\Gamma_H^{[v]}$ is non-orientable and does not have any cone or corner point. From Table \ref{table:2orbifolds}, $\R^2/\Gamma_H^{[v]}$ is diffeomorphic to one of the following orbifolds: $\Kal\Kal$, $\Kal\times$ or $\times\times$.
\end{enumerate}

\smallskip

\item {\it All possibles Seifert invariants for a fibration of $\Oo$ are listed in the table.}

By a direct computation using  the fact that the Euler number is zero and the identity in  Theorem \ref{InvariantRelation}, one checks that all  possible Seifert fibrations are listed in the statement.

\smallskip

\item {\it The orbifolds $(2_02_02_02_0),(2_02_02_12_1)$ and $(2_12_12_12_1)$ are not diffeomorphic.}
	
Each cone point in the base orbifold with Seifert invariant $2_0$ gives a singular fiber while the fibers associated to the cone points of invariant $2_1$ are not singular. By counting the components of the singular set we can distinguish the three orbifolds.

\smallskip
	
\item {\it Seifert orbifolds in the same line are orientation-preserving diffeomorphic.}

We first claim that a Seifert fibered orbifold with base space $\Bb\in\{\Kal\Kal,\Kal\times,\times\times\}$ from the table in the statement also admits a Seifert fibration with base space the orbifold $2222$.
%So let $\Gamma$ be a space group such that such that $\rho(\Gamma)=<diag(-1,-1,1)>$

To see this, let $\Gamma$ be a space group, $\R^3/\Gamma$ be one of those Seifert fibered orbifolds and $f:\R^3/\Gamma\to \Bb$ the associated Seifert fibration. As above, up to a conjugation with an orthogonal map we can assume that $\rho(\Gamma)$ is generated by the matrix $\mathrm{diag}(-1,-1,1)$, since the point orbifold is diffeomorphic to $22$.

Since $\Bb=\R^2/G$ is diffeomorphic to $\Kal\Kal,\Kal\times$ or $\times\times$, one can show (see Remark \ref{rmk reconstruct 1}) that in all three cases $G$ has two linearly independent translations whose directions are preserved  by $G$. By Proposition \ref{PropThreedirectionTransaltions} the group $\Gamma$ has three linearly independent translations whose directions are preserved by $\Gamma$. Therefore, by Proposition \ref{PropLinePresGroupWithVerical},  $\R^3/\Gamma$ admits Seifert fibrations induced by those directions. By Step 1, all invariant directions lie in a plane with a single exception, the direction $[e_3]$. As shown in Step 1, the base orbifold induced by the direction $[e_3]$ is diffeomorphic to $\R^2/\Gamma_H^{[e_3]}\cong 2222$. 

 Finally, we can identify the correct flat orbifold by looking at the number of components of the singular locus.

\begin{center}
	\begin{tabular}{l|llll}
		Orbifold $\Oo$              & $(\Kal_0\Kal_0)$ & $(\Kal_1\Kal_1)$ & $(\Kal_0\bar\times)$& $(\bar\times\bar \times)$ \\ \hline
		Components of $\Sigma_{\Oo}$ & $4$              & $2$            &$2$  & $0$             
	\end{tabular}
	
\end{center}
This concludes the proof. \qedhere
\end{steps}
\end{proof}
%STEP4------------------------------------------------------------------

%------------------------------------------------------------

\subsubsection*{Point orbifold is diffeomorphic to $222$.\label{Section Point orbifold}}
Finally, let us consider the remaining case.
\begin{lem}\label{LemPoint 222}
  A closed orientable flat Seifert 3-orbifold has point orbifold diffeomorphic to $222$ if and only it appears in one of the following tables:
	\begin{center}\begin{tabular}{ll}
		$(2_02_0\Kal_0)$ &$(\Kal_02_12_12_12_1)$  \\ \hline 
		$(2_02_1\Kal_1)$ &$(2_1\Kal_02_12_1)$   \\ \hline 
		$(2_12_1\Kal_0)$ &$(2_02_0\bar\times)$   \\ \hline 
		$(2_0\Kal_02_02_0)$&$(\Kal_12_02_02_12_1)$   \\ \hline 
	\end{tabular}
$\qquad$
\begin{tabular}{l}
	$(\Kal_02_02_02_02_0)$   \\ \hline 
	$(\Kal_12_02_12_02_1)$ \\ \hline 
	$(2_1\Kal_12_02_0)$    \\ \hline 
	$(2_0\Kal_12_12_1)$   \\ \hline 
	$(2_12_1\bar\times)$   \\ \hline 
\end{tabular}
\end{center}
  Two of those Seifert fibered orbifolds are orientation-preserving diffeomorphic if and only if they are in the same line of the left table.
\end{lem}
\begin{proof}
	Again we will divide the proof in four steps.

\begin{steps}
\item {\it The base orbifold must be $\Kal 2222,2\!\Kal\!22,22\Kal$ or $22\times$. }

If $f:\R^3/\Gamma\to \Bb$ is a Seifert fibration for $\Gamma$ a space group, by Proposition \ref{PropFlatFibrationAreInducedbyParallel Lines} we can assume that $\Bb=\R^2/\Gamma_H^{[v]}$ where $[v]$ is  an invariant direction of $\Gamma$.
Now, if $\mathbb S^2/\rho(\Gamma)=222$ then $\rho(\Gamma)$ can be chosen, up to conjugation by a linear map,  to be the group
\[\rho(\Gamma)=\langle\begin{bmatrix}
	-1 & 0 & 0\\
	0 & -1 & 0\\
	0 & 0 & 1
\end{bmatrix}, \begin{bmatrix}
-1 & 0 & 0\\
0 & 1 & 0\\
0 & 0 & -1
\end{bmatrix}, \begin{bmatrix}
1 & 0 & 0\\
0 & -1 & 0\\
0 & 0 & -1
\end{bmatrix}\rangle\] and the invariant directions are  $\{[e_1], [e_2],[e_3]\}\subseteq \R P^2$ where $e_1,e_2,e_3$ is the standard basis of $\R^3$. Therefore for $v=e_1,e_2,e_3$ we have:
  \[\rho(\Gamma_H^{[v]})=\langle\begin{bmatrix}
		-1 & 0\\
		0 & -1
	\end{bmatrix},\begin{bmatrix}
	-1 & 0\\
	0 & 1
\end{bmatrix},\begin{bmatrix}
1 & 0\\
0 & -1
\end{bmatrix}\rangle\cong D_4.\] \\This implies that $\R^2/\Gamma_{[v]}$ is non-orientable, has at least one cone point or corner point and each cone or corner point has singularity index equal to $2$. 
From Table \ref{table:2orbifolds},  $\R^2/\Gamma_H^{[v]}$ is diffeomorphic to one of the following orbifolds:
\[\Kal 2222\quad 2\!\Kal\!22\quad22\!\Kal\quad22\times\]

\smallskip
%STEP 2-------------------------------------------------------

\item  {\it All possible Seifert invariants for a fibration of $\Oo$ are listed in the table.}

 As in the previous case, using  Theorem \ref{InvariantRelation}  one can check that those listed in the statement are all the possible Seifert fibrations with base orbifold as in Step 1.

\smallskip

\item  {\it Seifert   orbifolds in the first three lines are orientation-preserving diffeomorphic.}

We will first show  the following three equivalences
\[(2_02_0\Kal_0)\cong(\Kal_02_12_12_12_1)\qquad(2_02_1\Kal_1)\cong(2_1\Kal_02_12_1)\qquad (2_12_1\Kal_0)\cong(2_02_0\bar\times)\]
by constructing, for each pair, a space group $\Gamma$ that admits two invariant directions inducing the correct fibrations. In the next step we will prove by a similar method the existence of a diffeomorphism between $(\ast_02_02_02_12_1)$ and $(2_0\ast_02_02_0)$.

The main idea to construct the space group $\Gamma$ is to apply concretely the construction in Corollary \ref{CorLines}. Let
%------------------%%%%%-----------------------%
\[G=\langle\left(\begin{bmatrix}	-1 & 0\\	0 & -1\\\end{bmatrix},\begin{bmatrix}	0\\1/2\end{bmatrix}\right),\left(\begin{bmatrix}-1 & 0\\0 & -1\\\end{bmatrix},\begin{bmatrix}	1\\1/2\end{bmatrix}\right),\left(\begin{bmatrix}1 & 0\\0 & -1\\\end{bmatrix},\begin{bmatrix}	0\\0\end{bmatrix}\right)\rangle< \Iso(\R^2)\]
be the wallpaper group generated by two rotations of $\pi$ with fixed points $(0,1/4)$ and $(1/2,1/4)$, and a reflection in the line $y=0$. Figure \ref{figure22ast}  shows a fundamental domain of this group action, and 
 $\R^2/G$ is diffeomorphic to $22\ast$. 
 \begin{figure}[htb]	\centering%\includegraphics[width=0.3\textwidth]{22ast.png}\end{figure}
	\includegraphics[height=5cm]{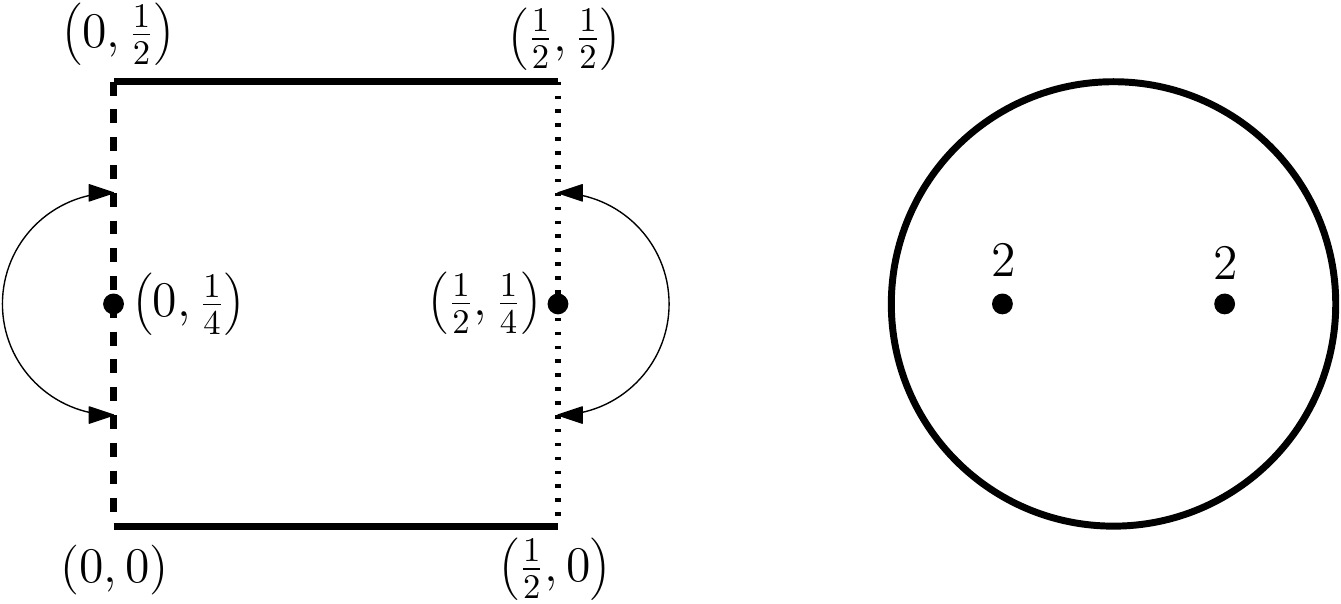}
	\caption{\small On the left a fundamental domain for the action, on the right the quotient $ 22\ast$. \label{figure22ast}}\end{figure}
For fixed $a,b\in\{0,1\}$, we can apply the construction in Corollary \ref{CorLines} to the fibration $(2_a2_b\ast)$, where the boundary invariant is determined by the first two. This gives the space group:
\[\Gamma=\langle\left(A,\begin{bmatrix}	0\\1/2\\a/2\end{bmatrix}\right), \left(A,\begin{bmatrix}1\\1/2\\b/2\end{bmatrix}\right),\left(B,\begin{bmatrix}0\\0\\0\end{bmatrix}\right), \left(\mathrm{Id},\begin{bmatrix}0\\0\\1\end{bmatrix}\right)\rangle< \Iso(\R^2)\times \Iso(\R)~,\]
where 
\[A=\begin{bmatrix}		-1 & 0&0\\		0 & -1&0\\		0&0&1\end{bmatrix}\qquad B=\begin{bmatrix}	1 & 0&0\\	0 & -1&0\\	0&0&-1\end{bmatrix}~.\] 
By construction $\Gamma$ admits a vertical translation, $[e_3]$ is an invariant direction, $G=\Gamma_H^{[e_3]}$ and the fibration $f_3:\R^3/\Gamma\to \R^2/G$ induced by the partition in lines parallel to $[e_3]$ is equivalent to the fibration $(2_a2_b\ast)$. 
%This last property follows directly from the construction in Corollary \ref{CorLines} and in Theorem \ref{FromSeifertToIsometry}, but one can explicitly check it by computing the Seifert invariants for the fibration $f_3$ as in Appendix \ref{An explcit computation}.

	It is easy to see that $G$ admits two translation in the directions $[e_1],[e_2]\in\R P^1$ and that these directions are preserved by $G$.
	By Proposition \ref{PropLinePresGroupWithVerical}, that $\Gamma$ admits three linearly independent translation with directions preserved by $\Gamma$. Since $\Gamma$ preserves only the directions $[e_1],[e_2],[e_3]$, we have that $\R^3/\Gamma$ has three Seifert fibrations $f_1,f_2,f_3$ induced by the partition in lines with directions $[e_1],[e_2]$ and $[e_3]$.
Furthermore $f_i:\R^3/\Gamma\to\R^2/\Gamma_H^{[e_i]}$ has base orbifold $\R^2/\Gamma_H^{[e_i]}$, where
	\begin{align*}
	\Gamma_H^{[e_1]}&=\langle\left(\begin{bmatrix}	-1 & 0\\	0 & 1\\\end{bmatrix},\begin{bmatrix}	1/2\\a/2\end{bmatrix}\right),\left(\begin{bmatrix}-1 & 0\\0 & 1\\\end{bmatrix},\begin{bmatrix}	1/2\\b/2\end{bmatrix}\right),\left(\begin{bmatrix}-1 & 0\\0 & -1\\\end{bmatrix},\begin{bmatrix}	0\\0\end{bmatrix}\right),\left(\mathrm{Id},\begin{bmatrix}	0\\1\end{bmatrix}\right)\rangle \\
	\Gamma_H^{[e_2]}&=\langle\left(\begin{bmatrix}	-1 & 0\\	0 & 1\\\end{bmatrix},\begin{bmatrix}	0\\a/2\end{bmatrix}\right),\left(\begin{bmatrix}-1 & 0\\0 & 1\\\end{bmatrix},\begin{bmatrix}	1\\b/2\end{bmatrix}\right),\left(\begin{bmatrix}1 & 0\\0 & -1\\\end{bmatrix},\begin{bmatrix}	0\\0\end{bmatrix}\right),\left(\mathrm{Id},\begin{bmatrix}	0\\1\end{bmatrix}\right)\rangle
	\end{align*}

Let us now distinguish three cases.
If $a=0$ and $b=0$, the group
	\begin{align*}
	\Gamma_H^{[e_2]}&=\langle\left(\begin{bmatrix}	-1 & 0\\	0 & 1\\\end{bmatrix},\begin{bmatrix}	0\\0\end{bmatrix}\right),\left(\begin{bmatrix}-1 & 0\\0 & 1\\\end{bmatrix},\begin{bmatrix}	1\\0\end{bmatrix}\right),\left(\begin{bmatrix}1 & 0\\0 & -1\\\end{bmatrix},\begin{bmatrix}	0\\0\end{bmatrix}\right),\left(\mathrm{Id},\begin{bmatrix}	0\\1\end{bmatrix}\right)\rangle \\
	&=\langle\left(\begin{bmatrix}	-1 & 0\\	0 & 1\\\end{bmatrix},\begin{bmatrix}	0\\0\end{bmatrix}\right),\left(\begin{bmatrix}-1 & 0\\0 & 1\\\end{bmatrix},\begin{bmatrix}	1\\0\end{bmatrix}\right),\left(\begin{bmatrix}1 & 0\\0 & -1\\\end{bmatrix},\begin{bmatrix}	0\\0\end{bmatrix}\right),\left(\begin{bmatrix}1 & 0\\0 & -1\\\end{bmatrix},\begin{bmatrix}	0\\1\end{bmatrix}\right)\rangle
	\end{align*}
	 is generated by the reflections in the edges of the square whose vertices have coordinates $0$ or $1/2$. 
	Then $\R^2/\Gamma_H^{[e_2]}$ is diffeomorphic to $\ast2222$.  Since $(2_02_0\ast_0)$ does not have vertex  points, then the Seifert invariants of the fibration $f_2$ over every corner point cannot be equal to $0/2$, and is therefore equal to $1/2$. Therefore $\R^3/\Gamma$ is diffeomorphic to $(\ast_02_12_12_12_1)$, where the boundary invariant can be determined by Theorem \ref{InvariantRelation} (see Remark \ref{rmk:bdy invariant determined}). The two inequivalent fibrations are shown in Figure \ref{fig:example intro}.
	%See figure \ref{FigMultiFiberExample} for a topological view of those inequivalent fibrations.

If $a=0$ and $b=1$, the group	
\begin{align*}
\Gamma_H^{[e_2]}&=\langle\left(\begin{bmatrix}	-1 & 0\\	0 & 1\\\end{bmatrix},\begin{bmatrix}	0\\0\end{bmatrix}\right),\left(\begin{bmatrix}-1 & 0\\0 & 1\\\end{bmatrix},\begin{bmatrix}	1\\1/2\end{bmatrix}\right),\left(\begin{bmatrix}1 & 0\\0 & -1\\\end{bmatrix},\begin{bmatrix}	0\\0\end{bmatrix}\right),\left(\mathrm{Id},\begin{bmatrix}	0\\1\end{bmatrix}\right)\rangle\\
	&=\langle\left(\begin{bmatrix}	-1 & 0\\	0 & 1\\\end{bmatrix},\begin{bmatrix}	0\\0\end{bmatrix}\right),\left(\begin{bmatrix}-1 & 0\\0 & -1\\\end{bmatrix},\begin{bmatrix}	1\\1/2\end{bmatrix}\right),\left(\begin{bmatrix}1 & 0\\0 & -1\\\end{bmatrix},\begin{bmatrix}	0\\0\end{bmatrix}\right),\left(\begin{bmatrix}1 & 0\\0 & -1\\\end{bmatrix},\begin{bmatrix}	0\\1\end{bmatrix}\right)\rangle
	\end{align*}
	is generated by three reflections and a rotation of order 2 with fixed point $(1/2,1/4)$.
	
	\begin{figure}[h!]
		\centering
		\includegraphics[height=5cm]{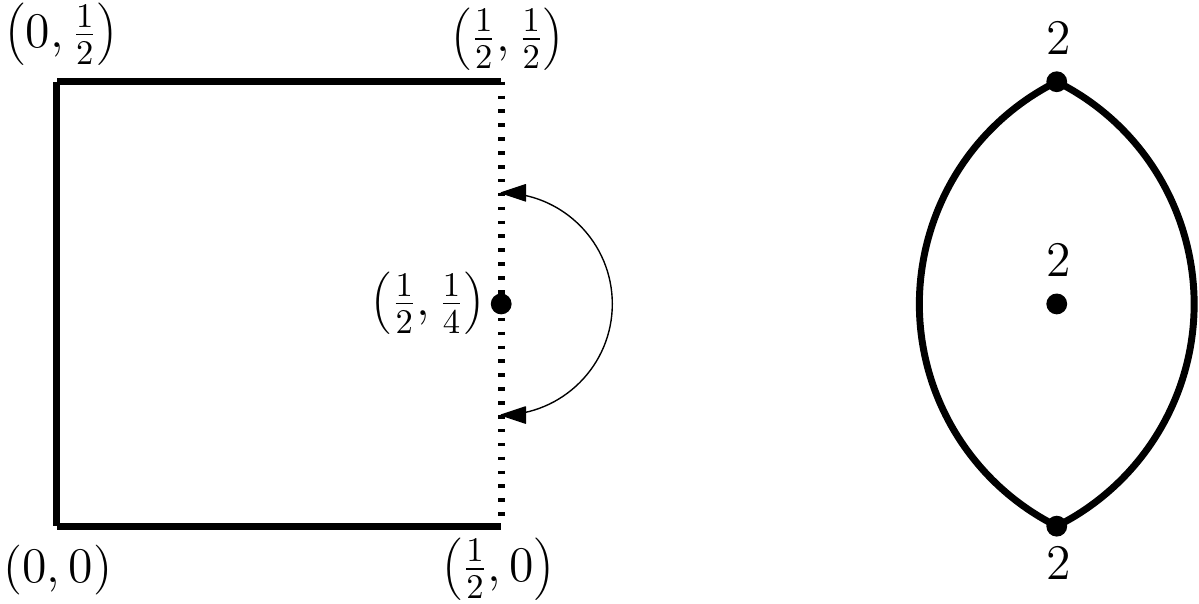}%{CornerNonStandard.png}
		\caption{\small On the left a fundamental domain for the action, on the right the quotient $ 2\!\ast\!22$. \label{figure22asrcon011}}
	\end{figure}
	
	A fundamental domain is shown in Figure \ref{figure22asrcon011}, and $\R^2/\Gamma_{[e_2]}$ is diffeomorphic to $ 2\!\ast\!22$. 
	As in the previous case, since $(2_02_1\ast_1)$ does not have vertex points, all local invariants of the fibration $f_2$ over a corner point must be $1/2$. Moreover the singular locus of $(2_02_1\ast_1)$ has two connected components, which implies that the local invariant for $f_2$ over the cone point must be $1/2$.
Indeed, if it was $0/2$, then the fibration $(2_0\ast_12_12_1)$ would have $3$ connected component, see Figure \ref{fig: sing loci} (E).
	%The element of $\Gamma$ corresponding to the rotation in $\Gamma_{[e_2]}$ with centre $(1/2,1/4)$ is 	\[(\begin{bmatrix}	-1 & 0&0\\	0 & 1&0\\0&0&-1\\\end{bmatrix},\begin{bmatrix}	1\\1/2\\1/2\end{bmatrix})\], then the local invariant at the cone point is $1/2$. The local invariant at the corner point image of $(0,0)$ is $1/2$ since the isometry \[(\begin{bmatrix}	-1 & 0&0\\	0 & 1&0\\0&0&-1\\\end{bmatrix},\begin{bmatrix}	1\\1/2\\1/2\end{bmatrix})\in\Gamma\] project  to the rotation in $\Gamma_{[e_2]}$ with centre $(0,0)$.	By Theorem \ref{InvariantRelation} follows that also the local invariant at the other corner point is $1/2$ and the boundary invariant is $0$.
	Therefore $\R^3/\Gamma$ is diffeomorphic to $(2_1\ast_02_12_1)$, where the boundary invariant is determined by Remark \ref{rmk:bdy invariant determined} as before.
	
Finally, if $a=1$ and $b=1$,
	the group
		\begin{align*}
		\Gamma_H^{[e_1]}&=\langle\left(\begin{bmatrix}	-1 & 0\\	0 & 1\\\end{bmatrix},\begin{bmatrix}	1/2\\1/2\end{bmatrix}\right),\left(\begin{bmatrix}-1 & 0\\0 & -1\\\end{bmatrix},\begin{bmatrix}	0\\0\end{bmatrix}\right),\left(\mathrm{Id},\begin{bmatrix}	0\\1\end{bmatrix}\right)\rangle\\
	&=\langle\left(\begin{bmatrix}	-1 & 0\\	0 & 1\\\end{bmatrix},\begin{bmatrix}	1/2\\1/2\end{bmatrix}\right),\left(\begin{bmatrix}-1 & 0\\0 & -1\\\end{bmatrix},\begin{bmatrix}	0\\0\end{bmatrix}\right),\left(\begin{bmatrix}-1 & 0\\0 & -1\\\end{bmatrix},\begin{bmatrix}	0\\1\end{bmatrix}\right)\rangle
	\end{align*}
	is generated by a vertical glide reflection and two order 2 rotations of with fixed points $(0,0)$ and $(0,1/2)$.
		\begin{figure}[htp]
		\centering
		\includegraphics[width=0.70\textwidth]{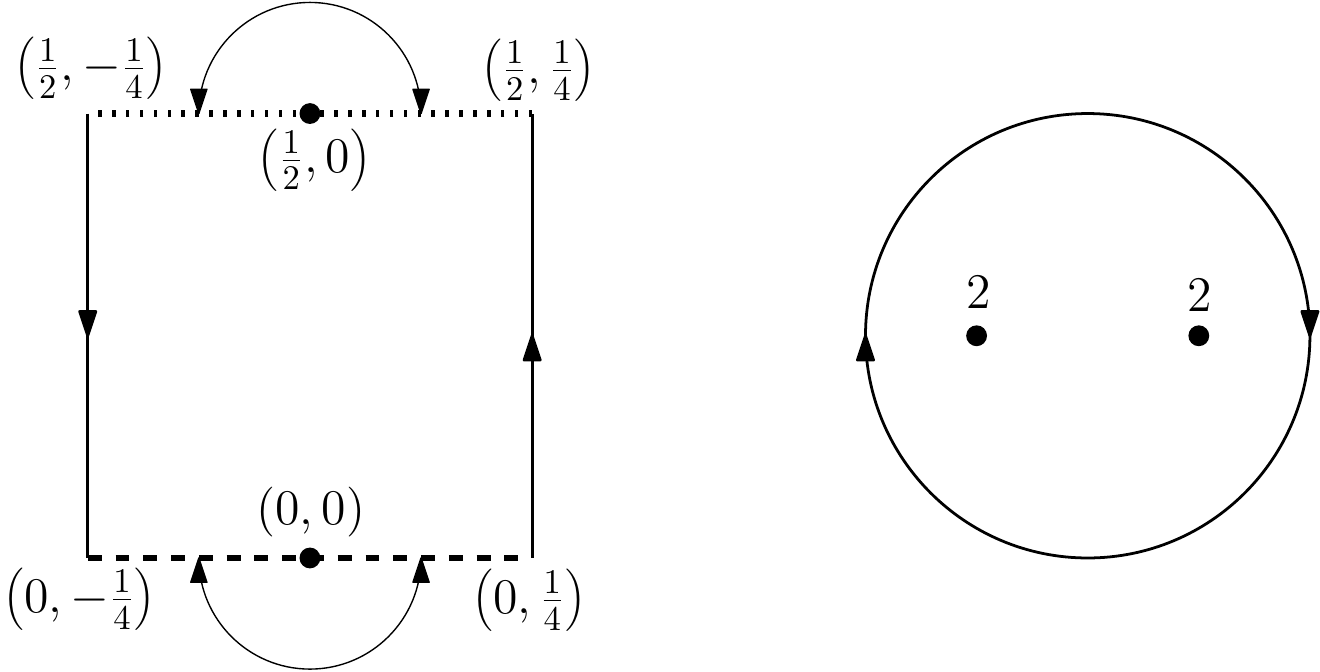}%{CornerNonStandard.png}
		\caption{\small On the left a fundamental domain for the action, on the right the quotient $22\times$.\label{figure22timescon00}}
	\end{figure}

	 	Figure \ref{figure22timescon00} shows a fundamental domain, and $\R^2/\Gamma_H^{[e_1]}$ is diffeomorphic to $ 22\times$. Since the singular locus $(2_12_1\ast_0)$ has two connected components, the local invariant over a cone point for the fibration $f_1$ is necessarily $0/2$.
	 	Therefore $\R^3/\Gamma$ is diffeomorphic to $(2_02_0\bar\times)$.

\item{\it The orbifolds $(\ast_12_02_02_12_1)$ and $(2_0\ast_02_02_0)$ are orientation-preserving diffeomorphic.}

Consider now the wallpaper group 
$$G=\langle \left(\begin{bmatrix}1 & 0\\0 & -1\\\end{bmatrix},\begin{bmatrix}	0\\0\end{bmatrix}\right), \left(\begin{bmatrix}	-1 & 0\\	0 & 1\\\end{bmatrix},\begin{bmatrix}	1\\0\end{bmatrix}\right),\left(\begin{bmatrix}1 & 0\\0 & -1\\\end{bmatrix},\begin{bmatrix}	0\\1\end{bmatrix}\right),\left(\begin{bmatrix}-1 & 0\\0 & 1\\\end{bmatrix},\begin{bmatrix}	0\\0\end{bmatrix}\right)\rangle$$
generated by the reflections in the edges of the square whose vertices have coordinates $0$ or $1/2$, so that $\R^2/G$ is diffeomorphic to $\ast2222$.
As in the previous step, by applying the construction of Corollary \ref{CorLines} to the fibration $(\ast_12_02_02_12_1)$, we  define the space group:
\[\Gamma=\langle\left(A,\begin{bmatrix}	0\\0\\0\end{bmatrix}\right), \left(B,\begin{bmatrix}1\\0\\1/2\end{bmatrix}\right),\left(A,\begin{bmatrix}0\\1\\0\end{bmatrix}\right),\left(B,\begin{bmatrix}0\\0\\0\end{bmatrix}\right), \left(\mathrm{Id},\begin{bmatrix}0\\0\\1\end{bmatrix}\right)\rangle.\]
where now
\[A=\begin{bmatrix}		1 & 0&0\\		0 & -1&0\\		0&0&-1\end{bmatrix}\qquad B=\begin{bmatrix}	-1 & 0&0\\	0 & 1&0\\	0&0&-1\end{bmatrix}~.\] 

By construction, $\Gamma$ has  a translation in the invariant direction $[e_3]$, $G=\Gamma_H^{[e_3]}$ and the fibration  $f_3:\R^3/\Gamma\to\R^2/\Gamma_H^{[e_3]}$ induced by the parallel lines with the direction $[e_3]$ is equivalent to the fibration $(\ast_12_02_02_12_1)$.

%As before, since $G$ has two linearly independent translations in invariant directions, by Proposition \ref{PropThreedirectionTransaltions}  the space group $\Gamma$  admits three Seifert fibrations $f_1,f_2,f_3$ induced by the partitions in lines with directions $[e_1],[e_2]$ and $[e_3]$. 
Now, observe that $[e_2]$ is an invariant direction for $\Gamma$ and that $\Gamma$ contains the translation in the direction $e_2$, obtained by composing the first and third generator. Hence we consider the fibration \[f_2:\R^3/\Gamma\to\R^2/\Gamma_H^{[e_2]}.\]
 By a direct computation, we obtain:
\begin{align*}
\Gamma_H^{[e_2]}&=\langle\left(\begin{bmatrix}	1 & 0\\	0 & -1\\\end{bmatrix},\begin{bmatrix}	0\\0\end{bmatrix}\right),\left(\begin{bmatrix}-1 & 0\\0 & -1\\\end{bmatrix},\begin{bmatrix}	1\\1/2\end{bmatrix}\right),\left(\begin{bmatrix}-1 & 0\\0 & -1\\\end{bmatrix},\begin{bmatrix}	0\\0\end{bmatrix}\right),\left(\mathrm{Id},\begin{bmatrix}	0\\1\end{bmatrix}\right)\rangle\\
	&=\langle\left(\begin{bmatrix}	1 & 0\\	0 & -1\\\end{bmatrix},\begin{bmatrix}	0\\0\end{bmatrix}\right),\left(\begin{bmatrix}-1 & 0\\0 & -1\\\end{bmatrix},\begin{bmatrix}	1\\1/2\end{bmatrix}\right),\left(\begin{bmatrix}-1 & 0\\0 & 1\\\end{bmatrix},\begin{bmatrix}	0\\0\end{bmatrix}\right),\left(\begin{bmatrix}1 & 0\\0 & -1\\\end{bmatrix},\begin{bmatrix}	0\\1\end{bmatrix}\right)\rangle~.
	\end{align*}
This group is generated by three reflection and an order 2 rotation with fixed point $(1/2,1/4)$, again as in Figure  \ref{figure22asrcon011}. The quotient $\R^2/\Gamma_H^{[e_2]}$ is therefore diffeomorphic to $2\ast22$.

Since $(\ast_12_02_02_12_1)$ has four vertex points (see Figure \ref{fig:one sing locus}), the local invariant of both corner points for the fibration $f_2$ has to be equal to $0/2$.
By Theorem \ref{InvariantRelation} the only possibilities for the fibration $f_2$ are $(2_0\ast_02_02_0)$ and $(2_1\ast_12_02_0)$. However we can exclude the latter since the singular locus of $(\ast_12_02_02_12_1)$ has two connected components (Figure \ref{fig:one sing locus} again), while $(2_1\ast_12_02_0)$ has connected singular locus, see Figure \ref{fig: sing loci} (D).
This shows that $\R^3/\Gamma$ is orientation-preserving diffeomorphic to $(\ast_12_02_02_12_1)$ and $(2_0\ast_02_02_0)$.\\

\begin{figure}[htb]
\includegraphics[width=0.25\textwidth]{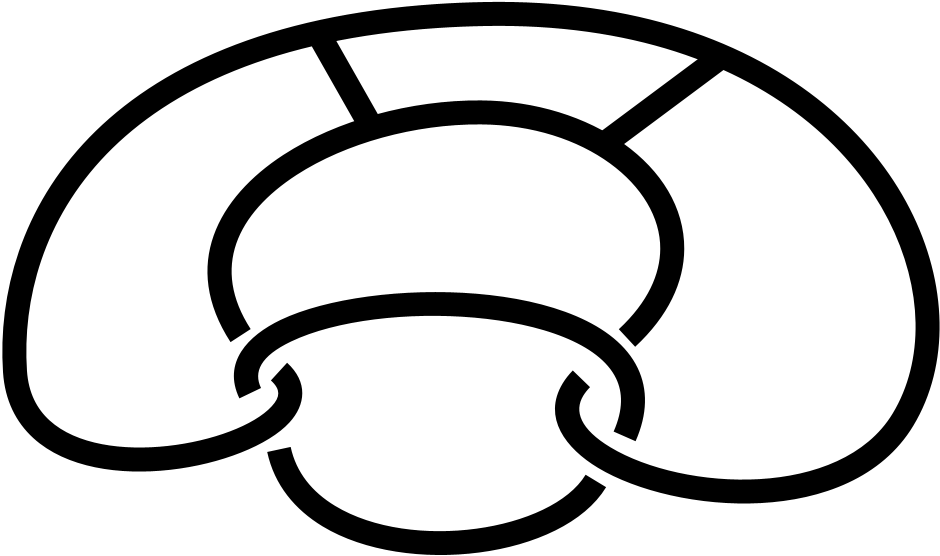}
\caption{\small The singular locus of $(\ast_12_02_02_12_1)$, which has $S^3$ as underlying manifold, and all the edges have singularity index equal to $2$.}
% (instructions to compute the underlying topological space can be found in Subsection \ref{SxR})}
\label{fig:one sing locus}
\end{figure}

\item {\it The orbifolds in different lines are not diffeomorphic.}

To distinguish the orbifolds in different lines of the table, we  analyse their singular loci, which are easily found following Section \ref{subsec:seifert fibr}. We summarize the structure of the singular loci in Table \ref{table:sing loci}; see also Figure \ref{fig: sing loci}.

\begin{table}[!htb]
\centering
\begin{tabular}{l|l|ll}
	Orbifold $\Oo$ & Description of $\Sigma_{\Oo}$ & Components of $\Sigma_{\Oo}$ & Vertex points \\ \hline
	$(2_02_0\ast_0)$&   four circles with singularity index$= 2$                            & $4$                       &                        $0$\\
	$(2_02_1\ast_1)$&two circles with singularity index$=2$&$2$&$0$\\
	$(2_12_1\ast_0)$&two circles with singularity index$=2$&$2$&$0$\\
	$(2_0\ast_02_02_0)$& see Figure \ref{2ast22all0}&$2$&$4$\\
	$(\ast_02_02_02_02_0)$&see Figure \ref{ast2222all0}&$1$&$8$\\
	$(\ast_12_02_12_02_1)$&see Figure \ref{ast2222con10101}&$2$&$4$\\
	$(2_1\ast_12_02_0)$&see Figure \ref{2ast22con1100}&$1$&$4$\\
	$(2_0\ast_12_12_1)$&see Figure \ref{2ast22con0111}&$3$&$0$\\
	$(2_12_1\bar\times)$&Empty&$0$&$0$\\
	
\end{tabular}
\caption{\small Structure of the singular loci}\label{table:sing loci}
\end{table}

\begin{figure}
	\centering
	\begin{subfigure}[b]{0.17\textwidth}
		\centering
		\includegraphics[width=\textwidth]{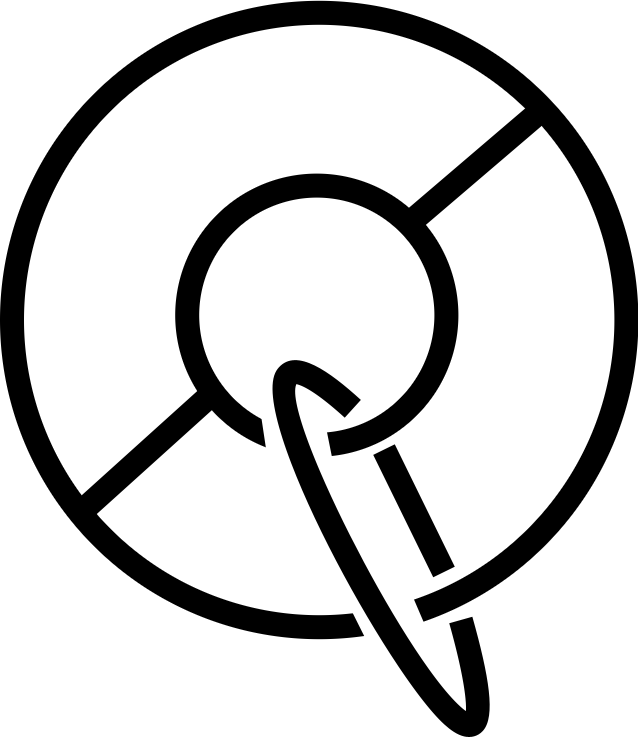}
		\caption{$(2_0\ast_02_02_0)$}
		\label{2ast22all0}
	\end{subfigure}
\hfill
	\begin{subfigure}[b]{0.17\textwidth}
	\centering
	\includegraphics[width=\textwidth]{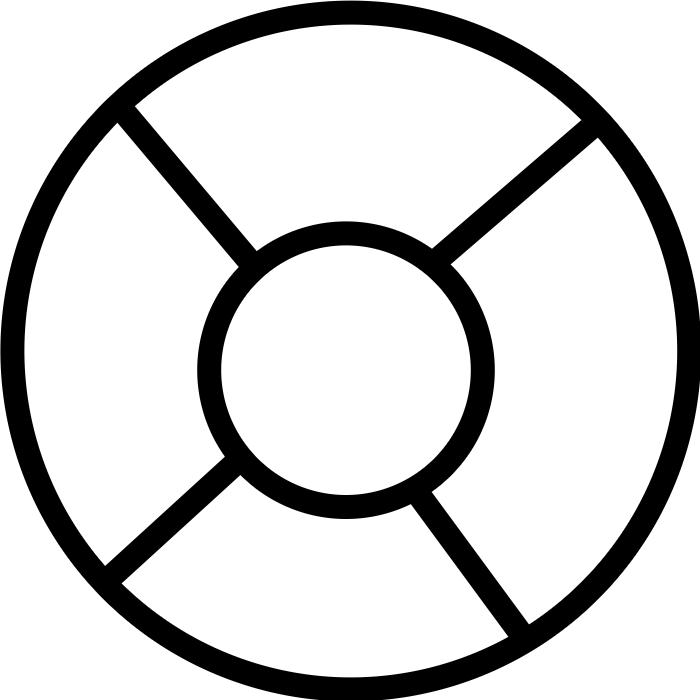}
	\caption{$(\ast_02_02_02_02_0)$}
	\label{ast2222all0}
\end{subfigure}
\hfill
	\begin{subfigure}[b]{0.17\textwidth}
	\centering
	\includegraphics[width=\textwidth]{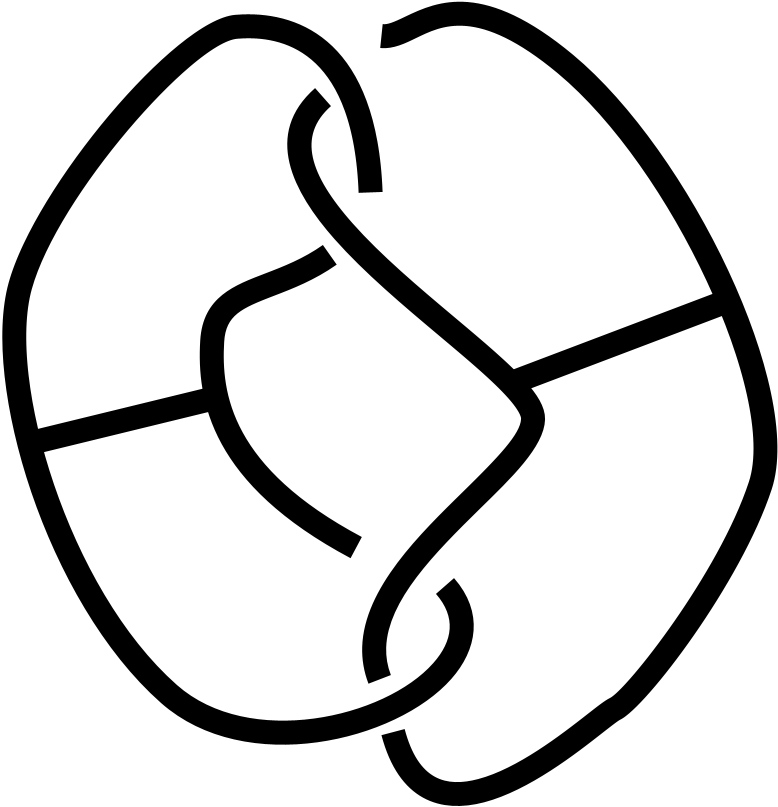}
	\caption{$(\ast_12_02_12_02_1)$}
	\label{ast2222con10101}
\end{subfigure}
\hfill
	\begin{subfigure}[b]{0.17\textwidth}
	\centering
	\includegraphics[width=\textwidth]{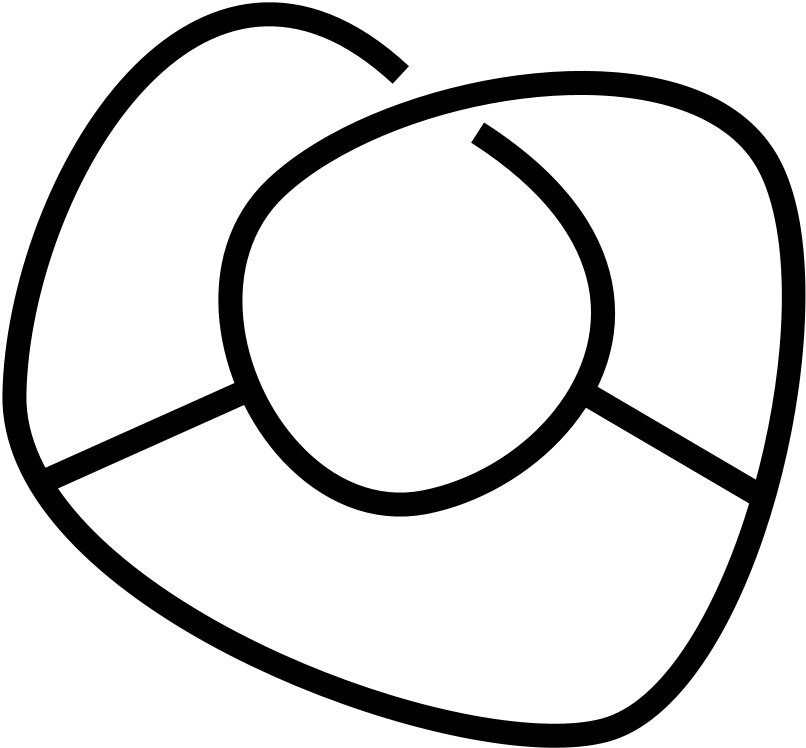}
	\caption{$(2_1\ast_12_02_0)$}
	\label{2ast22con1100}
\end{subfigure}
\hfill
	\begin{subfigure}[b]{0.17\textwidth}
	\centering
	\includegraphics[width=\textwidth]{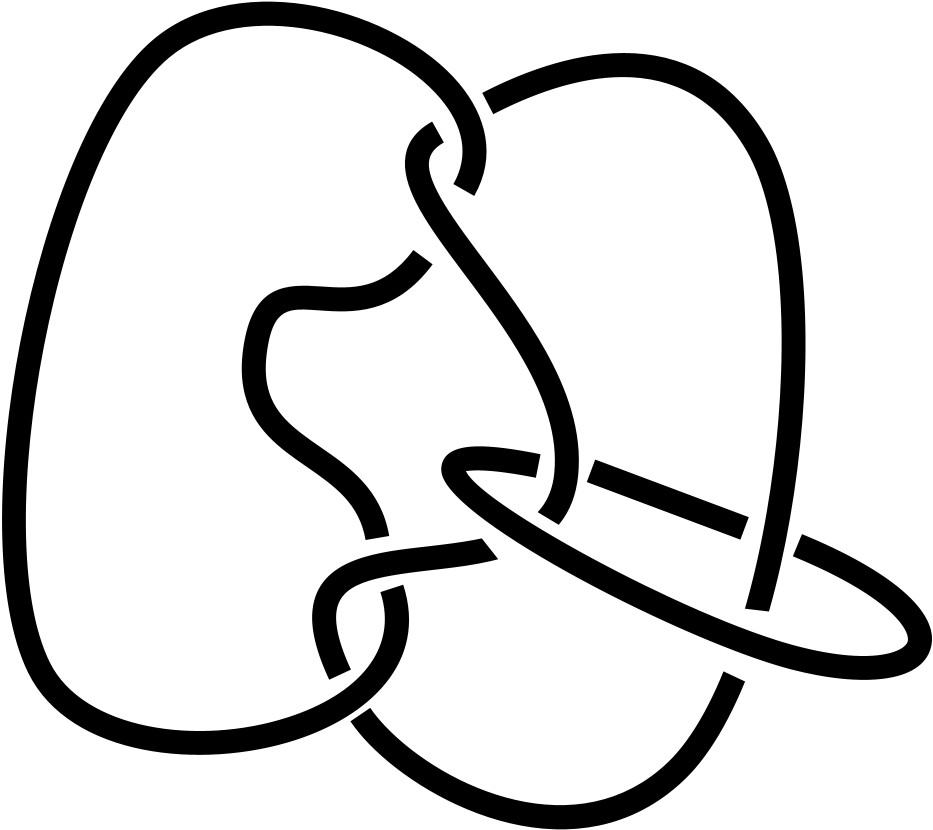}
	\caption{$(2_0\ast_12_12_1)$}
	\label{2ast22con0111}
\end{subfigure}
\caption{\small The orbifolds in  (A),(B),(C),(E) have $S^3$ as underlying manifold; the one in (D) has underlying manifold $\mathbb R P^3$. All the edges have singularity index equal to $2$. In (E) we can recognize the Borromean link, see also \cite[Example 13.1.5]{Thurstone} for a nice description of this orbifold.}
% (instructions to compute the underlying topological space can be found in Subsection \ref{SxR})} 
\label{fig: sing loci}
\end{figure}

All the orbifolds in the first column of Table \ref{table:sing loci} can be distinguish using the information from the last two columns, with two exceptions.
\begin{enumerate}\item $(2_0\ast_02_02_0)$ and $(\ast_12_02_12_02_1)$ have the same number of vertex points and components of the singular locus, but they are not diffeomorphic since in the former all vertex points lie in one connected component of the singular locus, while for the latter they are in different components. See Figure \ref{fig: sing loci}, (A) and (C). %\ref{2ast22all0} and \ref{ast2222con10101}.
	\item $(2_02_1\ast_1)$ and $(2_12_1\ast_0)$ have the same number of vertex points and components of the singular locus. Suppose by contradiction that they are orientation-preserving diffeomorphic. Since each of them has 2 inequivalent fibrations, there would exist a flat orbifold with 4 inequivalent fibrations with point orbifold equal to $222$. That is impossible by Proposition \ref{PropFlatFibrationAreInducedbyParallel Lines} and Step 1, since such an orbifold can have at most 3 inequivalent fibrations, one for each invariant direction.\qedhere
\end{enumerate}
\end{steps}
\end{proof}

%---------------------------------------------------------------

\subsection{Conclusion of the proof of Theorem \ref{FlatMultipleFibrationOrientable}}

We are now ready to conclude the main result of the section. 
\begin{proof}[Proof of Theorem \ref{FlatMultipleFibrationOrientable}]
	If a compact orientable flat Seifert orbifolds has several inequivalent Seifert fibrations, then its point orbifold is diffeomorphic to 1,22, or 222 by Theorem \ref{ThmMultipleFibr}.
	Therefore we can conclude the proof by combining Lemma \ref{LemPoint 1}, Lemma \ref{LemPoint 22} and Lemma \ref{LemPoint 222}. In particular two Seifert fibered orbifolds in the  table appearing in the statement   are orientation-preserving diffeomorphic if and only if they are in the same line by Lemma \ref{LemPoint 22} and Lemma \ref{LemPoint 222}.
\end{proof}

%-----------------------------------------------------------

\section{Geometry $\mathbb S^2\times\R$ and bad Seifert 3-orbifolds}\label{bad}\label{SxR}

The goal of this section is to prove Theorems \ref{Classification S^2xR} and  \ref{Classification bad} of the introduction, which we rewrite here using Conway notation.

\begin{reptheorem}{Classification S^2xR}
A closed orientable  Seifert 3-orbifold with geometry $\mathbb S^2\times \R$ has a unique Seifert fibration up to equivalence, with the exceptions contained in the following table:
	\begin{center}
		\begin{tabular}{ll|ll}
			&                       & for                               & such that    \\ \hline
			$(d_0d_0)$       & $(n_mn_{n-m})$        & $n\geq1$ and $1\leq m\leq n-1$ & $d=\mathrm{gcd}(n,m)$ \\ \hline
			$(\ast_0 d_0d_0)$ & $(\ast_1 n_mn_{n-m})$ & $n\geq1$ and $1\leq m\leq n-1$  & $d=\mathrm{gcd}(n,m)$
		\end{tabular}
		\\
	\end{center}
Two Seifert fibered orbifolds in the table are orientation-preserving diffeomorphic if and only if they appear in the same line, with $d=\mathrm{gcd}(n,m)$.
\end{reptheorem}

\begin{reptheorem}{Classification bad}
A closed orientable  Seifert bad 3-orbifold admits infinitely many non-equivalent Seifert fibrations. More precisely, two bad Seifert fibered orbifolds are orientation-preserving diffeomorphic if and only if they appear in the same line of the following table (for $c\neq d$):
\begin{center}
		\begin{tabular}{ll|ll}
			&                       & for                               & such that    \\ \hline
			$(c_0d_0)$       & $((c\nu)_{c\mu}(d\nu)_{d(\nu-\mu)})$        & $\nu\geq1$ and $1\leq \mu\leq \nu-1$ & $\mathrm{gcd}(\nu,\mu)=1$ \\ \hline
			$(\ast_0c_0d_0)$ & $(\ast_1 (c\nu)_{c\mu}(d\nu)_{d(\nu-\mu)})$ & $\nu\geq1$ and $1\leq \mu\leq \nu-1$  & $\mathrm{gcd}(\nu,\mu)=1$
		\end{tabular}
		\\
	\end{center}
\end{reptheorem}

\subsection{Explicit orbifold diffeomorphisms}

First, in the following lemma we prove some explicit diffeomorphisms between Seifert fibered 3-orbifolds. We will show later that these are actually the only diffeomorphisms between inequivalent fibrations of Seifert fibered 3-orbifolds which are bad or have geometry $\mathbb S^2\times\R$.

\begin{lem}\label{LemmaEquivalence}
	Let $\nu\geq2$, $\mu\in\{1,\dots,\nu-1\}$ such that  $\mathrm{gcd}(\nu,\mu)=1$ and $d,c\geq 1$. Then: 
	\begin{itemize}
		\item the orbifolds $((c\nu)_{c\mu}(d\nu)_{d(\nu-\mu)})$ and  $(c_0d_0)$ are orientation-preserving diffeomorphic;
		\item the orbifolds $(\ast_1(c\nu)_{c\mu}(d\nu)_{d(\nu-\mu)})$ and  $(\ast_0c_0d_0)$ are orientation-preserving diffeomorphic.
	\end{itemize}
	In particular, two orbifolds $((c\nu)_{c\mu}(d\nu)_{d(\nu-\mu)})$  and $((c'\nu')_{c'\mu'}(d'\nu')_{d'(\nu'-\mu')})$ are orientation-preserving diffeomorphic if and only if $\{c,d\}=\{c',d'\}$, and similarly for the second item.
\end{lem}

\begin{proof}
To prove the first item, we have to show that every  orbifold  $((c\nu)_{c\mu} (d\nu)_{d(\nu-\mu)})$ is diffeomorphic to the product orbifold $S^1\times (cd)$. 

We consider first the case when $c=d=1$. % and all the manifolds are diffeomorphic to $S^1 \times S^2$s. 
Let   $T_1$ and $T_2$ be two  solid tori. We denote by $l_i$ (resp. $m_i$) a longitude (resp. a meridian) of $T_i$. If we glue $T_1$ and $T_2$ along their boundaries by a diffeomorphism sending $m_1$ to $-m_2$, we obtain $S^1\times S^2$.

Now consider in $T_1$ (resp. $T_2$)  the Seifert fibration with the core as unique exceptional fiber with Seifert invariant  $\mu/\nu$ (resp. $(\nu-\mu)/\mu$).  The non-exceptional fibers of $T_1$ (resp. $T_2$) are homologous to $\alpha m_1 + \nu l_1$ (resp. $-\alpha m_2 + \nu l_2$), where $\alpha \nu\equiv 1 \mod \mu$ (the relation between the slope of the fiber and the local Seifert invariant is described in \cite[pag.361-364]{seifert}).  We can choose a gluing diffeomorphism sending $m_1$ to $-m_2$ such that the image of the fibration on the boundary of $T_1$ is the fibration on the  boundary of $T_2$. We have just equipped the orbifold $S^1\times S^2$ with the fibration $(\nu_\mu \nu_{\nu-\mu})$. 

The other cases can be obtained in a very similar way  considering solid tori with singular cores of indices $c$ and $d$, thus obtaining the first item.

For the second item, observe that the hyperelliptic involution of $T_i$ (reflecting  both the meridian and the longitude) respects the fibration of $T_i$. (Indeed the quotient of $T_i$ by the involution is a local model for the preimage of a neighbourhood of corner points, as explained in Section \ref{subsec:seifert fibr}.) The involutions of $T_i$ are equivariant with respect to the gluing map we have previously described, so we obtain an involution of $S^1\times (cd)$ that leaves invariant all the fibrations that we have considered above.  By construction the fibration $((c\nu)_{c\mu}(d\nu)_{d(\nu-\mu)})$    induces in the quotient a  fibration with base orbifold $\ast (c\nu)( d\nu)$, local invariants $c\nu/c\mu$ and $d(\nu-\mu)/d\nu$, and the boundary invariant is necessarily $1$ by Theorem \ref{InvariantRelation} and Remark \ref{rmk:bdy invariant determined}. Hence the fibration in the quotient is $(\ast_1(c\nu)_{c\mu}(d\nu)_{d(\nu-\mu)})$, and we obtain that  $(\ast_1(c\nu)_{c\mu}(d\nu)_{d(\nu-\mu)})$ is diffeomorphic to $(\ast_0c_0d_0)$. 	
\end{proof} 

\begin{oss}\label{rmk:not diffeomorphic in lemma}
From the proof, we see that the orbifolds in the first bullet point and those in the second bullet point of Lemma \ref{LemmaEquivalence}  are not diffeomorphic. Indeed,
for the former the underlying topological space is $S^2\times S^1$, for the latter it is $S^3$. 
\end{oss}

Based on Lemma \ref{LemmaEquivalence}, we can immediately prove Theorem \ref{Classification bad}. 

\begin{proof}[Proof of Theorem \ref{Classification bad}]
First, by \cite[Proposition 2]{Dunbar} the fibration of a  bad Seifert 3-orbifold  necessarily has a bad 2-orbifold as base orbifold. However, the viceversa does not hold: in fact if the base orbifold is bad and the Euler number is non-zero, then we obtain a spherical 3-orbifold (see again \cite[Proposition 2]{Dunbar}). As a consequence, the bad Seifert 3-orbifolds are exactly those which admit a Seifert fibration with vanishing Euler number and a bad 2-orbifold as base orbifold. Recall that bad 2-orbifolds are spheres with a single cone point or with two cone points of different indices, and discs with a single corner point or with two corner points of different indices. Hence, using Theorem \ref{InvariantRelation}, one checks that all Seifert 3-orbifolds with bad base 2-orbifold and vanishing Euler number are of the form that appears in the bullet list of Lemma \ref{LemmaEquivalence}, and therefore they admit infinitely many fibrations. Together with Remark \ref{rmk:not diffeomorphic in lemma}, this concludes the proof.
\end{proof}

We remark that Lemma \ref{LemmaEquivalence} also includes orbifolds with geometry  $\mathbb S^2\times \R$:  in fact, if the base 2-orbifold is bad then we have a bad Seifert 3-orbifold,  otherwise we get an orbifold with geometry $\mathbb S^2\times \R$. We now move on to the situation for the geometry  $\mathbb S^2\times \R$.

%\section{Seifert 3-Orbifolds with geometry $S^2\times \R$.} \label{SxR}

\subsection{All Seifert fibrations}
In this section we enumerate all  Seifert fibrations of a closed orientable Seifert fibred 3-orbifold with geometry $S^2\times \R$. By Theorem \ref{ThmSeiferTableGeometries} we know that those are the fibrations with base space a closed spherical 2-orbifold, with Euler number equal to $0$ and with \emph{compatible Seifert invariants} (as in Theorem \ref{InvariantRelation}). 

Therefore if $f:(\mathbb S^2\times\R)/\Gamma\to \Bb$ is a Seifert fibration, then $\Bb$ is diffeomorphic to an orbifold in the right column of Table \ref{table:2orbifolds}.
%in the following table.
% \begin{center}\begin{tabular}{lllll}
%		$532$ & $\Kal 532$ &           &       &  \\
%		$432$ & $\Kal 432$  &           &       &  \\
%		$332$ & $\Kal 332$ & $3\Kal 2$ &       &  \\
%		$22n$ & $\Kal 22n$ & $2\Kal n$ &       &  \\
%		$nn$  & $\Kal nn$  & $n\Kal$   & $n \times$ & 
%	\end{tabular}
%\end{center}
Using Theorem \ref{InvariantRelation} to find all compatible Seifert invariants,  all possible fibrations are computed in Table \ref{TableFibrationS2xR}. In the same table we report the underlying topological space of each fibered orbifold, which we computed using the following facts:

\begin{enumerate}\item We can remove a cone o corner point that has local invariant equal to $0/p$ without affecting the topology of the underlying space.
	\item We can substitute a cone or corner point that has invariant $m/n$  with  local invariant $(m/\mathrm{gcd}(m,n))/(n/\mathrm{gcd}(m,n))$ without changing the underlying space.
	%\item We can substitute a cone or corner point that has invariant $(n/2)/n$ for $n$ even with a cone or corner point with singularity index equal to $2$ and local invariant $1/2$ without changing the underlying space.
	\item A fibration with base space a disk without cone points has $S^3$ as underlying topological space (see \cite[Proposition 2.11]{DunbarTesi}).
	\item The fibration $(n_mn_{(n-m)})$ has $S^2\times S^1$ as underlying space (by the proof of the Lemma \ref{LemmaEquivalence} below).
	\item The fibration $(2_1\ast_1)$ has $\R P^3$ as  underlying space (see \cite[Proposition 2.11]{DunbarTesi}).
	\item The fibration $(n_0\bar \times)=(\R P^2;0/n;;0)$ has $\R P^3\# \R P^3$ as underlying space (we can reduce to the manifold case by using Fact 2 and for manifolds, namely to the case $n=1$, where the conclusion is obtained in \cite[pp. 457-459]{Scott}). 
\end{enumerate}

In particular, all closed orientable  Seifert orbifolds with geometry $\mathbb S^2\times\R$ have as underlying space  one of the following four manifolds (which are pairwise non-diffeomorphic):
\[S^3\qquad S^2\times S^1\qquad \R P^3\qquad \R P^3\# \R P^3\]

%We insert also $\pi_1(B)$ the fundamental group of the base orbifold of the fibration.

%-------- -------Nuova tabella
\begin{table}[]
	\centering
	\begin{tabular}{lllllll}
		&		$\Bb$                     & $\pi_1(\Bb)$   & $\pi_1^+(\Bb)$                        & Fibration                                                                                                  & $|\Oo|$                                                                           &                                                                                  \\ \hline
		1 &	$532$                          & $\mathbb{A}_5$  & $\mathbb{A}_5$                     & $(5_03_02_0)$                                                                                               & $S^2\times S^1$                                                                             &                                                                                  \\ \hline
		2& 	$\ast 532$                        & $\mathbb{Z}_2\times \mathbb{A}_5$                & $\mathbb{A}_5$                        & $(\ast_05_03_02_0)$                                                                                         & $S^3$                                                                                       &                                                                                  \\ \hline
		3 &	$432$  & $\mathbb{S}_4$  & $\mathbb{S}_4$                     & \begin{tabular}[c]{@{}l@{}}$(4_03_02_0)$\\ $(4_23_02_1)$\end{tabular}                                       & \begin{tabular}[c]{@{}l@{}}$S^2\times S^1$\\ $S^2\times S^1$\end{tabular}                   &                                                                                  \\ \hline
		4 & 	$\ast432$   & $\mathbb{Z}_2\times \mathbb{S}_4$                & $\mathbb{S}_4$                           & \begin{tabular}[c]{@{}l@{}}$(\ast_0 4_03_02_0)$\\ $(\ast_1 4_23_02_1)$\end{tabular}                         & \begin{tabular}[c]{@{}l@{}}$S^3$\\ $S^3$\end{tabular}                                       &                                                                                  \\ \hline
		5 &	$332$           & $\mathbb{A}_4$           & $\mathbb{A}_4$                                 & \begin{tabular}[c]{@{}l@{}}$(3_03_02_0)$\\ $(3_13_22_0)$\end{tabular}                                       & \begin{tabular}[c]{@{}l@{}}$S^2\times S^1$\\ $S^2\times S^1$\end{tabular}                   &                                                                                  \\ \hline
		6 &	$\ast332$         & $\mathbb{Z}_2\times\mathbb{A}_4$     & $\mathbb{A}_4$                                                                       & \begin{tabular}[c]{@{}l@{}}$(\ast_03_03_02_0)$\\ $(\ast_13_13_22_0)$\end{tabular}                           & \begin{tabular}[c]{@{}l@{}}$S^3$\\ $S^3$\end{tabular}                                       &                                                                                  \\ \hline
		7 &	$3\ast2$          & $\mathbb{S}_4$     & $\mathbb{A}_4$                                   & $(3_0\ast_0 2_0)$                                                                                           & $S^3$                                                                                       &                                                                                  \\ \hline
		8 &	$22n$                &        $\mathbb{D}_{2n}$      &        $\mathbb{D}_{2n}$                       & \begin{tabular}[c]{@{}l@{}}$(2_02_0n_0)$\\ $(2_12_1n_0)$\\ $(2_02_1n_{n/2})$\end{tabular}                   & \begin{tabular}[c]{@{}l@{}}$S^2\times S^1$\\ $S^2\times S^1$\\ $S^2\times S^1$\end{tabular} & \begin{tabular}[c]{@{}l@{}}$n\geq2$\\ $n\geq2$\\ $n\geq 4$ and even\end{tabular} \\ \hline
		9 &	$\ast22n$           & $\mathbb{Z}_2\times\mathbb{D}_{2n}$   &        $\mathbb{D}_{2n}$                                 & \begin{tabular}[c]{@{}l@{}}$(\ast_02_02_0n_0)$\\ $(\ast_12_12_1n_0)$\\ $(\ast_12_02_1n_{n/2})$\end{tabular} & \begin{tabular}[c]{@{}l@{}}$S^3$\\ $S^3$\\ $S^3$\end{tabular}                               & \begin{tabular}[c]{@{}l@{}}$n\geq2$\\ $n\geq2$\\ $n\geq 4$ and even\end{tabular} \\ \hline
		10 &	$2\ast n$            &  $\mathbb{D}_{4n}$         &        $\mathbb{D}_{2n}$                         & \begin{tabular}[c]{@{}l@{}}$(2_0\ast_0n_0)$\\ $(2_1\ast_1 n_0)$\end{tabular}                                & \begin{tabular}[c]{@{}l@{}}$S^3$\\ $\R P^3$\end{tabular}                                    & \begin{tabular}[c]{@{}l@{}}$n\geq2$\\ $n\geq2$\end{tabular}                      \\ \hline
		11 &	$nn$            &  $\mathbb{Z}_{n}$ &  $\mathbb{Z}_{n}$                               & \begin{tabular}[c]{@{}l@{}}$(n_0n_{0})$\\ $(n_mn_{(n-m)})$\end{tabular}                                & \begin{tabular}[c]{@{}l@{}}$S^2\times S^1$\\ $S^2\times S^1$\end{tabular}                                    & \begin{tabular}[c]{@{}l@{}}$n\geq1$\\ $n\geq2$ and $m\in\{1,\dots,n-1\}$\end{tabular}                      \\ \hline
		12 &$\ast nn$                                          &  $\mathbb{D}_{2n}$    &  $\mathbb{Z}_{n}$   &\begin{tabular}[c]{@{}l@{}}$(\ast_0n_0n_0)$\\ $(\ast_1n_mn_{(n-m)})$\end{tabular}                                                                                        &\begin{tabular}[c]{@{}l@{}}$S^3$\\ $S^3$\end{tabular}                                                                                      & \begin{tabular}[c]{@{}l@{}}$n\geq1$\\ $n\geq 2$ and $m\in\{1,\dots,n-1\}$\end{tabular}                                             \\ \hline
		13 & $n\ast$                   & $\mathbb{Z}_2\times\mathbb{Z}_{n}$   &  $\mathbb{Z}_{n}$                                                         & \begin{tabular}[c]{@{}l@{}}$(n_0\ast_0)$\\ $(n_{(n/2)}\ast_1)$\end{tabular}                                 & \begin{tabular}[c]{@{}l@{}}$S^3$\\ $\R P^3$\end{tabular}                                    & \begin{tabular}[c]{@{}l@{}}$n\geq2$\\ $n\geq2$ and even\end{tabular}             \\ \hline
		14 & $n\times$           & $\mathbb{Z}_{2n}$          &  $\mathbb{Z}_{n}$                       & $(n_0\bar\times)$                                                                                               & $\R P^3\#\R P^3$                                                                            & $n\geq1$                                                                        
	\end{tabular}
	\caption{\small All Seifert fibrations with geometry $\mathbb S^2\times\R$.\label{TableFibrationS2xR} Each line of this table correspond to a different fibration thanks to our restrictions in the parameters $n$ and $m$. Observe that the cases where the base orbifold is  $\ast$, $\times$ or $1$ (i.e. $S^2$) are also included in the table, by allowing $n=1$.}
\end{table}
%--------------Fine nuova  tabella

%\begin{remark}

%Since $d$ is the singular 

%\end{remark}

\subsection{Structure of the fundamental group}

A short exact sequence involving the fundamental group of  a Seifert fibered orbifold $f:\Oo\to\Bb$ is described in Proposition~\ref{fundamental}: the generic fiber generates a normal cyclic subgroup that we denote by $C$, and we have:
\begin{equation}\label{eq:short exact}
1 \rightarrow C \rightarrow \pi_1(\Oo) \xrightarrow[]{f_*} \pi_1(\Bb) \rightarrow 1~.
\end{equation}

 By  the  construction in the proof of Corollary~\ref{CorLines}, if $\Oo$ has geometry $\mathbb S^2\times\R$, then $\Oo\cong(\mathbb S^2\times\R)/\Gamma$ and $C$ is identified to the subgroup of $\Gamma$ consisting of vertical translations. The elements of $\pi_1(\Bb)\cong \Gamma_H$  can be seen as elements acting on $\mathbb S^2$; in this sense we will distinguish between orientation-preserving and orientation-reversing elements of $\pi_1(\Bb)$.

\begin{lem}\label{lem:+}
The elements of $\pi_1(\Oo)$ that project to orientation-preversing (resp. orientation-reversing) elements of $\pi_1(\Bb)$ in the short exact sequence \eqref{eq:short exact} commute with (resp. act dihedrally on) $C$.
\end{lem}
\begin{proof}
Let $(g,h)\in\Gamma<\Iso(\mathbb S^2)\times\Iso(\R)$. If $g$ acts on $\mathbb S^2$ preserving the orientation, then $h$ also preserves the orientation of $\R$ since $\Gamma$ is orientation-preserving. Since translations in $\Iso(\R)$ commute, this means that $(g,h)$ commute with elements of $C$. Otherwise, if $g$ reverses the orientation, then so does $h$, hence $h$ acts as a reflection on $\R$. In this case, the action of $(g,h)$ on $C$ is dihedral. 
\end{proof}

We will use the symbol $\pi_1^+(\Bb)$ to denote the subgroup of $\pi_1(\Bb)$ of orientation-preserving elements. It has index either one or two.

\begin{prop}\label{prop:maximal}
Given a closed Seifert orbifold $f:\Oo\to\Bb$ with geometry $\mathbb S^2\times\R$, 
every maximal abelian normal subgroup of $\pi_1(\Oo)$ is of the form $(f_*)^{-1}(M)$ where $M$ is 
a maximal abelian normal subgroup of $\pi_1^+(\Bb)$. 

Moreover, $|\pi_1(\Oo):(f_*)^{-1}(M)|=|\pi_1(\Bb):M|$.
\end{prop}
\begin{proof}
Let $N$ be a maximal abelian normal subgroup of $\pi_1(\Oo)$, meaning that $N$ is not contained in any other maximal abelian normal subgroup of $\pi_1(\Oo)$. First, we observe that $f_*(N)$ is contained in $\pi_1^+(\Bb)$. Indeed, if $N\cap C=\{1\}$ then the subgroup generated by $N$ and $C$ is normal, and is abelian since two normal subgroups intersecting trivially commute, which contradicts maximality of $N$. Therefore $N$ contains a non-trivial element of $C$, hence $f_*(N)$ is contained in $\pi_1^+(\Bb)$ by Lemma \ref{lem:+}. 

Second, we show that $f_*(N)$ is a maximal abelian normal subgroup of $\pi_1^+(\Bb)$. It is trivially abelian and is normal because $f_*$ is surjective. Finally, to show that $f_*(N)$ is maximal, let $M$ be an abelian normal subgroup of $\pi_1^+(\Bb)$ containing $f_*(N)$.  We claim that $(f_*)^{-1}(M)$ is an abelian normal subgroup of $\pi_1(\Oo)$, which will imply $M=f_*(N)$ by maximality of $N$. To show the claim, $(f_*)^{-1}(M)$ is clearly normal and it is abelian because, identifying $\pi_1(\Oo)$ with a subgroup $\Gamma<\Iso(\mathbb S^2)\times\Iso(\R)$, both projections of $(f_*)^{-1}(M)$ to $\Iso(\mathbb S^2)$ and to $\Iso(\R)$ are abelian. Indeed, the projection of $(f_*)^{-1}(M)$ to $\Iso(\mathbb S^2)$ is isomorphic to $M$, and the projection to $\Iso(\R)$ consists only of translations because $\Gamma$ is orientation-preserving and $M$ is contained in $\pi_1^+(\Bb)$. 

The same argument together with the maximality of $N$ implies that $N=(f_*)^{-1}(f_*(N))$, because $(f_*)^{-1}(f_*(N))$ is 
abelian and normal.

Finally, to compute the index of a maximal abelian normal subgroup $N=(f_*)^{-1}(M)$, we denote $\pi_1^+(\Oo):=(f_*)^{-1}(\pi_1^+(\Bb))$. Observe that the index of $\pi_1^+(\Oo)$ in $\pi_1(\Oo)$ is equal to the index of  $\pi_1^+(\Bb)$ in $\pi_1(\Bb)$ (and they are both equal either to one or to two), so it suffices to check that $|\pi_1^+(\Oo):(f_*)^{-1}(M)|=|\pi_1^+(\Bb):M|$. Now, we have $\pi_1^+(\Oo)/N=(\pi_1^+(\Oo)/C)/(N/C)=\pi_1^+(\Bb)/M$. 
%First, given a Seifert fibration of $\Oo$, the subgroup $C$ as in \eqref{eq:short exact} is an abelian normal subgroup, which moreover has finite index since $\pi_1(\Bb)\cong \pi_1(\Oo)/C$ is the fundamental group of a spherical 2-orbifold, hence finite. Now, it suffices to observe that, given two abelian normal subgroups $N_1$ and $N_2$, the subgroup generated by $N_1$ and $N_2$ is normal, and is also abelian since two normal
\end{proof}

\begin{oss}
The remarkable property of Proposition \ref{prop:maximal} is that it gives a bijective correspondence between the maximal abelian normal subgroup of $\pi_1(\Oo)$, which are invariants of the fundamental group of $\Oo$ (independent of the chosen fibration $f$), and the maximal abelian normal subgroup of $\pi_1^+(\Bb)$, which \emph{a priori} depend on the fibration. We will use extensively this observation in the proof of Theorem~\ref{Classification S^2xR}.
\end{oss}

%We note that if  an orbifold $\Oo$  has two fibrations over $\Bb$ and $\Bb'$, the  corresponding groups $C$ and $C'$, which correspond to the groups of vertical translation when we realise $\Oo$ as a quotient of $\mathbb S^2\times\R$, might be different. 

\subsection{Proof of Theorem~\ref{Classification S^2xR}}

We will now proceed to the proof of Theorem~\ref{Classification S^2xR}. As a preliminary step, in Table \ref{table invariants} we provide a table that computes some invariants that will be used in the proof in order to distinguish different diffeomorphism classes of orbifolds. However, these invariant will not always be sufficient to distinguish non-diffeomorphic orbifolds. Hence we need an additional lemma, which will be used several times in the  proof of Theorem~\ref{Classification S^2xR}.

%tabella nuova nuova
\begin{table}[]
	\begin{tabular}{llclcc}
		 &
		Fibrations &
		\multicolumn{1}{l}{$|\pi_1(\Oo):N|$} &
		$|\Oo|$ &
		\multicolumn{1}{l}{\begin{tabular}[c]{@{}l@{}}Number of\\ vertex points\end{tabular}} &
		\multicolumn{1}{l}{\begin{tabular}[c]{@{}l@{}}Number of \\ circles in $\Sigma_\Oo$\end{tabular}} \\ \hline
		8 &
		\begin{tabular}[c]{@{}l@{}}$(2_02_0n_0)$ \\ $(2_12_1n_0)$ \\ $(2_02_1n_{n/2})$\end{tabular} &
		$\begin{cases}1\text{ if }n=2\\2\text{ if }n\not\in\{2,4\}\end{cases}$ &
		\begin{tabular}[c]{@{}l@{}}$S^2\times S^1$\\ $S^2\times S^1$\\ $S^2\times S^1$\end{tabular} &
		&
		\begin{tabular}[c]{@{}c@{}}$3$\\ $1$\\ $2$\end{tabular} \\ \hline
		9 &
		\begin{tabular}[c]{@{}l@{}}$(\ast_02_02_0n_0)$ \\ $(\ast_12_12_1n_0)$ \\ $(\ast_12_02_1n_{n/2})$\end{tabular} &
		$\begin{cases}2\text{ if }n=2\\4\text{ if }n\not\in\{2,4\}\end{cases}$ &
		\begin{tabular}[c]{@{}l@{}}$S^3$\\ $S^3$\\ $S^3$\end{tabular} &
		\begin{tabular}[c]{@{}c@{}}$6$\\ $2$\\ $4$\end{tabular} &
		\\ \hline
		10 &
		\begin{tabular}[c]{@{}l@{}}$(2_0\ast_0n_0)$\\ $(2_1\ast_1n_0)$\end{tabular} &
		$\begin{cases}2\text{ if }n=2\\4\text{ if }n\not\in\{2,4\}\end{cases}$ &
		\begin{tabular}[c]{@{}l@{}}$S^3$\\ $\R P^3$\end{tabular} &
		\begin{tabular}[c]{@{}c@{}}$2$\\ $2$\end{tabular} &
		\\ \hline
		11 &
		$(n_0n_0)$ &
		$1$ &
		$S^2\times S^1$ &
		&
		$\begin{cases}0\text{ if }n=1\\2\text{ if }n\not=1\end{cases}$ \\ \hline
		12 &
		$(\ast_0n_0n_0)$ &
		$2$ &
		$S^3$ &
		\multicolumn{1}{l}{$\begin{cases}0\text{ if }n=1\\4\text{ if }n\not=1\end{cases}$} &
		$\begin{cases}2\text{ if }n=1\\0\text{ if }n\not=1\end{cases}$ \\ \hline
		13 &
		\begin{tabular}[c]{@{}l@{}}$(n_0\ast_0)$\\ $(n_{(n/2)}\ast_1)$\end{tabular} &
		$2$ &
		\begin{tabular}[c]{@{}l@{}}$S^3$\\ $\R P^3$\end{tabular} &
		\begin{tabular}[c]{@{}c@{}}$0$\\ $0$\end{tabular} &
		\begin{tabular}[c]{@{}c@{}}$3$\\ $2$\end{tabular} \\ \hline
		14 &
		$(n_0\bar\times)$ &
		$2$ &
		$\R P^3\#\R P^3$ &
		&
		
	\end{tabular}
	\caption{\small Computation of index, underlying topological space, number of vertex points and number of $S^1$ in the orbifolds in families 8-14 of Table \ref{TableFibrationS2xR}. Some spaces are left blank, since those invariants are not necessary in the proof of Theorem~\ref{Classification S^2xR}. The index $|\pi_1(\Oo):N|$ is computed for the maximal normal abelian subgroup $N$, when it is unique. \label{table invariants} }
\end{table}
%fine nuova tabella

\begin{lem}\label{lem:caso difficile}
For any $n\geq 2$, the orbifolds $(\ast_12_12_1n_0)$ and $(2_0\ast_0 n_0)$ are not diffeomorphic.
\end{lem}
\begin{proof}
It suffices to observe that removing a point  of $(2_0\ast_0 n_0)$ can separate one connected component of the singular locus into two, while in $(\ast_12_12_1n_0)$ no point  separates a connected component of the singular locus. {{See Figure \ref{fig:caso difficile}. }}
\end{proof}

\begin{figure}[h!]
	\centering
	\begin{subfigure}[b]{0.45\textwidth}
		\centering
		\includegraphics[width=0.6\textwidth]{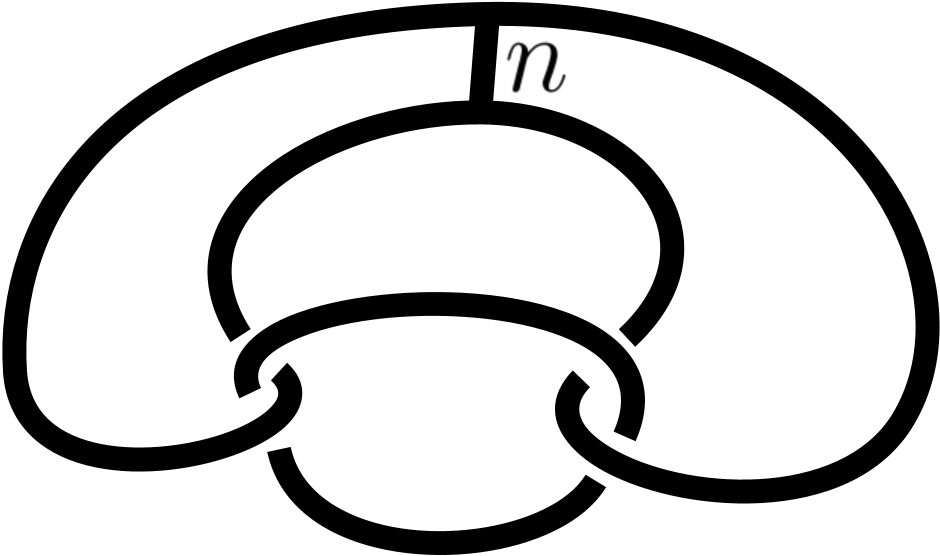}
		\caption{$(\ast_12_12_1n_0)$}
	\end{subfigure}
	\hfill
	\begin{subfigure}[b]{0.45\textwidth}
		\centering
		\includegraphics[width=0.6\textwidth]{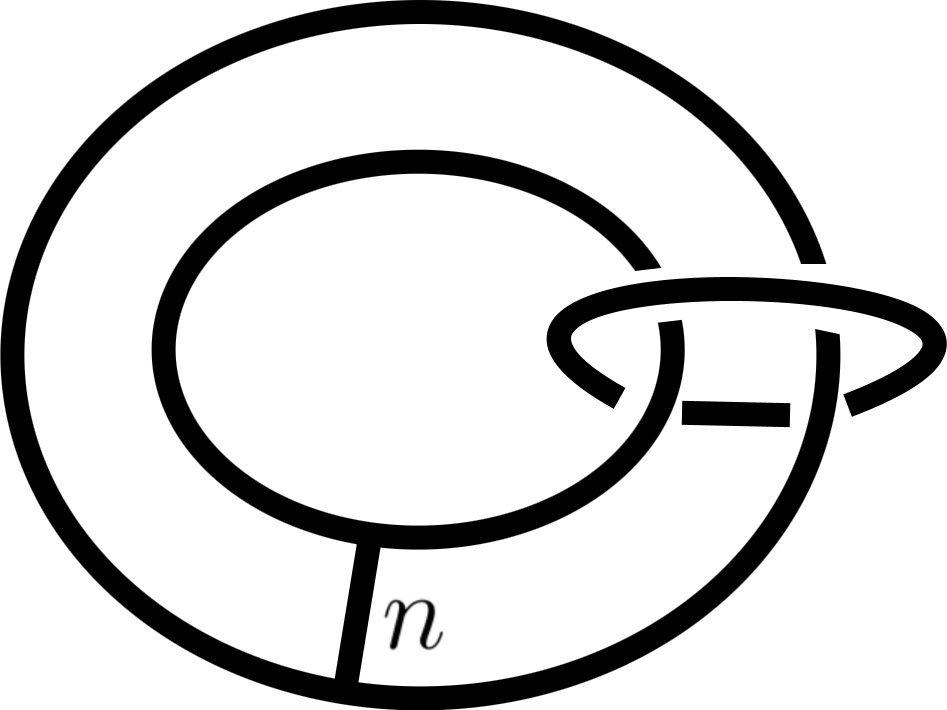}
		\caption{$(2_0\ast_0 n_0)$}
	\end{subfigure}

	\caption{\small The two orbifolds in the figure are not diffeomorphic, by the topology of their singular locus. Indeed, removing a point from $(\ast_12_12_1n_0)$ (on the left) never disconnects a connected component, while removing from $(2_0\ast_0 n_0)$ a point in the interval with label $n$  (right) disconnects its connected component.\label{fig:caso difficile}}
\end{figure}

We are now ready to provide the proof.

\begin{proof}[Proof of Theorem~\ref{Classification S^2xR}]

We will make a first distinction of the diffeomorphism type of the Seifert orbifolds in Table \ref{TableFibrationS2xR} by means of the number, and of the index, of maximal abelian normal subgroups of $\pi_1(\Oo)$. This is clearly a diffeomorphism invariant, since it is an isomorphism invariant of the fundamental group.

Therefore, we will split the analysis in several cases.

\begin{steps}
\item {\it Several maximal abelian normal subgroups.}

The only situation where $\pi_1^+(\Bb)$ has more than one maximal abelian normal subgroup occurs for Groups 8, 9 and 10 with $n=4$, for which $\pi_1^+(\Bb)\cong \pi_1(224)\cong \mathbb D_{8}$. Therefore these orbifolds are not diffeomorphic to any other orbifold in the list of Table \ref{TableFibrationS2xR}. To show that they have only one fibration, we  need to show that they are not diffeomorphic to each other. Among them, those with underlying space $S^2\times S^1$ are the three possibilities $(2_02_04_0)$, $(2_12_14_0)$ and $(2_02_14_2)$ of Group 8, whose singular loci have respectively $3$, $1$ and $2$ connected components, hence they are not diffeomorphic. Those with underlying space $S^3$ are $(\ast_02_02_04_0)$, $(\ast_12_12_14_0)$, $(\ast_12_02_14_2)$ and $(2_0\ast_0 4_0)$; they have $6, 2, 4$ and $2$ vertex points respectively. So we only have to distinguish $(\ast_12_12_14_0)$ and $(2_0\ast_0 4_0)$, which are not diffeomorphic by Lemma \ref{lem:caso difficile}.

All the other orbifolds have a unique maximal abelian normal subgroup that we will denote $N$, and we will use its index $|\pi_1(\Oo):N|$ as a distinguishing invariant.

\item {$|\pi_1(\Oo):N|>4$ or $|\pi_1(\Oo):N|=3$ .}

This is the case for Groups 1-7 in Table \ref{TableFibrationS2xR}. Their index, in the respective order, is: 
%wrong version 60, 120, 24, 48, 12, 24, 24
 60, 120, 6, 12, 3, 6, 6. Moreover, the two orbifolds in Group 3 are not diffeomorphic to those in Groups 6 and 7 because the underlying topological space is different. So we can treat each group separately, except Groups 6 and 7. The two orbifolds in Group 3 are distinguished by the number of connected components of the singular loci, and the same for the two orbifolds in Group 5. The two orbifolds in Group 4 are distinguished by the number of vertices of the singular locus. (We recall that the singular set contains a pair of vertices for each corner points with Seifert invariants $n_m$ such that  $\gcd(n,m)\neq 1$.) Finally, the number of connected components with singularity index 3 inside the singular locus for the three orbifolds in Groups 6 and 7 is respectively 2, 0, 1. This concludes that they are all pairwise non diffeomorphic.

\item{$|\pi_1(\Oo):N|=1$.}

From Proposition \ref{prop:maximal}, $|\pi_1(\Oo):N|=1$ (namely,  $\pi_1(\Oo)$ is abelian) if and only $\Bb$ is orientable and $\pi_1(\Bb)$ is abelian. Hence
this happens if and only if $\Oo$ admits a fibration either of type $(n_mn_{(n-m)})$ and $(n_0n_0)$ (Group 11 for each $n$ and $m$),  or $(2_{a}2_{a} 2_0)$ with $a\in\{0,1\}$ (Group 8 with $n=2$), see Table \ref{table invariants}.  In all cases the singular set of $\Oo$ is a (possible empty) union of simple closed curve. For $(n_mn_{(n-m)})$ and $(n_0n_0)$ the number of singular curves is either zero (if $\mathrm{gcd}(n,m)=1$) or two (if $\mathrm{gcd}(n,m)>1$), for  $(2_{0} 2_{0} 2_0)$ the number is three, for $(2_{1} 2_{1} 2_0)$ the number is one. This implies that the two orbifolds $(2_{a} 2_{a} 2_0)$ admit a unique fibration. The case of the Group 11 has been already considered in Lemma~\ref{LemmaEquivalence} with $c=d$ (see also Remark \ref{rmk:not diffeomorphic in lemma}), and it gives inequivalent fibrations on the same orbifold as claimed.

\item{$|\pi_1(\Oo):N|=2$.}

Index 2 is achieved by the orbifolds with a fibration included in the following groups: Group 8 with $n\not\in\{2,4\}$, Group 9 with $n=2$, Group 10 with $n=2$, Group 12, Group 13 and Group 14. See Table \ref{table invariants}.  We remark that, with the exception of the fibrations in the Group 12, fibered orbifolds in the same line (for different values of $n$) cannot be diffeomorphic. This can be seen by considering the underlying topological space,  the singularity indices and  the number of the vertices in the singular locus.   Consider now orbifolds that give the same underlying topological space. \begin{itemize}\item We have   Group 9 with $n=2$, Group 10 with $n=2$, Group 12, Group 13 where $S^3$ appear as underlying topological space. Counting the number of connected components and the number of vertices of the singular locus, one checks that these are pairwise non-diffeomorphic except for the case of $(\ast_12_12_12_0)$ and $(2_0\ast_0 2_0)$. These are distinguished by Lemma \ref{lem:caso difficile} with $n=2$.
	\item  The two orbifolds with underlying space $\R P^3$ are distinguished by the number of vertices.
	\item The only remaining possibility are the orbifolds within Group 8, having $S^2\times S^1$ as underlying  topological case, and also in this case the number of connected components distinguishes them.
\end{itemize}  We can conclude that in this case  two fibered orbifolds can be diffeomorphic only if they are included in the Group 12; these diffeomorphisms are treated in   Lemma~\ref{LemmaEquivalence}.

\item{$|\pi_1(\Oo):N|=4$.}

Index 4 is achieved by the orbifold with a fibration included in the following groups: Group 9 and Group 10 with $n\not\in\{2,4\}$, see Table \ref{table invariants}.  Diffeomorphisms between fibered orbifolds of the same group cannot occur by the same argument used in the previous step. Counting the number of vertices, we only have to exclude  a diffeomorphism between  $(\ast_12_12_1n_0)$ and $(2_0\ast_0n_0)$. This is impossible by Lemma \ref{lem:caso difficile}.
This concludes  the proof.\qedhere
\end{steps}
\end{proof}

Finally, we conclude by discussing Corollaries \ref{cor:conigli} and \ref{cor:char infinitely many}. The proof of Corollary \ref{cor:conigli} follows immediately from the results of this paper and of \cite{MecchiaSeppi2}, as explained in the introduction. Let us provide the proof of Corollary \ref{cor:char infinitely many}.

\begin{proof}[Proof of Corollary \ref{cor:char infinitely many}]
The results of \cite{MecchiSeppi3}  (for geometry $\mathbb S^3$) and those of Theorems \ref{Classification S^2xR} and \ref{Classification bad} (for geometry $\mathbb S^2\times\mathbb R$ and bad orbifolds) assure  that an orbifold admitting infinitely many inequivalent fibrations admits one fibration with base orbifold either a 2-sphere with at most two cone points or a disk with at most two corner points. This condition on the base orbifold is equivalent (see  \cite{DunbarTesi, Dunbar}) to the condition that $\Oo$ is either a lens space orbifold or has underlying topological $S^3$ and   singular set  a Montesinos graph with at most two rational tangles. Conversely, by direct inspection of these results, one sees that all Seifert fibered 3-orbifolds with base orbifold either a 2-sphere with at most two cone points or a disk with at most two corner points actually admit infinitely many fibrations.

\end{proof}

\begin{oss}
As a consequence of Corollary \ref{cor:char infinitely many}, it is easy to see that if a closed 3-orbifold $\Oo$ admits infinitely many inequivalent fibrations, then the orbifold fundamental group of $\Oo$ is either abelian of rank at most 2 (the group might be finite or not) or a generalized dihedral group such that the normal subgroup of index  2 has at most rank 2.

However, the converse does not hold. Indeed, the orbifold $(2_{1}2_{1} 2_0)$ (Group 8 with $n=2$, with geometry $\mathbb S^2\times\R$) admits a unique fibration up to equivalence, but its fundamental group is isomorphic to $\Z\times\Z_2$, hence is abelian of rank 2. 
\end{oss}

\bibliographystyle{alpha}
\bibliography{biblio}

\end{document}